\documentclass[12pt,a4paper,reqno]{amsart}
\usepackage[T1]{fontenc}
\usepackage[utf8]{inputenc}
\usepackage{makecell}
\usepackage[12pt]{moresize}

\usepackage{amsmath,amsthm}
\usepackage{amssymb}
\usepackage{mathrsfs}
\usepackage{graphicx}

\usepackage{
	color,
	tikz,
	xfrac,
	etoolbox,
	xparse,
	pgffor,
	longtable,
	booktabs,
	makecell,
	array,
	placeins,
	xcolor,
	microtype,
	multirow,
	tabularx
}

\IfFileExists{siunitx.sty}{\usepackage{siunitx}}{}
\IfFileExists{dsfont.sty}{\usepackage{dsfont}}{}
\IfFileExists{bropd.sty}{\usepackage{bropd}}{}
\usepackage{needspace}

\usepackage[shortlabels]{enumitem}

\setlist{
	font=\upshape,
	itemsep=0.25\baselineskip,
	leftmargin=\parindent,
	parsep=0pt,
	topsep=0.25\baselineskip
}

\usetikzlibrary{
	arrows.meta,
	calc,
	decorations.pathreplacing,
	positioning
}

\definecolor{freeviolet}{RGB}{0,0,255}
\definecolor{contactred}{RGB}{0,0,0}
\definecolor{phasegreen}{RGB}{255,255,255}

\newtheoremstyle{plain}
	{6.5pt}
	{6.5pt}
	{\itshape}
	{}
	{\bfseries}
	{.}
	{.5em}
	{}

\newtheoremstyle{definition}
	{6.5pt}
	{6.5pt}
	{\normalfont}
	{}
	{\bfseries}
	{.}
	{.5em}
	{}

\theoremstyle{plain}

\newtheorem{theorem}{Theorem}[section]
\newtheorem{lemma}[theorem]{Lemma}
\newtheorem{corollary}[theorem]{Corollary}
\newtheorem{proposition}[theorem]{Proposition}

\theoremstyle{definition}

\newtheorem{definition}[theorem]{Definition}

\newtheorem{remark}[theorem]{Remark}

\newtheorem*{question*}{Question}

\theoremstyle{plain}

\newtheorem{innercustomgeneric}{\customgenericname}
\providecommand{\customgenericname}{}

\newcommand{\newcustomtheorem}[2]{%
	\newenvironment{#1}[1]
	{%
		\renewcommand\customgenericname{#2}%
		\renewcommand\theinnercustomgeneric{##1}%
		\innercustomgeneric
	}
	{\endinnercustomgeneric}
}

\newcustomtheorem{maintheorem}{Theorem}

\numberwithin{equation}{section}

\usepackage[
	pdfa,
	colorlinks,
	allcolors=blue
]{hyperref}

\makeatletter

\let\keywords\relax
\let\msc\relax
\let\correspondingauthor\relax
\let\abstract\relax

\gdef\@keywords{}
\gdef\@msc{}
\gdef\@correspondingauthor{}

\newcommand{\keywords}[1]{%
	\gdef\@keywords{%
		\textit{Keywords:} \ignorespaces#1\unskip.%
	}%
}

\newcommand{\msc}[2][]{%
	\gdef\@msc{%
		\textit{Mathematics Subject Classification 2020:} \ignorespaces#2\unskip%
		\ifx&#1&\else\ (#1)\fi.%
	}%
}

\newcommand{\correspondingauthor}[2]{%
	\gdef\@correspondingauthor{%
		\textit{Corresponding author:} #1,
		\href{mailto:#2}{#2}.%
	}%
}

\newcommand{\ems@titlefootnote}[1]{%
	\insert\footins{%
		\reset@font
		\footnotesize
		\interlinepenalty\interfootnotelinepenalty
		\splittopskip\footnotesep
		\splitmaxdepth\dp\strutbox
		\floatingpenalty\@MM
		\hsize\columnwidth
		\@parboxrestore
		\color@begingroup
			\noindent
			\strut
			#1\par
		\color@endgroup
	}%
}

\newbox\emsabstractbox

\newenvironment{abstract}{%
	\global\setbox\emsabstractbox\vbox\bgroup
		\normalcolor
		\small
		\noindent
		\strut
		\textbf{Abstract.}\enskip
		\ignorespaces
}{%
		\unskip
		\strut
		\par
	\egroup
}

\renewcommand{\maketitle}{%

	\thispagestyle{empty}
	\enlargethispage{1pt}

	\begingroup

		\vspace*{26pt}

		\centering

		{\Large\bfseries\boldmath
			\@title
			\par
		}

		\vspace{8mm}

		{\large
			\@author
			\par
		}

		\vspace{5mm}

	\endgroup

	\ifx\@msc\@empty
	\else
		\ems@titlefootnote{\@msc}%
	\fi

	\ifx\@keywords\@empty
	\else
		\ems@titlefootnote{\@keywords}%
	\fi

	\ifx\@correspondingauthor\@empty
	\else
		\ems@titlefootnote{\@correspondingauthor}%
	\fi

	\ifvoid\emsabstractbox
	\else
		\noindent
		\unvbox\emsabstractbox
		\par
	\fi

	\vspace{26pt}
}

\newcounter{ems@authorcount}
\gdef\@author{}
\gdef\ems@authorids{}

\newcommand{\Author}[3]{%
	\stepcounter{ems@authorcount}%
	\listgadd{\ems@authorids}{#1}%
	\begingroup
		\def\givenname##1{\csgdef{auth@#1@first}{##1}}%
		\def\surname##1{\csgdef{auth@#1@last}{##1}}%
		\def\mrid##1{%
			\ifblank{##1}{}{\csgdef{auth@#1@mrid}{##1}}%
		}%
		\def\zblid##1{%
			\ifblank{##1}{}{\csgdef{auth@#1@zblid}{##1}}%
		}%
		\def\orcid##1{%
			\ifblank{##1}{}{\csgdef{auth@#1@orcid}{##1}}%
		}%
		#2%
	\endgroup
	\ifnum\value{ems@authorcount}=1
		\gdef\@author{%
			\csuse{auth@#1@first}\space\csuse{auth@#1@last}%
		}%
	\else
		\g@addto@macro\@author{%
			,\space\csuse{auth@#1@first}\space\csuse{auth@#1@last}%
		}%
	\fi
}

\def\ems@setmeta#1#2#3{%
	\ifblank{#3}{}{%
		\csgdef{affil@#1@#2}{#3}%
		\csgdef{affil@#1@present}{1}%
	}%
}

\def\ems@makeaffilmacro#1{%
	\expandafter\def\csname #1\endcsname##1{%
		\@ifnextchar\bgroup
			{\csuse{ems@#1@two}{##1}}%
			{\csuse{ems@#1@one}{##1}}%
	}%
	\expandafter\def\csname ems@#1@two\endcsname##1##2{%
		\ems@setmeta{\ems@currentaffil@id @##1}{#1}{##2}%
	}%
	\expandafter\def\csname ems@#1@one\endcsname##1{%
		\ems@setmeta{\ems@currentaffil@id @1}{#1}{##1}%
	}%
}

\newcommand{\Affil}[2]{%
	\begingroup
		\def\ems@currentaffil@id{#1}%
		\ems@makeaffilmacro{department}%
		\ems@makeaffilmacro{organisation}%
		\ems@makeaffilmacro{address}%
		\ems@makeaffilmacro{post}%
		\ems@makeaffilmacro{city}%
		\ems@makeaffilmacro{country}%
		\ems@makeaffilmacro{affemail}%
		#2%
	\endgroup
}

\def\ems@emitfield#1{%
	\ifcsdef{affil@\ems@affkey @#1}{%
		\ems@sep
		\csuse{affil@\ems@affkey @#1}%
		\def\ems@sep{,\space}%
	}{}%
}

\def\ems@printid#1#2#3{%
	\ifcsdef{auth@\ems@aid @#1}{%
		\edef\ems@tempid{\csuse{auth@\ems@aid @#1}}%
		\ems@idsep
		#2~\href{#3\ems@tempid}{\ems@tempid}%
		\def\ems@idsep{\quad}%
	}{}%
}
\def\ems@checkid#1{%
	\ifcsdef{auth@\ems@aid @#1}{\def\ems@hasid{1}}{}%
}

\def\ems@printauthor#1{%
	\begingroup
		\def\ems@aid{#1}%
		\Needspace{5\baselineskip}%
		\noindent
		{\bfseries\csuse{auth@#1@first}\space\csuse{auth@#1@last}}%
		\par\nobreak
		\foreach \j in {1,2,3,4}{%
			\edef\ems@affkey{\ems@aid @\j}%
			\ifcsdef{affil@\ems@affkey @present}{%
				\def\ems@sep{}%
				\ems@emitfield{department}%
				\ems@emitfield{organisation}%
				\ifx\ems@sep\@empty\else\par\nobreak\fi
				\def\ems@sep{}%
				\ems@emitfield{address}%
				\ems@emitfield{post}%
				\ems@emitfield{city}%
				\ems@emitfield{country}%
				\ifx\ems@sep\@empty\else\par\fi
				\ifcsdef{affil@\ems@affkey @affemail}{%
					\edef\ems@tempemail{%
						\csuse{affil@\ems@affkey @affemail}%
					}%
					\href{mailto:\ems@tempemail}{\ems@tempemail}\par
				}{}%
			}{}%
		}%
		\def\ems@hasid{0}%
		\forcsvlist{\ems@checkid}{zblid,mrid,orcid}%
		\ifnum\ems@hasid=1
			\def\ems@idsep{}%
			\noindent Author IDs:\space
			\ems@printid{zblid}{zbMATH}{https://zbmath.org/authors/}%
			\ems@printid{mrid}{MR}{https://mathscinet.ams.org/mathscinet/MRAuthorID/}%
			\ems@printid{orcid}{ORCID}{https://orcid.org/}%
			\par
		\fi
		\addvspace{13pt}%
	\endgroup
}

\AtEndDocument{%
	\ifnum\value{ems@authorcount}>0
		\par\addvspace{19.5pt}%
		\begingroup
			\small\raggedright
			\forlistloop{\ems@printauthor}{\ems@authorids}%
		\endgroup
	\fi
}

\makeatother

\newcommand{\dx}[1]{%
	\mathop{}\!\mathrm{d}#1
}

\newcommand{\dH}[1]{%
	\mathop{}\!\mathrm{d}\mathcal{H}^{#1}
}

\def\XXint#1#2#3{%

	{%
		\setbox0=\hbox{%
			$#1{#2#3}{\int}$%
		}%

		\vcenter{%
			\hbox{$#2#3$}%
		}%

		\kern-.53\wd0
	}%
}

\newcommand{\R}{\mathbb{R}}

\newcommand{\cB}{\ensuremath{\mathcal B}}

\newcommand{\cD}{\ensuremath{\mathcal D}}

\newcommand{\cH}{\ensuremath{\mathcal H}}

\newcommand{\cL}{\ensuremath{\mathcal L}}

\newcommand{\cN}{\ensuremath{\mathcal N}}

\newcommand{\cR}{\ensuremath{\mathcal R}}

\let\ge\geqslant
\let\le\leqslant

\newcommand{\mres}{%
	\mathbin{%
		\vrule
			height 1.6ex
			depth 0pt
			width 0.13ex
		\vrule
			height 0.13ex
			depth 0pt
			width 1.3ex
	}%
}

\title{%
	On the Stokes conjecture for compressible gravity waves of extreme form
}

\Author{1}{%

	\givenname{Lili}
	\surname{Du}

	\mrid{}
	\zblid{}
	\orcid{}

}{L.~Du}

\Author{2}{%

	\givenname{Chunlei}
	\surname{Yang}

	\mrid{}
	\zblid{}
	\orcid{}

}{C.~Yang}

\Affil{1}{%

	\department{Department of Mathematics}

	\organisation{Sichuan University}

	\address{}

	\post{610065}

	\city{Chengdu}

	\country{China}

	\affemail{dulili@scu.edu.cn}
}

\Affil{2}{%

	\department{School of Mathematical Sciences}

	\organisation{Shenzhen University}

	\address{}

	\post{518060}

	\city{Shenzhen}

	\country{China}

	\affemail{yangchunlei@szu.edu.cn}
}

\msc{%
	35R35, 35B65, 76B15
}


\correspondingauthor
	{C.~Yang}
	{yangchunlei@szu.edu.cn}

\begin{document}

\begin{abstract}
	In 1880, G.~G. Stokes [\emph{Mathematical and Physical Papers}, Vol.~I] examined an extreme two-dimensional irrotational periodic traveling water wave under the influence of gravity and conjectured a symmetric corner of \(120^\circ\) at the crest. The conjecture was first proved rigorously by Amick, Fraenkel, and Toland [\emph{Acta Math.} 148 (1982)] and, independently, by Plotnikov [\emph{Dinamika Sploshn. Sredy} 57 (1982)] via the Nekrasov integral equation. In the early 2010s, V\u{a}rv\u{a}ruc\u{a} and Weiss [\emph{Acta Math.} 206 (2011)] revisited the conjecture by introducing a novel geometric method, which provided an affirmative answer to the conjecture without requiring the structural assumptions used in earlier works, such as the isolation of the stagnation points and the symmetry and monotonicity of the free surface.

	The main purpose of this paper is to establish the validity of the Stokes conjecture in the context of compressible gravity water waves. For two-dimensional subsonic variational solutions satisfying a local growth bound near the stagnation level, we prove that the rescaled stream function converges to the Stokes profile at every non-degenerate stagnation point. At any stagnation point where the free boundary is locally an embedded continuous curve, its two branches have secant slopes \(\pm1/\sqrt{3}\) and enclose an angle of \(120^\circ\). This gives the first rigorous confirmation of the compressible counterpart of the classical Stokes conjecture of 1880.

	The central aspect of our approach is the discovery of a new monotonicity formula and a modified version of it for quasilinear free boundary problems of Bernoulli type. Another key ingredient is a new nonlinear frequency formula, which we combine with the concentration-cancellation theory of Evans and M\"{u}ller [\emph{J. Amer. Math. Soc.} 7 (1994)]. These newly developed monotonicity formulas are valid for nonlinear problems and enable us to carry out a blow-up analysis at each stagnation point and to provide a complete classification of the blow-up limits near each stagnation point.
\end{abstract}

\maketitle

\tableofcontents

\section{Introduction and setting}

In 1880, G.~G. Stokes studied two-dimensional gravity water waves of permanent form, propagating at constant velocity on the free surface of an inviscid incompressible fluid. He was, however, uncertain about the wave profile near the crest; in \cite[pp.~225--228]{MR2858161}, he formulated the following conjecture:

\begin{quotation}
	\dots After careful consideration I feel satisfied that there is no such earlier limit, but that we may actually approach as near as we please to the form in which the curvature at the vertex becomes infinite, and the vertex becomes a multiple point where the two branches with which alone we are concerned enclose an angle of \(120^\circ\)\dots
\end{quotation}

The assertion that the interface between air and water forms a symmetric \(120^\circ\) corner at the crest has become known as the \emph{Stokes conjecture}. For incompressible irrotational waves, the conjecture was first proved in \cite{MR666110} and \cite{MR752600,MR1883094} under structural assumptions such as isolation of stagnation points and symmetry and monotonicity of the free surface. Later, in a remarkable work \cite{MR2810856}, the conjecture was proved again without these assumptions, via a geometric approach. The generalized Stokes conjecture has also been studied for two-dimensional waves with vorticity \cite{MR2995099} and for steady axisymmetric flows, both with and without vorticity \cite{MR4595616,MR3225630}. The purpose of the present paper is to answer the following question in the affirmative:

\begin{question*}
	\textit{Does a corner of \(120^\circ\) form at the crest of a \textbf{compressible} irrotational gravity water wave?}
\end{question*}

We first formulate the problem and state the main results. Related results in the literature are discussed later.

\subsection{The Euler equations for a compressible water wave}\label{sec.euler}

We begin by recalling the equations governing compressible water waves in \(\R^2\); see \cite[Sect.~1]{MR3887218,MR4439376} for the \(n\)-dimensional case. Let \(\cD(t):=\{(x_1,x_2)\in\R^2:x_2<S(t,x_1)\}\) denote the domain occupied by the fluid at time \(t\), where \(S\) denotes the interface between the fluid and the air. The boundary \(\partial\cD(t):=\{(x_1,x_2):x_2=S(t,x_1)\}\) of \(\cD(t)\) moves with the fluid velocity. The motion of the fluid is described by the Euler equations
\begin{equation}\label{eq.euler}
	\left\{
		\begin{aligned}
			&\partial_tv+(v\cdot\nabla)v=-\frac{1}{\rho}\nabla p-g\mathbf e_2\quad\text{ in }\cD,\\
			&\partial_t\rho+(v\cdot\nabla)\rho+\rho\operatorname{div}v=0\quad\text{ in }\cD,
		\end{aligned}
	\right.
\end{equation}
where \(v\) is the fluid velocity, \(p\) the pressure, \(\rho\) the density, \(g>0\) the gravitational constant, and \(\mathbf e_2=(0,1)^{\mathsf T}\). In~\eqref{eq.euler},
\[
	\cD:=\bigcup_{t\in[0,T]}\{t\}\times\cD(t).
\]
We assume temporarily that \(S(t,x_1)\) is \(C^1\) in \(t\) and \(x_1\), that \(v\) and \(\rho\) are \(C^1\) in the interior of the space-time fluid region \(\cD\), and that \(\rho>0\) on \(\cD(t)\). Moreover, \(v\) and \(\rho\) are assumed to extend continuously to the boundary \(\partial\cD(t)\) locally from the fluid side. At the free surface, the kinematic condition reads:
\[
	\partial_t S(t,x_1)+v_1(t,x_1,S(t,x_1))\,\partial_{x_1}S(t,x_1)=v_2(t,x_1,S(t,x_1)).
\]
The equation of state is
\begin{equation}\label{eq.state}
	p\in C^1((0,\infty)),\qquad p=p(\rho),\qquad  p'(\rho)>0\quad\text{ for }\rho>0,
\end{equation}
and the sound speed \(c(\rho)\) of the fluid and the Mach number \(M\) are given by
\[
	c(\rho) = \sqrt{p'(\rho)}\qquad\text{ and }\qquad M=\frac{q}{c(\rho)},
\]
respectively. Here \(q=\sqrt{v_1^2+v_2^2}\) denotes the flow speed. In addition, the flow is assumed to be irrotational,
\[
	\operatorname{curl}\,v=0.
\]
We restrict our attention to traveling-wave solutions of~\eqref{eq.euler}: there exist \(D\subset\R^2\), \(c_0\in\R\), and functions \(\tilde v\), \(\tilde\rho\) such that
\[
	\begin{aligned}
		\cD(t)&=D+c_{0}t(1,0)\qquad\text{ for all }t\in[0,T],\\
		v(t,x_1,x_2)&=\tilde v(x_1-c_0t,x_2)+c_0(1,0),
	\end{aligned}
\]
and
\[
	\rho(t,x_1,x_2)=\tilde\rho(x_1-c_0t,x_2)
\]
for all \(t\in[0,T]\) and \((x_1,x_2)\in\cD(t)\). Dropping the tildes, we let \(v\) denote the velocity relative to the traveling frame. In what follows, \(q=|v|\), and the Mach number is computed with respect to this relative velocity. The time-dependent system \eqref{eq.euler} reduces to the following stationary system
\begin{align}\label{eq.euler-stationary}
	\left\{
		\begin{aligned}
			&(v\cdot\nabla)v=-\frac{1}{\rho}\nabla p-g\mathbf{e}_{2}\qquad&&\text{ in }D,\\
			&\rho \operatorname{div}v+(v\cdot\nabla)\rho=0\qquad&&\text{ in }D,\\
			&\operatorname{curl}v=0\qquad&&\text{ in }D.
		\end{aligned}
	\right.
\end{align}
Let \(\partial_aD\) denote the free interface in contact with the air. For a traveling wave, we have \(S(t,x_1)=S(0,x_1-c_0t)\). Substitution into the kinematic condition yields
\[
	v_2(x_1,S(0,x_1))=v_1(x_1,S(0,x_1))\,\partial_{x_1}S(0,x_1),
\]
where \(v\) is the relative velocity. Equivalently,
\[
	v\cdot\nu=0\qquad\text{on }\partial_aD,
\]
where \(\nu\) is the outward unit normal to the fluid domain. Neglecting surface tension, the dynamic boundary condition is
\[
	p(\rho)=p_{\rm atm}\quad\text{ on }\partial_aD,
\]
where \(p_{\rm atm}\) is the constant atmospheric pressure. We assume that \(p_{\rm atm}\in p((0,\infty))\) and let \(\bar\rho_0\) be the unique positive number such that \(p(\bar\rho_0)=p_{\rm atm}\). By~\eqref{eq.state}, this is equivalent to
\begin{align}\label{eq.boundary-density}
	\rho = \bar\rho_0\qquad\text{ on }\partial_aD.
\end{align}
We further assume that \(p\in C^2\) in a neighborhood of \(\bar\rho_0\). The enthalpy is defined by
\[
	h(\rho):=\int_{\bar\rho_0}^{\rho}\frac{p'(r)}{r}\,\dx{r}.
\]
Then irrotationality gives \((v\cdot\nabla)v=\nabla(|v|^2/2)\), and the equation of state yields \(\nabla p/\rho=\nabla h(\rho)\). The momentum equation in~\eqref{eq.euler-stationary} therefore becomes
\[
	\nabla\left(\frac{|v|^2}{2}+h(\rho)+gx_2\right)=0\qquad\text{in }D.
\]
Hence there exists a constant \(\mathcal B\), the Bernoulli constant, such that
\begin{align}\label{eq.bernoulli-law}
	\frac{q^2}{2}+h(\rho)+gx_2=\cB\qquad\text{in }D.
\end{align}
Since \(h'(\rho)=p'(\rho)/\rho>0\) and \(h(\bar\rho_0)=0\), the dynamic condition~\eqref{eq.boundary-density} gives
\[
	\frac{q^2}{2}+gx_2=\cB\qquad\text{on }\partial_aD.
\]
After a vertical translation, we may assume that \(\cB=\tfrac{p'(\bar\rho_0)}{2}\), and hence
\begin{equation}\label{eq.bernoulli-free-surface}
	q^2=p'(\bar\rho_0)-2gx_2\qquad\text{on }\partial_aD.
\end{equation}

\begin{remark}
	The assumptions on the pressure \(p\) in \eqref{eq.state}, together with \(p\in C^2\) in a neighborhood of \(\bar\rho_0\), correspond to a barotropic compressible liquid with positive free-surface density \(\bar\rho_0\), as in the free boundary problem studied in \cite{Lin05}. In particular, the pressure law
	\[
		p(\rho)=p_{\rm atm}+A\left[\left(\frac{\rho}{\bar\rho_0}\right)^\gamma-1\right],
		\qquad A>0,\quad\gamma\ge1,
	\]
	satisfies these hypotheses (see also \cite[Section~1]{MR4439376} for a normalized version).
\end{remark}

\subsection{A quasilinear free boundary problem}

The second equation in~\eqref{eq.euler-stationary} is equivalent to \(\operatorname{div}(\rho v)=0\) in \(D\). Since \(D\) is simply connected, there exists a stream function \(\psi\in C^2(D)\) with \(\nabla\psi=(-\rho v_2,\rho v_1)\) in \(D\). Since \(\psi\) is constant on \(\partial_aD\), we normalize \(\psi\) so that \(\psi=0\) there, further restrict to flows with \(\psi>0\) in \(D\), and extend \(\psi\) by zero across the air region. We then identify the fluid domain with the set \(\{\psi>0\}\). Therefore, we deduce from equations \eqref{eq.bernoulli-law} and~\eqref{eq.bernoulli-free-surface} that
\begin{align}\label{bl1}
	\left\{
		\begin{aligned}
			&\frac{|\nabla\psi|^2}{2\rho^2}+h(\rho)+gx_2=\frac{p'(\bar\rho_0)}{2}\qquad&&\text{ in }\{\psi>0\},\\
			&|\nabla\psi|^2=\bar\rho_{0}^2\left( p'(\bar\rho_{0})-2gx_2 \right)\qquad&&\text{ on }\partial\{\psi>0\}.
		\end{aligned}
	\right.
\end{align}
Let us now write \(t:=|\nabla\psi|^2\), \(z=x_2\), and set
\[
	\mathscr F(\rho;z):=2\rho^2\Bigl( \frac{p'(\bar\rho_0)}2-gz-h(\rho) \Bigr).
\]
Then if \(t=\mathscr F(\rho;z)\), we set \(q^2:=t/\rho^2\) and \(c^2:=p'(\rho)\). A simple calculation gives \(\partial_\rho\mathscr F=2\rho(q^2-c^2)\). Therefore the implicit-function theorem yields a local density branch \(\rho=\rho(t;z)\) at any subsonic state \(q^2<c^2\), with
\[
	\partial_t\rho=-\frac{1}{2\rho(c^2-q^2)}<0,\qquad\text{ and}\qquad\partial_z\rho=-\frac{g\rho}{c^2-q^2}<0.
\]
Combining this branch with \(\operatorname{curl}v=0\) and using the definition of \(\psi\), we obtain
\begin{align*}
	\operatorname{div}\left(\frac{\nabla\psi}{\rho(|\nabla\psi|^2;x_2)}\right)=0\qquad\text{ in }\{\psi>0\}.
\end{align*}
Together with the second equation in~\eqref{bl1} and the identity \(x_2^{\rm st}=\frac{p'(\bar{\rho}_{0})}{2g}\), this leads to the following quasilinear free boundary problem
\begin{equation}\label{eq.main-physical-problem}
	\left\{
		\begin{aligned}
			&\operatorname{div}\left(\frac{\nabla\psi}{\rho(|\nabla\psi|^2;x_2)}\right)=0
			&&\text{in }\Omega\cap\{\psi>0\},\\[4pt]
			&|\nabla\psi|^2=2\bar\rho_0^2g(x_2^{\mathrm{st}}-x_2)
			&&\text{on the regular part of the free boundary}.
		\end{aligned}
	\right.
\end{equation}
Here \(x_2^{\rm st}=\frac{p'(\bar{\rho}_{0})}{2g}\) and \(\Omega\subset\R^2\) is a bounded domain with \(\Omega\cap\{x_2=x_2^{\rm st}\}\neq\varnothing\).

\subsection{The subsonic state near a stagnation point}

A stagnation point is a point at which the relative velocity vanishes. In this paper, we consider only stagnation points on the free boundary. At such a point \(x^{\rm st}=(x_1^{\rm st},x_2^{\rm st})\), the free boundary conditions imply \(q(x^{\rm st})=0\) and \(\rho(x^{\rm st})=\bar\rho_0\), so the Mach number is zero. The following lemma shows that the flow is uniformly subsonic on the fluid side near \(x^{\rm st}\).

\begin{lemma}
	Let \(x^{\rm st}=(x_1^{\rm st},x_2^{\rm st})\) be a stagnation point on the free boundary of a flow satisfying the Bernoulli law \eqref{eq.bernoulli-law}, with \(x_2^{\rm st}=\tfrac{p'(\bar{\rho}_{0})}{2g}\). Assume that
	\[
		\lim_{\substack{x\to x^{\rm st},\\ x\in\{\psi>0\}}}\rho(x)=\bar\rho_0.
	\]
	Then there exists a neighborhood \(\Omega\) of \(x^{\rm st}\) such that
	\[
		M^2(x)\le\frac12\qquad\text{in }\Omega\cap\{\psi>0\}.
	\]
	Moreover, \(|\nabla\psi(x)|\to0\) as \(\{\psi>0\}\ni x\to x^{\rm st}\).
\end{lemma}

\begin{proof}
	By~\eqref{eq.bernoulli-law} and \(h(\bar\rho_0)=0\),
	\[
		q^2(x)=2\bigl(\cB-gx_2-h(\rho(x))\bigr)\to0
		\qquad\text{as }\{\psi>0\}\ni x\to x^{\rm st},
	\]
	since \(x_2\to x_2^{\rm st}\), \(\rho(x)\to\bar\rho_0\), and \(\cB=p'(\bar\rho_0)/2=gx_2^{\rm st}\). Since \(p'(\rho(x))\to p'(\bar\rho_0)>0\), we obtain \(M^2=q^2/p'(\rho)\to0\). Hence, after shrinking the neighborhood if necessary, \(M^2\le1/2\) there. Finally, the identity \(|\nabla\psi|^2=\rho^2q^2\) yields \(|\nabla\psi|\to0\).
\end{proof}

\section{Main results and overview of the approach}\label{sec.main-results}

\subsection{Main results}
Our main purpose in this paper is to investigate the singular shape of the free surface near the stagnation points of a compressible gravity water wave. To this end, we consider subsonic variational solutions and subsonic weak solutions of problem \eqref{eq.main-physical-problem} (cf. Definitions~\ref{def.subsonic-variation-solution} and~\ref{Definition: subsonic weak solution}). For subsonic variational solutions, we assume the growth condition
\begin{equation}\label{gro1}
	|\nabla\psi|^2\le C_0(x_2^{\mathrm{st}}-x_2)^+\qquad\text{a.e. near }\{x_2=x_2^{\rm st}\}\cap\Omega,
\end{equation}
where \(C_0\) may be chosen uniformly on compact subsets.

Let \(\psi\) be a subsonic variational solution of \eqref{eq.main-physical-problem}. Then \(\psi\in C^2(\Omega\cap\{\psi>0\})\). By the interior continuity of \(\nabla\psi\), the a.e. growth bound holds pointwise in the positive phase. Since the zero extension of \(\psi\) across the air region lies in \(W_{\rm loc}^{1,2}(\Omega)\), we have \(\nabla\psi=0\) a.e. on \(\Omega\cap\{\psi=0\}\). As \(\Omega\cap\{\psi>0\}\subset\{x_2<x_2^{\mathrm{st}}\}\), on any fixed open subset \(U\subset\Omega\) near \(\{x_2=x_2^{\mathrm{st}}\}\), the following two conditions are equivalent, with the same constant:
\[
	\begin{aligned}
		|\nabla\psi|^2&\le C_0(x_2^{\mathrm{st}}-x_2)^+ &&\text{a.e. in }U,\\
		|\nabla\psi|^2&\le C_0(x_2^{\mathrm{st}}-x_2)^+ &&\text{in }
		U\cap\{\psi>0\}.
	\end{aligned}
\]

We denote the set of stagnation points by
\begin{equation}\label{set.S}
	S=\Omega\cap\partial\{\psi>0\}\cap\{x_2=x_2^{\mathrm{st}}\}.
\end{equation}
We also introduce the full weighted density subset \(\Sigma\subset S\), consisting of the points whose normalized Weiss density equals \(2/(3\bar\rho_0)\):
\begin{equation}\label{set.Sigma}
	\Sigma=\left\lbrace x^{\mathrm{st}}=(x_1^{\mathrm{st}},x_2^{\mathrm{st}})\in S:\Theta(x^{\mathrm{st}})=\frac{2}{3\bar\rho_0}\right\rbrace,
\end{equation}
where
\begin{equation}\label{eq.normalized-weiss-density}
	\Theta(x^{\mathrm{st}}):=\frac{1}{\bar\rho_0}\lim_{r\to 0^+}r^{-3}\int_{B_r(x^{\mathrm{st}})}(x_2^{\mathrm{st}}-x_2)^+\chi_{\{\psi>0\}}\,\dx{x}
\end{equation}
is the normalized Weiss density. Its limit exists by Lemmas~\ref{lem.blow-up limits} and~\ref{Lemma: densities at the stagnation points}.

To describe the Stokes-corner asymptotics, we use polar coordinates \(x=(x_1,x_2)=(\varrho\cos\theta,\varrho\sin\theta)\), \(-\pi\le\theta\le\pi\) in two dimensions, and set
\begin{equation}\label{eq.main-Stokes-profile}
	\psi_{\mathrm{st}}(x_1,x_2)=
	\begin{cases}
		\displaystyle\frac{\sqrt2}{3}\bar\rho_0\sqrt{2g}\,\varrho^{3/2}
		\cos\!\left(\frac32\left(\theta+\frac\pi2\right)\right),
		&-\frac{5\pi}{6}<\theta<-\frac{\pi}{6},\\[3pt]
		0,&\text{otherwise}.
	\end{cases}
\end{equation}
Note that \(\psi_{\mathrm{st}}\) vanishes at the origin and that \(\{\psi_{\mathrm{st}}>0\}=\{(x_1,x_2):-\frac{5\pi}{6}<\theta<-\frac{\pi}{6}\}\).

\begin{maintheorem}{A}\label{thm.main-A}
	Let \(\psi\) be a subsonic weak solution of \eqref{eq.main-physical-problem} satisfying \eqref{gro1}. Assume in addition that every point of \(S\) has an open neighborhood in \(\Omega\) meeting only finitely many connected components of \(\Omega\cap\{\psi=0\}\). Then \(S\) is locally
	finite in \(\Omega\), and at every \(x^{\mathrm{st}}=(x_1^{\mathrm{st}},x_2^{\mathrm{st}})\in S\),
	\[
		\frac{\psi\bigl((x_1^{\mathrm{st}},x_2^{\mathrm{st}})+r(x_1,x_2)\bigr)}{r^{3/2}}\longrightarrow\psi_{\mathrm{st}}(x_1,x_2)
		\qquad(r\to0^+)
	\]
	locally uniformly and strongly in \(W^{1,2}_{\mathrm{loc}}(\R^2)\).
\end{maintheorem}

We make some comments on the above theorem.

\begin{remark}
	The limit function \(\psi_{\mathrm{st}}\) in \eqref{eq.main-Stokes-profile} is shown in Figure~\ref{Fig: psi0}.
	\begin{figure}[!ht]
		\centering
		\tikzset{every picture/.style={line width=0.75pt}} 

		\begin{tikzpicture}[x=1pt,y=1pt,yscale=-1,xscale=1]

		\draw [color={rgb, 255:red, 128; green, 128; blue, 128 }  ,draw opacity=1 ] [dash pattern={on 0.84pt off 2.51pt}]  (239.67,169.5) -- (470.33,169.5) ;
		\draw [shift={(472.33,169.5)}, rotate = 180] [color={rgb, 255:red, 128; green, 128; blue, 128 }  ,draw opacity=1 ][line width=0.75]    (10.93,-3.29) .. controls (6.95,-1.4) and (3.31,-0.3) .. (0,0) .. controls (3.31,0.3) and (6.95,1.4) .. (10.93,3.29)   ;
		\draw [dash pattern={on 0.84pt off 2.51pt}] (350.35,170.05) -- (485.35,{170.05+135/sqrt(3)}) ;
		\draw [dash pattern={on 0.84pt off 2.51pt}] (350.35,170.05) -- (215.35,{170.05+135/sqrt(3)}) ;
		\draw    (350.35,170.05) -- (455.33,{170.05+104.98/sqrt(3)}) ;
		\draw    (350.35,170.05) -- (258.67,{170.05+91.68/sqrt(3)}) ;
		\draw    (350.35,170.05) -- (350.35,170.05) ;
		\draw [shift={(350.35,170.05)}, rotate = 90] [color={rgb, 255:red, 0; green, 0; blue, 0 }  ][fill={rgb, 255:red, 0; green, 0; blue, 0 }  ][line width=0.75]      (0, 0) circle [x radius= 1.34, y radius= 1.34]   ;

		\draw (330.7,144.4) node [anchor=north west][inner sep=0.75pt]  [font=\footnotesize]  {\(\psi_{\mathrm{st}} =0\)};
		\draw (332.7,195.3) node [anchor=north west][inner sep=0.75pt]  [font=\footnotesize]  {\(\psi_{\mathrm{st}}  >0\)};
		\draw (236.3,220) node [anchor=north west][inner sep=0.75pt]  [font=\footnotesize]  {\(-\frac{5}{6} \pi \)};
		\draw (458.77,225) node [anchor=north west][inner sep=0.75pt]  [font=\footnotesize]  {\(-\frac{1}{6} \pi \)};
		\draw (412.9,177.2) node [anchor=north west][inner sep=0.75pt]  [font=\footnotesize]  {\(\psi_{\mathrm{st}} =0\)};
		\draw (253.9,174.2) node [anchor=north west][inner sep=0.75pt]  [font=\footnotesize]  {\(\psi_{\mathrm{st}} =0\)};
		\draw (472.6,166.33) node [anchor=north west][inner sep=0.75pt]    {\(x_{1}\)};

		\end{tikzpicture}
		\caption{The limit function \(\psi_{\mathrm{st}}(x_1,x_2)\)}
		\label{Fig: psi0}
	\end{figure}
\end{remark}

\begin{remark}
	In Section~\ref{sec:notations-preliminaries} we consider an \(n\)-dimensional scalar generalization of problem \eqref{eq.main-physical-problem} and introduce the normalized variables. Definitions~\ref{def.subsonic-variation-solution} and~\ref{Definition: subsonic weak solution} then specify the subsonic variational and weak solution classes, including the additional regularity requirements for weak solutions.
\end{remark}

\begin{remark}
	Theorem~\ref{thm.main-A} is a compressible counterpart of \cite[Theorem~A and Theorem~11.2]{MR2810856}. In both results, the additional topological hypothesis is local finiteness of the connected components of the zero phase \(\Omega\cap\{\psi=0\}\). However, if the fluid is compressible, then the Laplace equation reduces to a quasilinear equation on a subsonic density branch, and let us remark that this is one of the main difficulties of our analysis. Moreover, we remark that the condition \eqref{gro1} allows an arbitrary finite local constant \(C_0\), whereas Theorem A in \cite{MR2810856} requires the sharp Bernstein bound, which in the present physical normalization reads
	\[
		|\nabla\psi|^2\le 2g\bar\rho_0^2(x_2^{\mathrm{st}}-x_2)^+.
	\]
\end{remark}

\begin{remark}[A physical interpretation of the growth condition \eqref{gro1}]
	Recall that \(q=|v|\) is the speed of the stationary flow, let us set \(s:=x_2^{\mathrm{st}}-x_2\), then inside the fluid, we have \(s>0\). It follows from 
	\eqref{eq.bernoulli-free-surface} that \(q^2=2gs\) on the free boundary. Consequently, we have \(|\nabla\psi|^2=\rho^2q^2=2g\bar\rho_0^2s\), since \(\rho=\bar\rho_0\) on the free boundary.

	It follows from the Bernoulli's law and the identity \(\nabla\psi=(-\rho v_2,\rho v_1)\) that
	\[
		|\nabla\psi(x_1,x_2)|^2=\rho^2(x_1,x_2)q^2(x_1,x_2),\qquad\frac{q^2}{2}+h(\rho)=gs.
	\]
	Since the density on the local branch is bounded above and below, the condition \eqref{gro1} is equivalent to \(q^2\le Cgs\). By \(\frac{q^2}{2}+h(\rho)=gs\) and \(h'(\rho)=\tfrac{p'(\rho)}{\rho}>0\), the same condition can be written as
	\[
		h(\rho)\ge-Cgs,\qquad\text{or equivalently,}\qquad(p_{\rm atm}-p(\rho))^+\le C\bar\rho_0gs,
	\]
	In this way, we see that the density \(\rho\) differs from its free-surface value \(\bar\rho_0\) only to order \(s\) near the stagnation line.
\end{remark}

\begin{remark}
	In \cite[Theorem~11.2]{MR2810856}, it is proved that in a neighborhood of the stagnation points, the topological free boundary is the union of two \(C^1\)-graphs with right and left tangents. The argument there relies on the flatness-implies-regularity result of \cite{MR618549} and the sharp Bernstein estimate, expressed in the present physical normalization as
	\[
		|\nabla\psi|^2\le 2g\bar\rho_0^2(x_2^{\mathrm{st}}-x_2)^+
		\quad\text{in }\Omega\cap\{\psi>0\}.
	\]

	For the quasilinear Bernoulli-type problem, the flatness and regularity results of \cite[Definition~5.1 and Theorems~6.1--6.3]{MR752578} are established only for local minimizers, whereas our weak solutions need not be minimizers. In \cite[Remark~6.4]{MR752578}, the authors already considered the nonlinear term \(\rho(|\nabla\psi|^2;x_2)\) that depends on the spatial coordinates. Therefore the main difficulty lies at the stagnation line \(\{x_2=x_2^{\mathrm{st}}\}\), where the prescribed boundary gradient \(\bar\rho_0\sqrt{2g(x_2^{\mathrm{st}}-x_2)}\) vanishes. We do not pursue this \(C^1\)-regularity step in Theorem~\ref{thm.main-A} and leave it for future work.
\end{remark}

\begin{maintheorem}{B}\label{thm.main-B}
	Let \(\psi\) be a subsonic variational solution of \eqref{eq.main-physical-problem} satisfying \eqref{gro1}. Then the set of stagnation points \(S\) is at most countable, \(\Sigma\) is locally finite in \(\Omega\), and every accumulation point of \(S\) lying in \(\Omega\) belongs to \(\Sigma\).

	At each \(x^{\mathrm{st}}=(x_1^{\mathrm{st}},x_2^{\mathrm{st}})\in S\setminus\Sigma\), the rescaling
	\[
		\frac{\psi\bigl((x_1^{\mathrm{st}},x_2^{\mathrm{st}})+r(x_1,x_2)\bigr)}{r^{3/2}}\longrightarrow\psi_{\mathrm{st}}(x_1,x_2)
	\]
	locally uniformly and strongly in \(W^{1,2}_{\mathrm{loc}}(\R^2)\) as \(r\to0^+\). In addition,
	\begin{equation*}
		\begin{split}
			\mathcal L^2\Bigl(B_R\cap\Bigl(
			\bigl\{(x_1,x_2):\psi\bigl((x_1^{\mathrm{st}},x_2^{\mathrm{st}})+r(x_1,x_2)\bigr)>0\bigr\}\\
			{}\triangle\bigl\{(x_1,x_2):\psi_{\mathrm{st}}(x_1,x_2)>0\bigr\}
			\Bigr)\Bigr)\longrightarrow0
			\qquad(R>0,\ r\to0^+).
		\end{split}
	\end{equation*}
	Here \(\triangle\) denotes the symmetric difference of sets.
	At each \(x^{\mathrm{st}}=(x_1^{\mathrm{st}},x_2^{\mathrm{st}})\in\Sigma\), there is an integer \(N=N (x_1^{\mathrm{st}},x_2^{\mathrm{st}})\ge2\) such that
	\begin{equation*}
		\frac{\psi\bigl((x_1^{\mathrm{st}},x_2^{\mathrm{st}})+r(x_1,x_2)\bigr)}{r^\beta}\longrightarrow0
		\qquad(r\to0^+)
	\end{equation*}
	strongly in \(W^{1,2}_{\mathrm{loc}}(\R^2)\),
	for every real \(\beta<N\). The \(L^2\)-boundary-normalized rescaling satisfies
	\begin{equation}\label{eq.main-normalized-profile}
		\begin{split}
			&\frac{\psi\bigl((x_1^{\mathrm{st}},x_2^{\mathrm{st}})+r(x_1,x_2)\bigr)}
			{\displaystyle\sqrt{r^{-1}\int_{\partial B_r((x_1^{\mathrm{st}},x_2^{\mathrm{st}}))}\psi^2\,\dH{1}}}\\
			&\qquad\longrightarrow
			\frac{\varrho^N
			\bigl|\sin\bigl(N\min\{\max\{\theta,-\pi\},0\}\bigr)\bigr|}
			{\sqrt{\pi/2}}
			\qquad(r\to0^+),
		\end{split}
	\end{equation}
	strongly in \(W^{1,2}_{\mathrm{loc}}(\R^2)\). Moreover, the positive phase converges
	\[
		\begin{split}
			\mathcal L^2\Bigl(B_R\cap\Bigl(
			\bigl\{(x_1,x_2):\psi\bigl((x_1^{\mathrm{st}},x_2^{\mathrm{st}})+r(x_1,x_2)\bigr)>0\bigr\}\\
			{}\triangle\bigl\{(x_1,x_2):x_2<0\bigr\}
			\Bigr)\Bigr)\longrightarrow0
			\qquad(R>0,\ r\to0^+).
		\end{split}
	\]
\end{maintheorem}

\begin{remark}
	Theorem~\ref{thm.main-B} is a compressible counterpart of \cite[Theorem~B and Theorem~11.1]{MR2810856} under weaker assumptions \eqref{gro1}.
\end{remark}

Theorems~\ref{thm.main-A} and~\ref{thm.main-B} extend the classification of stagnation points and blow-up limits of V\u{a}rv\u{a}ruc\u{a} and Weiss~\cite[Theorems A and B]{MR2810856} to compressible irrotational flows under the weaker local growth bound~\eqref{gro1}. To help the reader compare our results with the incompressible theory, we summarize the main features of both in Table~\ref{tab.main-comparison}, with both growth bounds expressed in the present physical normalization.

\begin{table}[!ht]
	\centering
	\caption{Comparison of the main results with the incompressible case}
	\label{tab.main-comparison}
	\small
	\renewcommand{\arraystretch}{1.2}
	\setlength{\tabcolsep}{4pt}
	\begin{tabular}{@{}>{\raggedright\arraybackslash}p{0.25\textwidth}>{\raggedright\arraybackslash}p{0.34\textwidth}>{\raggedright\arraybackslash}p{0.34\textwidth}@{}}
	\toprule
	& \textbf{Compressible water wave}
	& \textbf{Incompressible water wave \cite{MR2810856}} \\
	\midrule
	\textbf{Governing Equation}
	& Quasilinear equation \(\operatorname{div}(\tfrac{\nabla\psi}{\rho(|\nabla\psi|^2;x_2)})=0\)
	& Linear equation \(\Delta\psi=0\) \\
	\textbf{Growth assumption}
	& \(|\nabla\psi|^2\le C_0(x_2^{\mathrm{st}}-x_2)^+\)
	& \(\begin{aligned}
		&|\nabla\psi|^2\\[-2pt]
		&\quad\le 2g\bar\rho_0^2(x_2^{\mathrm{st}}-x_2)^+
	   \end{aligned}\) \\
	\textbf{Main conclusions}
	& Stokes corner (Theorem~\ref{thm.main-A}); complete classification in two dimensions for subsonic variational solutions (Theorem~\ref{thm.main-B})
	& Stokes corner (Theorem A); complete classification in two dimensions for weak solutions (Theorem B) \\
	\bottomrule
	\end{tabular}
\end{table}
\FloatBarrier

The geometric conclusion proved below concerns the secant tangents at a stagnation point when the free boundary is locally an embedded continuous curve.

\begin{definition}[Local curve assumption]\label{def.pc-local-curve}
	We assume that, in a neighborhood of \(x^0\in\partial\{\psi>0\}\cap\{x_2=x_2^{\mathrm{st}}\}\), the free boundary is a single embedded continuous curve with a continuous injective parametrization
	\[
		\sigma:(-\tau_0,\tau_0)\longrightarrow\mathbb R^2,
		\qquad\sigma(0)=x^0.
	\]
	The parametrization is a homeomorphism onto its image locally at \(x^0\). The images \(\sigma((-\tau_0,0))\) and \(\sigma((0,\tau_0))\) are called its two local branches.
\end{definition}

\begin{maintheorem}{C}\label{thm.main-C}
	Let \(\psi\) be a subsonic variational solution of \eqref{eq.main-physical-problem} satisfying \eqref{gro1}. Fix \(x^{\mathrm{st}}=(x_1^{\mathrm{st}},x_2^{\mathrm{st}})\in S\). Suppose in addition that the free boundary near \(x^{\mathrm{st}}\) satisfies the local curve assumption of Definition \ref{def.pc-local-curve}. Then \(x^{\mathrm{st}}\notin\Sigma\) and is isolated in \(S\), and
	\[
		\frac{\psi\bigl((x_1^{\mathrm{st}},x_2^{\mathrm{st}})+r(x_1,x_2)\bigr)}{r^{3/2}}\longrightarrow\psi_{\mathrm{st}}(x_1,x_2)
	\]
	locally uniformly and strongly in \(W^{1,2}_{\mathrm{loc}}(\R^2)\) as \(r\to0^+\). In addition,
	\begin{equation*}
		\begin{split}
			\mathcal L^2\Bigl(B_R\cap\Bigl(
			\bigl\{(x_1,x_2):\psi\bigl((x_1^{\mathrm{st}},x_2^{\mathrm{st}})+r(x_1,x_2)\bigr)>0\bigr\}\\
			{}\triangle\bigl\{(x_1,x_2):\psi_{\mathrm{st}}(x_1,x_2)>0\bigr\}
			\Bigr)\Bigr)\longrightarrow0
			\qquad(R>0,\ r\to0^+).
		\end{split}
	\end{equation*}
	Finally, the two branches have secant tangents at \(x^{\mathrm{st}}\) (cf. Figure \ref{Fig: stokes}), and either
	\[
		\lim_{t\to0^+}\frac{\sigma_2(t)-x_2^{\mathrm{st}}}{\sigma_1(t)-x_1^{\mathrm{st}}}=\frac{1}{\sqrt{3}},\qquad\text{ and }\qquad\lim_{t\to0^-}\frac{\sigma_2(t)-x_2^{\mathrm{st}}}
		{\sigma_1(t)-x_1^{\mathrm{st}}}=-\frac{1}{\sqrt{3}},
	\]
	or
	\[
		\lim_{t\to0^+}\frac{\sigma_2(t)-x_2^{\mathrm{st}}}{\sigma_1(t)-x_1^{\mathrm{st}}}=-\frac{1}{\sqrt{3}},\qquad\text{ and }\qquad\lim_{t\to0^-}\frac{\sigma_2(t)-x_2^{\mathrm{st}}}
		{\sigma_1(t)-x_1^{\mathrm{st}}}=\frac{1}{\sqrt{3}}.
	\]
\end{maintheorem}

\begin{remark}
	If the local curve assumption stated above holds at every stagnation point of an open set \(O\subset\Omega\), then \(S\cap O\) is locally finite in \(O\). Moreover, for small nonzero \(t\), the two branches lie on opposite sides of \(x_1=x_1^{\mathrm{st}}\), since the two slopes in Theorem~\ref{thm.main-C} have opposite signs.
\end{remark}

\begin{figure}[!ht]
	\centering
	\tikzset{every picture/.style={line width=0.75pt}} 

	\begin{tikzpicture}[x=1pt,y=1pt,yscale=-1,xscale=1]

	\draw [color={rgb, 255:red, 128; green, 128; blue, 128 }  ,draw opacity=1 ] [dash pattern={on 0.84pt off 2.51pt}]  (219.4,100.35) -- (429,100.35) ;
	\draw    (328.19,100.35) ;
	\draw [shift={(328.19,100.35)}, rotate = 0] [color={rgb, 255:red, 0; green, 0; blue, 0 }  ][fill={rgb, 255:red, 0; green, 0; blue, 0 }  ][line width=0.75]      (0, 0) circle [x radius= 2.01, y radius= 2.01]   ;
	\draw  [color={rgb, 255:red, 0; green, 0; blue, 255 }  ,draw opacity=1 ] (241,139.03) .. controls (278,130) and ({328.19-10*sqrt(3)},110.35) .. (328.19,100.35) .. controls ({328.19+10*sqrt(3)},110.35) and (380,126) .. (416.31,137.34) ;
	\draw  [dash pattern={on 0.84pt off 2.51pt}]  (328.19,100.35) -- (420.67,{100.35+92.48/sqrt(3)}) ;
	\draw  [dash pattern={on 0.84pt off 2.51pt}]  (328.19,100.35) -- (241.33,{100.35+86.86/sqrt(3)}) ;

	\draw (324.39,83.42) node [anchor=north west][inner sep=0.75pt]  [font=\footnotesize]  {\(x^{\mathrm{st}}\)};
	\draw (318.45,108.06) node [anchor=north west][inner sep=0.75pt]  [font=\footnotesize]  {\(120^{\circ }\)};
	\draw (316.54,132.55) node [anchor=north west][inner sep=0.75pt]  [font=\footnotesize]  {\(\psi  >0\)};
	\draw (248.87,106.01) node [anchor=north west][inner sep=0.75pt]  [font=\footnotesize]  {\(\psi =0\)};
	\draw (383.85,104.78) node [anchor=north west][inner sep=0.75pt]  [font=\footnotesize]  {\(\psi =0\)};
	\draw (432.33,91.93) node [anchor=north west][inner sep=0.75pt]  [font=\footnotesize]  {\(x_{2} =x_{2}^{\mathrm{st}}\)};

	\end{tikzpicture}
	\caption{The Stokes corner and its two secant directions}
	\label{Fig: stokes}
\end{figure}

\begin{remark}
	Note that the conclusion of Theorem~\ref{thm.main-C} determines the two geometric tangent directions at \(x^{\mathrm{st}}\). In physical coordinates they are
	\[
		\lim_{t\to0^\pm}
		\frac{\sigma(t)-x^{\mathrm{st}}}{|\sigma(t)-x^{\mathrm{st}}|}
		\in
		\left\{
		\left(\frac{\sqrt3}{2},-\frac12\right),
		\left(-\frac{\sqrt3}{2},-\frac12\right)
		\right\},
	\]
	with the two limits distinct.
\end{remark}

\subsection{Ideas of the proof}

First, note that the quasilinear free boundary problem \eqref{eq.main-physical-problem} has a variational structure, as in the theory of quasilinear Bernoulli problems and compressible jets developed by Alt, Caffarelli, and Friedman~\cite{MR752578,MR772122}. This means that solutions of \eqref{eq.main-physical-problem} can be found as minimizers of the corresponding energy functional. However, we are interested in solutions that need not minimize the energy functional. In this paper, we focus on a broader class of solutions known as~\emph{subsonic variational solutions}, which are suitable for studying singularities in water wave problems~\cite{MR4238496}. In our context, they are defined as solutions of the equation in the sense of the first domain variation (see Definition~\ref{def.subsonic-variation-solution}), with additional requirements such as nonnegativity and uniform subsonicity. To study the singular asymptotics at the stagnation points, we aim to classify the subsonic variational solutions that remain invariant under the scaling of the equation. Here we encounter our first obstacle: the absence of a Weiss-type monotonicity formula for quasilinear elliptic equations. The Weiss-type monotonicity formula, first developed by Weiss in~\cite{MR1620644,MR1759450}, yields the following observation for the two-dimensional incompressible water wave problem. Let \(\psi\) be a variational solution of the problem with energy
\begin{align*}
	\int\bigl(|\nabla\psi|^2+2g\bar\rho_0^2(x_2^{\mathrm{st}}-x_2)\chi_{\{\psi>0\}}\bigr)\,\dx{x},
\end{align*}
where \(\chi_{\{\psi>0\}}=1\) on \(\{\psi>0\}\) and \(0\) otherwise. Then, at each stagnation point \(x^{\mathrm{st}}\in\partial\{\psi>0\}\), the function
\[
	\begin{split}
		r\mapsto\Bigg[
		  &r^{-3}\int_{B_r(x^{\mathrm{st}})}
		  \Bigl(|\nabla\psi|^2
		  +2g\bar\rho_0^2(x_2^{\mathrm{st}}-x_2)\chi_{\{\psi>0\}}\Bigr)\,\dx{x}\\
		  &-\frac32r^{-4}
		    \int_{\partial B_r(x^{\mathrm{st}})}\psi^2\,\dH{1}
		\Bigg]
	\end{split}
\]
is nondecreasing in \(r\); see~\cite[Theorem 3.5, Boundary case]{MR2810856}.
To analyze the blow-up at a stagnation point \(x^{\mathrm{st}}\in S\), we center the coordinates at \(x^{\mathrm{st}}\) and introduce the natural rescaling
\begin{equation}\label{eq.sca}
	\psi_r(x)=\frac{\psi(x^{\mathrm{st}}+rx)}{r^{3/2}}.
\end{equation}
The Weiss-type monotonicity allows us to derive that every subsequential limit \(\psi_0:=\lim_{j\to\infty}\psi_{r_j}\), with \(r_j\to0^+\), is a homogeneous function of degree \(3/2\).
The original quasilinear problem, however, is not invariant under this rescaling. One therefore needs a monotonicity formula adapted to quasilinear free-boundary problems. In our context, the governing equation for compressible irrotational flows is a quasilinear equation \eqref{eq.main-physical-problem} and the energy takes the form
\[
	\int
	\Bigl[\int_0^{|\nabla\psi|^2}\frac{\dx{\tau}}{\rho(\tau;y_2)}+2g\,\chi_{\{\psi>0\}}\int_{y_2}^{x_2^{\mathrm{st}}}\rho(0;h)\,\dx{h}\Bigr]\dx{y},
\]
where \(y=(y_1,y_2)\) and \(x_2^{\mathrm{st}}\) is the stagnation height. The variational structure in this setting is not invariant under the \(3/2\)-homogeneous rescaling \eqref{eq.sca} due to the non-homogeneity of the governing operator. This lack of rescaling invariance makes it difficult to obtain any information from compactness arguments. However, in this paper, we observe that

\begin{itemize}
	\item The energy above consists of a homogeneous part that has a rescaling invariance similar to that in the incompressible case.
	\item The nature of the subsonicity condition provides additional convexity, which helps control the extra terms arising from the compressibility of the flow.
\end{itemize}

Let us remark that the robust structure of the Euler equations plays a key role in our construction of the first monotonicity formula for quasilinear free boundary problems of Bernoulli type. To the best of our knowledge, there is no general method for constructing monotonicity formulas for nonlinear free boundary problems; see~\cite{MR4556790}. We believe, however, that our method can be extended to construct monotonicity formulas for general quasilinear problems of Bernoulli type, and we leave this for future work.

More precisely, for the energy functional of problem \eqref{eq.main-physical-problem}, we consider the Weiss-type quantity
\[
	\begin{split}
		\bar\Phi(x^{\mathrm{st}},\psi;r)
		={}&r^{-3}\int_{B_r(x^{\mathrm{st}})}
		\Bigl[\int_0^{|\nabla\psi|^2}\frac{\dx{\tau}}{\rho(\tau;y_2)}\\
		&\qquad\qquad+2g\,\chi_{\{\psi>0\}}\int_{y_2}^{x_2^{\mathrm{st}}}\rho(0;h)\,\dx{h}\Bigr]\dx{y}\\
		&-\frac{3}{2\bar\rho_0}r^{-4}\int_{\partial B_r(x^{\mathrm{st}})}\psi^2\,\dH{1},
	\end{split}
\]
We prove in Proposition~\ref{prop:wsm} that its derivative is the sum of a nonnegative square and four remainder terms with coefficients, the latter arising from the compressibility of the flow. Under the local growth condition, this yields the almost-monotonicity of \(r\mapsto\bar\Phi(x^{\mathrm{st}},\psi;r)+Cr^2\) for sufficiently small \(r\) (cf. Lemma~\ref{lem.blow-up limits}). Consequently the limit \(\bar\Phi(x^{\mathrm{st}},\psi;0^+)\) exists, the natural blow-ups are compact with respect to local uniform and strong \(W^{1,2}_{\mathrm{loc}}\) convergence, and their limits are \(3/2\)-homogeneous and harmonic in their positive phases.

Note that the strong Sobolev convergence alone does not determine the limit of \(\chi_{\{\psi_r>0\}}\), which need not coincide with the positivity indicator \(\chi_{\{\psi_0>0\}}\) of the limiting function \(\psi_0\). We therefore prove phase compactness for subsonic variational solutions, using the weighted perimeter estimate available under the growth condition \eqref{gro1}. The limiting function \(\psi_0\) and the limiting phase then satisfy the frozen domain-variation identity. Using this identity, Proposition~\ref{Proposition: 2-dimensional case} shows that the normalized density \(\Theta(x^{\mathrm{st}})\) defined in \eqref{eq.normalized-weiss-density} takes one of three values, summarized in Table~\ref{tab.densities}.

\begin{table}[!ht]
	\centering
	\caption{The three possible normalized Weiss densities \eqref{eq.normalized-weiss-density} and the corresponding singularities}
	\label{tab.densities}
	\small
	\renewcommand{\arraystretch}{1.25}
	\begin{tabular}{@{}cc@{}}
		\toprule
		\(\Theta(x^{\mathrm{st}})\) & Singularity \\
		\midrule
		\(0\) & zero-density cusp \\
		\(\dfrac{\sqrt3}{3\bar\rho_0}\) & Stokes corner \\
		\(\dfrac{2}{3\bar\rho_0}\) & points of full weighted density \\
		\bottomrule
	\end{tabular}
\end{table}
\FloatBarrier

It remains to exclude the cusp case and points of full weighted density. Inspired by 
McCurdy~\cite{McC25}, we combined a weighted isoperimetric argument with the Radon 
measure estimate established in Lemma~\ref{Lemma: weak v.s variational} to show that 
the cusp singularities are not possible. It should be noted that our argument does not rely on the strong Bernstein estimates used in \cite{MR2810856}.

It remains to exclude points of full weighted density. At these points, we have the vanishing rescaling \(r^{-3/2}\psi(x^{\mathrm{st}}+rx)\to 0\). To overcome this, we consider the following \(L^2\)-boundary normalization
\[
	\varphi_r(x)=\frac{\psi(x^{\mathrm{st}}+rx)}{\left(r^{-1}\int_{\partial B_r(x^{\mathrm{st}})}\psi^2\,\dH{1}\right)^{1/2}}.
\]
At this stage, we need to establish the compactness for such a blow-up sequence. In 
Propositions~\ref{prop.fd-modified-Weiss} and~\ref{prop.fd-frequency}, we construct a modified Weiss-type monotonicity \(\Psi_\beta\) and a modified frequency inequality for \(\mathsf H_\beta\), with perturbation parameter \(0<\beta<1\). These estimates make \(r\mapsto e^{Cr}\mathsf H_\beta(r)\) nonnegative and nondecreasing, which gives a finite limiting frequency and uniform local \(W^{1,2}\) bounds for the sequence \(\varphi_r\).

However, the uniform \(W^{1,2}\) bounds of \(\varphi_r\) give only the 
weak convergence, so we have to improve it to strong convergence. To this end, we apply the concentration cancellation property developed by Evans and M\"uller~\cite{MR1220787}, and proved the strong convergence in Theorem~\ref{thm.pc-strong-convergence}. The results show that the blow-up limits in two dimensions have integer homogeneity of degree at least two, and the full weighted density points form a locally finite set (cf. Theorem~\ref{thm.pc-quantization} and Proposition~\ref{prop.pc-local-finiteness}). Finally, the local topological hypotheses of Theorems~\ref{thm.main-A} and~\ref{thm.main-C} exclude this remaining alternative. Thus completing the proof of the Stokes conjecture for compressible irrotational flows.

\subsection{History and background}

The classical Stokes conjecture concerns the angle at a stagnation crest of an extreme incompressible gravity wave. Existence of extreme waves and the analysis of their crest angles were studied through conformal and integral-equation formulations; see, for example, Toland~\cite{MR513927} and McLeod~\cite{MR1446239}. The Stokes conjecture was rigorously established by Amick, Fraenkel, and Toland~\cite{MR666110} and independently by Plotnikov~\cite{MR752600,MR1883094}. We also remark that in \cite{MR2038344}, Plotnikov and Toland proved the convexity of the free surface, which is also part of the conjecture proposed by Stokes.

The variational theory of Bernoulli free boundaries originates in the existence and regularity theory of Alt and Caffarelli~\cite{MR618549}. The corresponding quasilinear theory of elliptic free boundary problems was first investigated by Alt, Caffarelli and Friedman~\cite{MR752578}, who studied non--negative solutions \(\psi\in C^{0}(\Omega)\) that solve a quasilinear elliptic equation in \(\Omega \cap \{\psi>0\}\), as well as exhibit a non--degeneracy property at each point on \(\Omega\cap\partial\{\psi>0\}\). In a series of work of Caffarelli~\cite{MR990856,MR973745,MR1029856}, the viscosity approach of the free boundary was developed, and this classical method was 
advanced by De Silva in \cite{MR2813524}. In \cite{MR1620644,MR1759450}, Weiss introduced the notion of variational solutions and developed a complete theory of monotonicity formulas. In the present problem, the Bernoulli coefficient vanishes on the stagnation level \(x_2=x_2^{\mathrm{st}}\), and the free boundary can meet that line at points where the natural blow-up is zero.

A geometric approach to the Stokes conjecture for incompressible waves was introduced by V\u{a}rv\u{a}ruc{a} and Weiss~\cite{MR2810856}. Their monotonicity and nonlinear frequency formulas classify stagnation points into three cases. Under the strong Bernstein estimates and topological hypotheses, Stokes asymptotics are obtained without assuming symmetry, monotonicity of the free surface, or isolation of stagnation points in advance. The approach was subsequently developed for waves with vorticity~\cite{MR2995099} and for axisymmetric gravity flows~\cite{MR3225630}. Related studies of axisymmetric flows with vorticity are given in~\cite{MR4595616,MR4808256,DPY25}. A broader account of the relation between Bernoulli free boundaries and water waves can be found in~\cite{MR4238496}.

Euler equations involving a free boundary have been studied by many authors. For small initial data in infinite depth, the almost-global well-posedness of the two-dimensional incompressible irrotational water-wave problem and the global well-posedness of the three-dimensional problem were established by Wu~\cite{MR2507638,MR2782254}. Local well-posedness and a priori estimates for free-boundary compressible gas models with physical vacuum were studied in \cite{MR2608125,MR2980528,MR2547977,MR3280249,MR3218831}. The theory of free-boundary compressible liquids on bounded fluid domains, including linearized and nonlinear well-posedness, a priori estimates, and limiting problems, was developed in \cite{MR3139610,MR4097326,MR4072680,MR1981993,MR3812074}. We also refer to~\cite{MR2560044} for an existence result when the fluid domain is unbounded.

The study of subsonic flows has a long history. Bers studied existence and uniqueness for steady subsonic potential flow past a two-dimensional body~\cite{MR96477}; Finn and Gilbarg obtained the far-field behavior past a profile~\cite{MR86556} and studied three-dimensional flows~\cite{MR92912}, as did Dong~\cite{MR1134129}. For supersonic flow past a perturbed cone, Chen, Xin, and Yin proved global existence of shock waves~\cite{MR1911248}, with the three-dimensional curved cone treated in~\cite{MR2216450}; Xin et al.~\cite{MR2533922} established the existence and nonexistence of transonic shocks. Subsonic-sonic flows were studied by compensated compactness~\cite{MR2291790}, which was later extended to higher dimensions~\cite{MR3437861}, while Xin and collaborators established well-posedness of compressible flow in general nozzles~\cite{MR2375709,MR2607929,MR2644144,MR2533922}. Subsonic irrotational flows in a finitely long nozzle with variable end pressure were studied in~\cite{MR3178073}, and subsonic-sonic flows in a convergent nozzle in~\cite{MR3048597}; see~\cite{MR2824469} for higher-dimensional subsonic-sonic potential flows and~\cite{MR3196988} for the case with vorticity.

\section{Notation and preliminaries}\label{sec:notations-preliminaries}

In the rest of the paper, we work with an \(n\)-dimensional (\(n\ge2\)) scalar generalization of the problem described in the introduction (cf. \eqref{eq.main-physical-problem}). We consider nonnegative functions \(\psi\) on \(\Omega\), extended by zero in \(\Omega\cap\{y_n\ge y_n^{\rm st}\}\), satisfying the quasilinear free boundary problem
\begin{align}\label{eq.quasilinear-free-boundary-problem1}
	\left\{
		\begin{alignedat}{2}
			&\operatorname{div}\left( \frac{\nabla\psi(y)}{\rho(|\nabla\psi(y)|^2;y_n)}\right) = 0\qquad&&\text{ in }\Omega\cap\{\psi>0\},\\
			&|\nabla\psi(y)|^2 =2\bar\rho_0^2g\left( y_n^{\rm st}-y_n \right) \qquad&&\text{ on }\Omega\cap \partial\{\psi>0\},
		\end{alignedat}
	\right.
\end{align}
where we denote a point \(y\in\R^n\) as \(y=(y',y_n)=(y_1,y_2,\dots,y_n)\) and the candidate stagnation point \(y^{\rm st}=(y_1^{\rm st},\dots,y_n^{\rm st})\). In \eqref{eq.quasilinear-free-boundary-problem1}, \(\Omega\) is a bounded domain in \(\R^n\) which has a nonempty intersection with the hyperplane \(\{y_n=y_n^{\rm st}\}\). \(\bar\rho_0>0\) is the constant density on the free boundary \(\Omega\cap\partial\{\psi>0\}\) and
\begin{align*}
	y_n^{\mathrm{st}}:=\frac{\cB}{g}=\frac{p'(\bar{\rho}_{0})}{2g}>0.
\end{align*}
In the new coordinates \(x'=y'\) and \(x_n=y_n^{\rm st}-y_n\), we set \(\tilde\Omega:=\{(y',y_n^{\rm st}-y_n):y\in\Omega\}\). Using the zero extension of \(\psi\), we define the scaled stream function
\begin{align}\label{eq.upsi}
	u(x)=\frac{\psi(y)}{\bar{\rho}_0\sqrt{2g}},
\end{align}
on \(\tilde\Omega\) and the rescaled density
\begin{equation*}
	H(t;s)=\rho(2\bar{\rho}_0^2gt;y_n^{\rm st}-s).
\end{equation*}
Denote
\begin{equation*}
	H_0:=\bar\rho_0.
\end{equation*}
It follows from \eqref{eq.bernoulli-law} that the following normalized Bernoulli relation holds:
\begin{align}\label{eq.normalized-bernoulli}
	\frac{gH_0^2t}{H^2(t;s)}+h(H(t;s))-gs=0.
\end{align}
We introduce the auxiliary function \(\mathscr G(\rho,t,s)=\rho^{-2}gH_0^2t+h(\rho)-gs\). It follows that
\[
	\mathscr G(H_0,0,0)=0,\qquad
	\partial_\rho\mathscr G(H_0,0,0)=\frac{p'(H_0)}{H_0}>0.
\] The implicit function theorem implies that there exists a unique local positive function \(\rho=H(t;s)\) with \(H(0;0)=H_0\). Differentiating \eqref{eq.normalized-bernoulli} with respect to \(t\) and \(s\) gives
\[
	\partial_tH(t;s) = -\frac{gH_0^2}{H(c^2-q^2)}<0,\qquad\text{ and }\qquad\partial_sH(t;s) = \frac{gH}{c^2-q^2}>0,
\]
where \(q^2(t;s)=2gH_0^2t/H^2(t;s)\) and \(c^2(t;s)=p'(H(t;s))\). In particular, the following identities and derivative formulas hold
\begin{equation}\label{eq.H-derivatives-identity}
	\frac{\partial_sH}{H^2} =-\frac{\partial_tH}{H_0^2},\qquad\partial_tH(0;0) =-\frac{gH_0}{p'(H_0)},\qquad\partial_sH(0;0) = \frac{gH_0}{p'(H_0)},
\end{equation}
in the subsonic region \(\{q^2<c^2\}\). A further useful identity is
\begin{equation}\label{eq.useful-identity}
	H(s;s)=H_0\quad\text{ for sufficiently small }s.
\end{equation}
Indeed, \(\rho=H_0\) solves \eqref{eq.normalized-bernoulli} when \(t=s\), and the local uniqueness of the solution guaranteed by the implicit function theorem implies the identity \eqref{eq.useful-identity}. In these normalized variables, a direct calculation shows that the free boundary problem \eqref{eq.quasilinear-free-boundary-problem1} becomes
\begin{align}\label{eq.quaslinear-free-boundary-problem2}
	\left\{
		\begin{alignedat}{2}
			&\operatorname{div}\left( \frac{\nabla u(x)}{H(|\nabla u(x)|^2;x_n)} \right)=0\qquad&&\text{ in }\tilde\Omega\cap\{u>0\},\\
			&|\nabla u(x)|^2=x_n\qquad&&\text{ on }\tilde\Omega\cap\partial\{u>0\}.
		\end{alignedat}
	\right.
\end{align}
The boundary equalities in \eqref{eq.quasilinear-free-boundary-problem1}
and \eqref{eq.quaslinear-free-boundary-problem2} are understood classically
on regular one-sided patches; for the generalized solution classes below,
they are defined through domain stationarity.
In addition, we assume that
\[
	u\ge0\quad\text{in }\tilde\Omega,\qquad u\equiv0\quad\text{in }\tilde\Omega\cap\{x_n\le0\}.
\]
The stagnation points are now located on the hyperplane \(\{x_n=0\}\). The domain \(\tilde\Omega\) is bounded and satisfies \(\tilde\Omega\cap\{x_n=0\}\neq\varnothing\). The next lemma records the local coefficient bounds on this branch.

\begin{lemma}\label{lem.preliminary}
	Let \(g,H_0>0\), and let \(p\in C^2(I)\), where \(I\subset(0,\infty)\) is an open interval containing \(H_0\) and \(p'>0\) on \(I\). Let \(h\) be the enthalpy from Section~\ref{sec.euler}, and let \(H\) be the local density branch determined by the normalized Bernoulli relation \eqref{eq.normalized-bernoulli}, with \(H(0;0)=H_0\). Here \(M^2(t;s)=q^2(t;s)/c^2(t;s)\) is the squared Mach number computed from the quantities in \eqref{eq.H-derivatives-identity}.

	There exist a neighborhood \(\cN\) of \((0,0)\) and constants \(0<c\le C\), depending only on \(g\), \(H_0\), and the pressure law near \(H_0\), such that
	\begin{equation}\label{eq.esm-H}
		\begin{aligned}
			&\frac{H_0}{2}\le H(t;s)\le2H_0,\qquad M^2(t;s)\le\frac12,\\
			&-C\le\partial_tH(t;s)\le-c<0,\\
			&0<c\le\partial_sH(t;s)\le C
		\end{aligned}
	\end{equation}
	for every \((t,s)\in\cN\) with \(t\ge0\).

	Assume in addition that \(x^0\in\tilde\Omega\cap\partial\{u>0\}\) satisfies \(x_n^0=0\), and that a solution \(u\) of \eqref{eq.quaslinear-free-boundary-problem2} satisfies
	\[
		\lim_{\substack{x\to x^0\\ x\in\{u>0\}}}|\nabla u(x)|=0.
	\]
	Then there is a radius \(0<r_0<\tfrac12\operatorname{dist}(x^0,\partial\tilde\Omega)\) such that \eqref{eq.esm-H} holds at
	\[
		(t,s)=(|\nabla u(x)|^2,x_n)
		\qquad\text{for every }x\in B_{r_0}(x^0)\cap\{u>0\}.
	\]
\end{lemma}

\begin{proof}
	The implicit-function theorem yields a \(C^2\) branch \(H\) near \((0,0)\). At this point, \(H=H_0\) and \(M^2=0\). The derivatives are given by \eqref{eq.H-derivatives-identity} as follows.
	\[
		\partial_tH(0;0)=-\frac{gH_0}{p'(H_0)}<0,
		\qquad
		\partial_sH(0;0)=\frac{gH_0}{p'(H_0)}>0.
	\]
	By continuity, we may shrink \(\cN\) so that \eqref{eq.esm-H} holds.

	For the second assertion, the fluid-side limit gives \((|\nabla u(x)|^2,x_n)\to(0,0)\) as \(\{u>0\}\ni x\to x^0\). Choosing \(r_0>0\) small enough, we obtain both \(r_0<\tfrac12\operatorname{dist}(x^0,\partial\tilde\Omega)\) and \((|\nabla u(x)|^2,x_n)\in\cN\) for every \(x\in B_{r_0}(x^0)\cap\{u>0\}\). Since \(|\nabla u|^2\ge0\), the bound \eqref{eq.esm-H} applies at these points.
\end{proof}

\begin{remark}
	Let \(\cN\) be the neighborhood determined by Lemma \ref{lem.preliminary}. For \(P\in\R^n\) and \((|P|^2,s)\in\cN\), define
	\[
		\mathbf A(P,s)
		:=D_P\left(\frac{P}{H(|P|^2;s)}\right)
		=\frac{1}{H(|P|^2;s)}\operatorname{Id}
		-\frac{2\partial_tH(|P|^2;s)}{H^2(|P|^2;s)}P\otimes P,
	\]
	where \(P\otimes P\) is the symmetric matrix with entries \(P_iP_j\). A simple calculation shows that its eigenvalues are \(1/H\) in directions orthogonal to \(P\) and \(1/[H(1-M^2)]\) in the direction of \(P\), where \(H\) and \(M\) are evaluated at \((|P|^2,s)\). For \(P=0\), all eigenvalues equal \(1/H(0;s)\). Therefore
	\[
		\frac{1}{2H_0}|\eta|^2\le \mathbf A(P,s)\eta\cdot\eta\le\frac{4}{H_0}|\eta|^2\qquad\text{for all }\eta\in\R^n,
	\]
	whenever \((|P|^2,s)\in\cN\). In particular, equation \eqref{eq.quaslinear-free-boundary-problem2} for \(u\) is uniformly elliptic in the fluid-side neighborhood provided by Lemma \ref{lem.preliminary}.
\end{remark}

From this point on, we write \(\Omega\) for the reflected domain \(\tilde\Omega\), and we fix the following notation. For a set \(A\), \(\chi_A\) denotes its characteristic function, and for sets \(A\) and \(B\),
\[
	A\triangle B:=(A\setminus B)\cup(B\setminus A).
\]
For \(a\in\R\), set \(a^{\pm}:=\max\{\pm a,0\}\), so that \(a=a^+-a^-\) and \(|a|=a^++a^-\). The Euclidean inner product is denoted by \(x\cdot y\), and the open ball of center \(x^0\) and radius \(r\) is denoted by \(B_r(x^0)\), with \(B_r:=B_r(0)\). On a given surface, \(\nu\) denotes the outward unit normal.

For vectors \(a,b\in\R^n\), the symbol \(\otimes\) denotes the tensor product \(a\otimes b\), the \(n\times n\) matrix with entries \((a\otimes b)_{ij}=a_ib_j\); equivalently, \((a\otimes b)\eta=(b\cdot\eta)a\) for every \(\eta\in\R^n\).

Let \(BV(\Omega)\) denote the space of functions of bounded variation on \(\Omega\). For \(f\in BV(\Omega)\), the distributional derivative is the vector-valued Radon measure \(Df\), and \(|Df|\) denotes its total variation; see~\cite{MR3409135}. If \(E\subset\R^n\) is a smooth open set, then
\[
	|D\chi_E|=\cH^{n-1}\mres\partial E.
\]
If \(E\) has infinite perimeter, this identity is understood locally.

Here and throughout, \(\mathcal D'(\Omega)\) denotes the space of distributions on \(\Omega\), that is, the continuous dual of \(C_c^\infty(\Omega)\). Let \(L\) denote the quasilinear operator
\[
	Lu:=\operatorname{div}\left(\frac{\nabla u}{H(|\nabla u|^2;x_n)}\right).
\]
For \(x^0\in\Omega\), we write
\begin{align}\label{d0}
	\delta:=\delta(x^0):=\frac12\operatorname{dist}(x^0,\partial\Omega)>0.
\end{align}
Whenever the coefficient estimates of Lemma \ref{lem.preliminary} are used at a stagnation point, the radius is also required to be at most the corresponding \(r_0\).

We fix once and for all \(0<s_0<t_0\) such that the rectangle
\begin{equation}\label{eq.cR}
	\cR:=(-t_0,t_0)\times(-s_0,s_0)
\end{equation}
satisfies \(\overline{\cR}\subset\subset\cN\), where \(\cN\) is the coefficient neighborhood in Lemma \ref{lem.preliminary}. The branch \(H\) is defined on a neighborhood of \(\overline{\cR}\). For \(0\le t<t_0\), \(|s|<s_0\), we set
\[
	F(t;s):=\int_0^t\frac{1}{H(\tau;s)}\,\dx{\tau},
	\qquad
	\lambda(s):=\frac1{H_0^2}\int_0^s H(0;\sigma)\,\dx{\sigma}.
\]
For functions of \((t,s)\), we write \(\partial_1=\partial_t\) and \(\partial_2=\partial_s\). Then
\begin{equation}\label{eq.F-height-identity}
	\begin{aligned}
		\partial_1F(t;s)&=\frac{1}{H(t;s)},\\
		\partial_2F(t;s)&=-\int_0^t\frac{\partial_2H(\tau;s)}{H^2(\tau;s)}\,\dx{\tau}=\frac{H(t;s)-H(0;s)}{H_0^2}.
	\end{aligned}
\end{equation}
The last equality regarding \(\partial_2F\) follows from \eqref{eq.H-derivatives-identity}. We also define \(\Lambda(t;s):=2t\partial_1F(t;s)-F(t;s)\). For \(0\le s<s_0\), the identity \(H(s;s)=H_0\) implies
\[
	\frac{\dx{}}{\dx{s}}\left(\frac{2s}{H_0}-F(s;s)\right)=\frac{H(0;s)}{H_0^2}=\lambda'(s).
\]
Since the primitives on both sides vanish at \(s=0\), we have \(\lambda(s)=\frac{2s}{H_0}-F(s;s)=\Lambda(s;s)\). Integration by parts and \(H(s;s)=H_0\) give
\begin{align}\label{eq.lambda}
	\lambda(s)=\frac{s}{H_0}+\int_0^s \tau\,\partial_1\!\left(\frac1H\right)(\tau;s)\,\dx{\tau}.
\end{align}
\begin{remark}
	Since \(\partial_1(1/H)=-\partial_1H/H^2>0\), we have \(\lambda(0)=0\) and \(\lambda(s)>\frac{s}{H_0}>0\) for \(0<s<s_0\). The bounds for \(H\) and \(\partial_1H\) in Lemma \ref{lem.preliminary} also imply
	\[
		\lambda(s)=\frac{s}{H_0}+O(s^2)\qquad\text{as }s\to0^+.
	\]
	Moreover,
	\begin{equation}\label{eq.Lambda-t-positive}
		\partial_1\Lambda(t;s)=\frac1H-\frac{2t\partial_1H}{H^2}=\frac{1}{H(1-M^2)}>0,
	\end{equation}
	where \(H\) and \(M\) are evaluated at \((t,s)\). Consequently, \(\Lambda(t;s)=\lambda(s)\) if and only if \(t=s\), for \(0\le t<t_0\) and \(0\le s<s_0\).
\end{remark}

\begin{remark}
	Recall the expression for \(\partial_2F\) in \eqref{eq.F-height-identity}. Differentiating \(\lambda(s)=2s/H_0-F(s;s)\) yields
	\begin{equation}\label{eq.lambda-derivative}
		\lambda'(s)=\frac{H(0;s)}{H_0^2}>0,
		\qquad
		\lambda''(s)=\frac{gH(0;s)}{H_0^2p'(H(0;s))}>0.
	\end{equation}
	In particular,
	\begin{equation*}
		\partial_2F(t;s)+\lambda'(s)=\frac{H(t;s)}{H_0^2}.
	\end{equation*}
	We also record the pressure representations for \(\Lambda(t;s)\) and \(\lambda(s)\):
	\begin{equation}\label{eq.Lambda-pressure-representation}
		\begin{aligned}
			\Lambda(t;s)&=\frac{p(H(0;s))-p(H(t;s))}{gH_0^2},\\
			\lambda(s)&=\frac{p(H(0;s))-p(H_0)}{gH_0^2}.
		\end{aligned}
	\end{equation}
\end{remark}

\section{A monotonicity formula for quasilinear problems}

In this section, we introduce the notion of subsonic variational solutions of \eqref{eq.quaslinear-free-boundary-problem2}, establish two identities used below, and derive a Weiss-type monotonicity formula. Note that, in the incompressible case \(H\equiv H_0\), the interior equation reduces to the Laplace equation, while \(F(t;s)=t/H_0\) and \(\lambda(s)=s/H_0\).

The second equation in \eqref{eq.quaslinear-free-boundary-problem2} is \(|\nabla u|^2=x_n\) on the regular part of the free boundary. Thus a regular free-boundary point at which the velocity vanishes must lie on \(\Pi=\{x_n=0\}\). We first introduce the stagnation hyperplane and the stagnation set.

\begin{definition}[Stagnation hyperplane and stagnation set]\label{def.stagnation-hyperplane}
	We call
	\begin{equation*}
		\Pi:=\{x\in\R^n:x_n=0\}
	\end{equation*}
	the \emph{stagnation hyperplane}. For an open set \(\Omega\subset\R^n\) and a nonnegative function \(u\in C^0(\Omega)\), define the set
	\begin{equation*}
		S_u:=\Omega\cap\partial\{u>0\}\cap\Pi.
	\end{equation*}
	If condition \eqref{eq.subsonic-condition} below holds at every point of \(S_u\), we call this set the \emph{stagnation set} of \(u\).
\end{definition}

We use the fixed coefficient rectangle \(\cR\) chosen in \eqref{eq.cR}. Throughout this section, \(\Omega\) is a bounded domain such that \(\Omega\subset\{x\in\R^n:|x_n|<s_0\}\) and \(\Omega\cap\Pi\ne\varnothing\). Let \(u\in W^{1,2}_{\mathrm{loc}}(\Omega)\cap C^0(\Omega)\) be nonnegative, with \(u\equiv0\) in \(\Omega\cap\{x_n\le0\}\), and suppose that \(u\in C^1(\{u>0\})\). We then impose the following two assumptions.
\begin{enumerate}[label=(\alph*)]
	\item \emph{Subsonic admissibility near \(\Pi\).} For every compact \(K\subset\subset\Omega\), there is a compact set \(\mathcal K_K\subset\subset\cR\) such that
	\begin{equation}\label{eq.subsonic-admissibility}
		(|\nabla u(x)|^2,x_n)\in\mathcal K_K
		\qquad\text{for a.e. }x\in K.
	\end{equation}

	\item \emph{Stagnation on \(\Pi\).} For every
	\(x^0\in S_u\),
	\begin{equation}\label{eq.subsonic-condition}
		\lim_{\substack{x\to x^0\\x\in\{u>0\}}}|\nabla u(x)|=0.
	\end{equation}
\end{enumerate}

\begin{remark}
	Condition \eqref{eq.subsonic-admissibility} is imposed in a neighborhood of \(\Pi\). It keeps the states \((|\nabla u|^2,x_n)\) in the fixed subsonic coefficient branch, which ensures that locally the governing equation is uniformly elliptic. Condition \eqref{eq.subsonic-condition} is separate: it is a fluid-side limit at the contact points of the free boundary with the stagnation hyperplane \(\Pi\).
\end{remark}

\begin{remark}
	The functions \(F\) and \(\lambda\) are defined for \(|s|<s_0\), and not only for \(s\ge0\).
\end{remark}

\begin{remark}
	On the coefficient rectangle \(\cR\) given in \eqref{eq.cR}, we infer from equations \eqref{eq.F-height-identity}, \eqref{eq.lambda-derivative}, and \(F(0;s)=\partial_2F(0;s)=0\) that
	\begin{equation}\label{eq.signed-height-coefficient-identity}
		\partial_2F(t;s)+\lambda'(s)=\frac{H(t;s)}{H_0^2}\qquad (0\le t<t_0,\ |s|<s_0).
	\end{equation}
\end{remark}

\subsection{Subsonic variational and weak solutions}

We now introduce the notion of a subsonic variational solution as follows.

\begin{definition}[Subsonic variational solution]\label{def.subsonic-variation-solution}
	A function \(u\in W^{1,2}_{\mathrm{loc}}(\Omega)\) is a \emph{subsonic variational solution} of \eqref{eq.quaslinear-free-boundary-problem2} if the following conditions hold.
	\begin{enumerate}[label=(\roman*)]
		\item \emph{Nonnegativity.} \(u\ge0\) in \(\Omega\), and
		\(u\equiv 0\) in \(\Omega\cap\{x_n\le0\}\).

		\item \emph{Regularity and the interior equation.} \(u\in C^0(\Omega)\cap C^2(\Omega\cap\{u>0\})\), and
		\begin{equation}\label{eq.quasilinear-equation}
			Lu:=\operatorname{div}\left(\frac{\nabla u}{H(|\nabla u|^2;x_n)}\right)=0\qquad\text{ in }\Omega\cap\{u>0\}
		\end{equation}
		in the classical sense.

		\item \emph{Subsonic admissibility and stagnation.} The function \(u\) satisfies \eqref{eq.subsonic-admissibility} in \(\Omega\) and \eqref{eq.subsonic-condition} at every \(x^0\in S_u\).

		\item \emph{Domain stationarity.} For a relatively compact open set \(U\subset\subset\Omega\) and for functions whose coefficient states lie in the fixed rectangle \(\cR\) a.e., define
		\begin{equation}\label{energy}
			J_F(v;U):=\int_U\left(F(|\nabla v|^2;x_n)+\lambda(x_n)\chi_{\{v>0\}}\right)\,\dx{x},
		\end{equation}
		and assume that the first variation with respect to domain variations of the functional \(J_F\) vanishes at \(v=u\), i.e.,
		\begin{equation}\label{eq.domain-variation-formula}
			\begin{split}
				0={}&\int_\Omega
				\left(F(|\nabla u|^2;x_n)+\lambda(x_n)\chi_{\{u>0\}}\right)
				\operatorname{div}\phi\,\dx{x}\\
				&-2\int_\Omega\partial_1F(|\nabla u|^2;x_n)
				\nabla u\cdot D\phi\,\nabla u\,\dx{x}\\
				&+\int_\Omega
				\left(\partial_2F(|\nabla u|^2;x_n)
				+\lambda'(x_n)\chi_{\{u>0\}}\right)\phi_n\,\dx{x},
			\end{split}
		\end{equation}
		for every \(\phi\in C_c^1(\Omega;\R^n)\).
	\end{enumerate}
\end{definition}

\begin{remark}
	Condition \eqref{eq.subsonic-admissibility} keeps \((|\nabla u|^2,x_n)\) in the local subsonic branch determined by \eqref{eq.normalized-bernoulli}. In particular, \(M^2\le1/2\), so the coefficient bounds of Lemma \ref{lem.preliminary} apply. Since these states lie in a compact subset of the coefficient rectangle \eqref{eq.cR}, every sufficiently small compactly supported domain variation remains in that rectangle.

	Condition \eqref{eq.subsonic-condition} is a separate assumption. At each \(x^0\in S_u\), it yields
	\[
		(|\nabla u(x)|^2,x_n)\to(0,0)
		\qquad\text{as }\{u>0\}\ni x\to x^0.
	\]
	By continuity of \(H\), it follows that
	\[
		H(|\nabla u(x)|^2;x_n)\to H_0,\qquad
		M^2(|\nabla u(x)|^2;x_n)\to0.
	\]
\end{remark}

\begin{remark}
	Condition \eqref{eq.subsonic-admissibility} implies that \(\nabla u\in L_{\mathrm{loc}}^\infty\). Hence the continuous representative of \(u\) is locally Lipschitz.
\end{remark}

\begin{remark}\label{rem.classical-boundary-condition}
	The domain variation is the inner variation associated with \eqref{energy}; see also \cite{MR1368401}. Suppose that, in an open set \(V\subset\subset\Omega\), the free boundary is smooth, \(V\cap\{u>0\}\) lies on precisely one side of it, and \(u\) is smooth up to the free boundary from \(\{u>0\}\). Let \(\nu_u\) denote the outward unit normal to \(\{u>0\}\). Then, for \(\phi\in C_c^1(V;\R^n)\), integration by parts in \eqref{eq.domain-variation-formula} yields
	\begin{equation*}
		\begin{split}
			0={}&2\int_{V\cap\{u>0\}}(Lu)(\nabla u\cdot\phi)\,\dx{x}\\
			&+\int_{V\cap\partial\{u>0\}}\left(\lambda(x_n)-\Lambda(|\nabla u|^2;x_n)\right)(\phi\cdot\nu_u)\,\dH{n-1}.
		\end{split}
	\end{equation*}
	The interior term vanishes by \eqref{eq.quasilinear-equation}. Since normal variations on the smooth patch are arbitrary, we obtain
	\begin{equation}\label{eq.pre-free-boundary-condition}
		\Lambda(|\nabla u|^2;x_n)=\lambda(x_n)=\Lambda(x_n;x_n)\qquad\text{on }\partial\{u>0\}\cap V.
	\end{equation}
	At nonnegative heights, \eqref{eq.Lambda-t-positive} gives \(\partial_1\Lambda(t;s)>0\). Consequently, \eqref{eq.pre-free-boundary-condition} is equivalent to \(|\nabla u|^2=x_n\) on this patch.

	The same conclusion holds for \(u\in C^2(V\cap\{u>0\})\cap C^1(V\cap\overline{\{u>0\}})\), with the equation imposed in \(\{u>0\}\), by integrating over interior parallel domains and passing to the boundary. Here the \(C^1\) regularity is understood in the one-sided sense.
\end{remark}

\begin{remark}\label{rem.flux-measure}
	Let \(u\) satisfy parts (i)--(iii) of Definition \ref{def.subsonic-variation-solution}. Then
	\[
		\mu:=Lu=\operatorname{div}\left(\frac{\nabla u}{H(|\nabla u|^2;x_n)}\right)
	\]
	is a nonnegative Radon measure in \(\Omega\), supported on \(\Omega\cap\partial\{u>0\}\).

	Indeed, take a nonnegative test function \(\xi\in C_c^\infty(\Omega)\), and choose a smooth nondecreasing function \(\theta\) which satisfies \(\theta=0\) on \(( -\infty,1]\) and \(\theta=1\) on \([2,\infty)\). It is easy to check that \(\xi\theta(ku)\) is an admissible test function, and consequently
	\[
		-\int_\Omega\frac{\theta(ku)\nabla u\cdot\nabla\xi}{H(|\nabla u|^2;x_n)}\,\dx{x}=k\int_\Omega\frac{\xi\theta'(ku)|\nabla u|^2}{H(|\nabla u|^2;x_n)}\,\dx{x}\ge0.
	\]
	Since \(\nabla u=0\) a.e. on \(\{u=0\}\), we may pass to the limit as \(k\to\infty\) and then deduce from the dominated convergence theorem that
	\[
		\langle\mu,\xi\rangle=-\int_\Omega\frac{\nabla u\cdot\nabla\xi}{H(|\nabla u|^2;x_n)}\,\dx{x}\ge0.
	\]
	This proves the desired assertion.
\end{remark}

\begin{remark}
	We note that the classical equation in \(\Omega\cap\{u>0\}\) also holds in the sense of distributions there. After zero extension, the equation on all of \(\Omega\) instead reads
	\[
		\operatorname{div}\left(\frac{\nabla u}{H(|\nabla u|^2;x_n)}\right)
		=\mu\ge0\qquad\text{in }\mathcal D'(\Omega).
	\]
	For weak solutions, the density of this measure on the regular part of \(\Omega\cap\partial\{u>0\}\cap\{x_n>0\}\) is identified in \eqref{eq.flux-perimeter-measure} below.
\end{remark}

We now introduce the notion of a subsonic weak solution. It is a subsonic variational solution with additional free-boundary regularity away from the stagnation hyperplane \(\Pi\).

\begin{definition}[Subsonic weak solution]\label{Definition: subsonic weak solution}
	A function \(u\in W^{1,2}_{\mathrm{loc}}(\Omega)\) is a \emph{subsonic weak solution} of \eqref{eq.quaslinear-free-boundary-problem2} if:
	\begin{enumerate}[label=(\roman*)]
		\item \(u\) is a subsonic variational solution in the sense of Definition \ref{def.subsonic-variation-solution}.

		\item With \(\Gamma_+:=\Omega\cap\partial\{u>0\}\cap\{x_n>0\}\), there is a relatively open subset \(\partial_{\mathrm{reg}}\{u>0\}\) of \(\Gamma_+\) with the following properties:

		For each \(x^0\in\partial_{\mathrm{reg}}\{u>0\}\), there exists an open neighborhood \(V\subset\subset\Omega\cap\{x_n>0\}\) such that \(\partial\{u>0\}\cap V\) is a \(C^{2,\alpha}\) hypersurface (\(\alpha\in(0,1)\)), \(\{u>0\}\cap V\) lies on precisely one side of this hypersurface, and \(u\in C^1(\overline{\{u>0\}}\cap V)\), with the \(C^1\) regularity understood from the positive phase. Moreover, for every \(\tau>0\),
		\[
			\cH^{n-1}\left(
			(\Gamma_+\setminus\partial_{\mathrm{reg}}\{u>0\})
			\cap\{x_n>\tau\}\right)=0.
		\]
	\end{enumerate}
\end{definition}

\begin{remark}
	The notation \(\partial_{\mathrm{reg}}\{u>0\}\) refers to the classical regular part of the topological free boundary.
\end{remark}

\begin{remark}\label{rem.weak-boundary-condition}
	Applying Remark \ref{rem.classical-boundary-condition} to each regular part gives
	\begin{equation*}
		|\nabla u|^2=x_n\qquad\text{on }\partial_{\mathrm{reg}}\{u>0\}.
	\end{equation*}
	In particular, \(H(|\nabla u|^2;x_n)=H_0\) there. Positivity on the fluid side then implies \(\nabla u\cdot\nu_u=-\sqrt{x_n}\).
\end{remark}

\begin{lemma}\label{Lemma: weak v.s variational}
	Let \(u\) be a subsonic variational solution of \eqref{eq.quaslinear-free-boundary-problem2}. Assume that
	\begin{equation}\label{eq.local-gradient-growth}
		|\nabla u|^2\le C_0x_n^+\qquad\text{locally a.e.\ in }\Omega,
	\end{equation}
	where \(C_0\) may depend on the compact subset. Then
	\[
		\chi_{\{u>0\}}\in BV_{\mathrm{loc}}(\Omega\cap\{x_n>0\}).
	\]
	Moreover, on each relatively compact open localization where \eqref{eq.local-gradient-growth} holds with one fixed constant \(C_0\), the perimeter measure and the Radon measure from Remark \ref{rem.flux-measure} satisfy
	\begin{equation}\label{eq.variational-perimeter-domination}
		\begin{aligned}
			\lambda(x_n)\,\dx{|D\chi_{\{u>0\}}|}
			&\le2\sqrt{C_0x_n}\,\dx{\mu},\\
			\sqrt{x_n}\,\dx{|D\chi_{\{u>0\}}|}
			&\le2H_0\sqrt{C_0}\,\dx{\mu},
		\end{aligned}\quad\text{ in }\{x_n>0\}.
	\end{equation}
	In addition, there is a constant \(C_1=C_1(n,C_0,H_0,H_{\min})\) such that
	\begin{equation}\label{eq.weighted-perimeter-bound}
		r^{1/2-n}\int_{B_r(y)\cap\{x_n>0\}}\sqrt{x_n}\,\dx{|D\chi_{\{u>0\}}|}\le C_1
	\end{equation}
	whenever \(B_{2r}(y)\subset\subset\Omega\) and \(y_n=0\), with \(C_0\) taken on \(B_{2r}(y)\). Here \(H_{\min}>0\) is a lower bound for \(H(|\nabla u|^2;x_n)\) on that ball.

	If \(u\) is also a subsonic weak solution, then
	\begin{equation}\label{eq.flux-perimeter-measure}
		\begin{aligned}
			    |D\chi_{\{u>0\}}|&=\cH^{n-1}\mres \partial_{\mathrm{reg}}\{u>0\},\\
			\dx{\mu}&=\frac{\sqrt{x_n}}{H_0}\,\dx{|D\chi_{\{u>0\}}|},
		\end{aligned}\quad\text{ in }\Omega\cap\{x_n>0\}.
	\end{equation}
\end{lemma}

\begin{remark}
	It follows from \eqref{eq.esm-H} that one may take \(H_{\min}=H_0/2\).
\end{remark}

\begin{remark}
	Under the assumptions of Lemma \ref{Lemma: weak v.s variational}, one has 
	\(\mu(\Omega\cap\Pi)=0\). The weighted perimeter measure on \(\Omega\cap\{x_n>0\}\) extends by zero to a locally finite measure on \(\Omega\), dominated locally by \(2H_0\sqrt{C_0}\mu\). For weak solutions this extension equals \(H_0\mu\). We defer the complete proof to Appendix \ref{Appendix: proof of weak v.s variational}.
\end{remark}

\subsection{Two preliminary identities}

The next lemma gives a Poho\v{z}aev-type identity and an energy identity. Both identities are used in the monotonicity formula.

\begin{lemma}
	Let \(u\) be a subsonic variational solution of \eqref{eq.quaslinear-free-boundary-problem2}. Fix \(x^0\in\Omega\), and put \(\delta=\operatorname{dist}(x^0,\partial\Omega)/2\) as in \eqref{d0}. For a.e.\ \(r\in(0,\delta)\), the following identities hold.
	\begin{enumerate}[label=(\roman*)]
		\item \emph{Poho\v{z}aev-type identity:}
		\begin{equation}\label{eq.pohozaev-identity}
			\begin{split}
				&nJ_F(u;B_r(x^0))-rJ_F(u;\partial B_r(x^0))\\
				&\quad=2\int_{B_r(x^0)}
				\frac{|\nabla u|^2}{H(|\nabla u|^2;x_n)}\,\dx{x}
				-2r\int_{\partial B_r(x^0)}
				\frac{(\nabla u\cdot\nu)^2}{H(|\nabla u|^2;x_n)}
				\,\dH{n-1}\\
				&\qquad\quad-\int_{B_r(x^0)}
				\left(\partial_2F(|\nabla u|^2;x_n)
				+\lambda'(x_n)\chi_{\{u>0\}}\right)(x_n-x_n^0)\,\dx{x},
			\end{split}
		\end{equation}
		where \(\nu=(x-x^0)/r\) on \(\partial B_r(x^0)\). In \eqref{eq.pohozaev-identity},
		\begin{equation}\label{eq.JF-boundary}
			J_F(u;\partial B_r(x^0)):=
			\int_{\partial B_r(x^0)}
			\left(F(|\nabla u|^2;x_n)+\lambda(x_n)\chi_{\{u>0\}}\right)
			\,\dH{n-1}.
		\end{equation}

		\item \emph{Energy identity:}
		\begin{equation}\label{eq.energy-identity}
			\int_{B_r(x^0)}
			\frac{|\nabla u|^2}{H(|\nabla u|^2;x_n)}\,\dx{x}
			=\int_{\partial B_r(x^0)}
			\frac{u\nabla u\cdot\nu}{H(|\nabla u|^2;x_n)}
			\,\dH{n-1}.
		\end{equation}
	\end{enumerate}
	Combining \eqref{eq.pohozaev-identity} and \eqref{eq.energy-identity}, we obtain
	\begin{equation}\label{eq.pohozaev-energy-identity}
		\begin{split}
			&nJ_F(u;B_r(x^0))-rJ_F(u;\partial B_r(x^0))\\
			&\quad=2\int_{\partial B_r(x^0)}
			\frac{u\nabla u\cdot\nu}{H(|\nabla u|^2;x_n)}\,\dH{n-1}
			-2r\int_{\partial B_r(x^0)}
			\frac{(\nabla u\cdot\nu)^2}{H(|\nabla u|^2;x_n)}
			\,\dH{n-1}\\
			&\qquad\quad-\int_{B_r(x^0)}
			\left(\partial_2F(|\nabla u|^2;x_n)
			+\lambda'(x_n)\chi_{\{u>0\}}\right)(x_n-x_n^0)\,\dx{x}.
		\end{split}
	\end{equation}
\end{lemma}

\begin{proof}
	For the abuse of the notation, we write 
	\[
		a(x):=\frac1{H(|\nabla u|^2;x_n)},\qquad e(x):=F(|\nabla u|^2;x_n)+\lambda(x_n)\chi_{\{u>0\}},
	\]
	and
	\[
		b(x):=\partial_2F(|\nabla u|^2;x_n)+\lambda'(x_n)\chi_{\{u>0\}}.
	\]
	Since \(u\) is a subsonic variational solution, we deduce from \eqref{eq.subsonic-admissibility} that \((|\nabla u|^2,x_n)\) is in a compact subset of the rectangle given in \eqref{eq.cR}. Thus, the coefficients and the quantities used below are locally integrable.

	\emph{(i) The Poho\v{z}aev identity.}
	Fix the center \(x^0\). Write \(R=|x-x^0|\) and set \(\nu_R=(x-x^0)/R\) when \(R>0\). Let \(\eta\) be a smooth radial cutoff function that is compactly supported in \([0,\delta)\) and constant near zero. We then introduce \(\phi(x)=(x-x^0)\eta(R)\) into \eqref{eq.domain-variation-formula}. Then a direct calculation gives
	\[
		\operatorname{div}\phi=n\eta(R)+R\eta'(R)
	\]
	and
	\[
		\nabla u\cdot D\phi\,\nabla u=\eta(R)|\nabla u|^2+R\eta'(R)(\nabla u\cdot\nu_R)^2.
	\]
	We obtain
	\begin{equation}\label{eq.radial-domain-variation}
		\begin{split}
			0={}&\int_\Omega
			\left(ne-2a|\nabla u|^2+b(x_n-x_n^0)\right)\eta(R)\,\dx{x}\\
			&+\int_\Omega R
			\left(e-2a(\nabla u\cdot\nu_R)^2\right)\eta'(R)\,\dx{x}.
		\end{split}
	\end{equation}
	We also note that the identity \eqref{eq.radial-domain-variation} holds also for compactly supported Lipschitz radial cutoff functions. Fix \(0<r<\delta\) and choose \(h>0\) so that \(r+h<\delta\). Define
	\[
		\eta_h(R)=
		\begin{cases}
			1,&0\le R\le r,\\
			(r+h-R)/h,&r<R<r+h,\\
			0,&R\ge r+h.
		\end{cases}
	\]
	Note that the first integral in \eqref{eq.radial-domain-variation} converges to the corresponding integral over \(B_r(x^0)\). Moreover, the second integral in \eqref{eq.radial-domain-variation} becomes
	\[
		-\frac1h\int_r^{r+h}s
		\int_{\partial B_s(x^0)}
		\left(e-2a(\nabla u\cdot\nu)^2\right)\,\dH{n-1}\,\dx{s}.
	\]
	For a.e.\ \(r\), the one-dimensional Lebesgue differentiation theorem identifies the limit as
	\[
		-r\int_{\partial B_r(x^0)}
		\left(e-2a(\nabla u\cdot\nu)^2\right)\,\dH{n-1}.
	\]
	Rearrangement gives \eqref{eq.pohozaev-identity}.

	\emph{(ii) The energy identity.}
	Take \(\zeta\in C_c^1(\Omega)\) and \(\varepsilon>0\). By continuity of \(u\), the support of \(\zeta(u-\varepsilon)_+\) is a compact subset of \(\{u>0\}\). It is therefore an admissible Sobolev test function in \eqref{eq.quasilinear-equation}. We obtain
	\[
		0=\int_{\{u>\varepsilon\}}
		a\zeta|\nabla u|^2\,\dx{x}
		+\int_\Omega a(u-\varepsilon)_+\nabla u\cdot\nabla\zeta\,\dx{x}.
	\]
	Passing to the limit as \(\varepsilon\to 0^+\) and using the fact that \(\nabla u=0\) a.e.\ on \(\{u=0\}\), we have
	\begin{equation}\label{eq.energy-test-identity}
		\int_\Omega a|\nabla u|^2\zeta\,\dx{x}=-\int_\Omega au\nabla u\cdot\nabla\zeta\,\dx{x}.
	\end{equation}
	This implies that
	\[
		\operatorname{div}(au\nabla u)=a|\nabla u|^2
		\qquad\text{in the sense of distributions in }\Omega.
	\]
	We now introduce the Lipschitz cutoff function \(\zeta(x)=\eta_h(|x-x^0|)\) into \eqref{eq.energy-test-identity}. Then the right-hand side becomes
	\[
		\frac1h\int_r^{r+h}
		\int_{\partial B_s(x^0)}au\nabla u\cdot\nu\,\dH{n-1}\,\dx{s}.
	\]
	The left-hand side converges to \(\int_{B_r(x^0)}a|\nabla u|^2\,\dx{x}\). Then for a.e.\ \(r\), the equation \eqref{eq.energy-identity} follows from the coarea formula and the Lebesgue differentiation theorem. Substituting \eqref{eq.energy-identity} into \eqref{eq.pohozaev-identity} gives \eqref{eq.pohozaev-energy-identity}, for a.e. \(r\in(0,\delta)\). This completes the proof.
\end{proof}

\begin{remark}
	Since \(\nabla u=0\) a.e. outside \(\{u>0\}\), equation \eqref{eq.signed-height-coefficient-identity} implies
	\begin{equation*}
		\partial_2F(|\nabla u|^2;x_n)+\lambda'(x_n)\chi_{\{u>0\}}=\frac{H(|\nabla u|^2;x_n)}{H_0^2}\chi_{\{u>0\}} \qquad\text{a.e. in }\Omega.
	\end{equation*}
	Consequently, the final volume term in both
	\eqref{eq.pohozaev-identity} and \eqref{eq.pohozaev-energy-identity} can be
	written as
	\[
		-\frac1{H_0^2}\int_{B_r(x^0)\cap \{u>0\}}
		H(|\nabla u|^2;x_n)(x_n-x_n^0)\,\dx{x}.
	\]
	In particular, this term is nonpositive when \(x_n^0=0\).
\end{remark}

\subsection{A Weiss-type monotonicity formula}

We derive a Weiss-type identity at centers on the stagnation hyperplane \(\Pi\). The boundary adjustment uses the \(3/2\)-homogeneous scaling, inspired by the incompressible problem in \cite{MR1759450,MR2810856}. The dependence of \(H\) on the gradient and on the height gives rise to four additional terms, arising from the quasilinear nature of the problem.

\begin{proposition}[Weiss-type monotonicity formula]\label{prop:wsm}
	Let \(u\) be a subsonic variational solution of \eqref{eq.quaslinear-free-boundary-problem2}. Fix \(x^0\in\Omega\cap\Pi\), and put \(\delta=\operatorname{dist}(x^0,\partial\Omega)/2\) as in \eqref{d0}. For \(0<r<\delta\), define
	\begin{equation}\label{Ur}
		U(x^0,u;r):=r^{-n-1}J_F(u;B_r(x^0)),
	\end{equation}
	\begin{equation}\label{Wr}
		W(x^0,u;r):=\frac{r^{-n-2}}{H_0}
		\int_{\partial B_r(x^0)}u^2\,\dH{n-1},
	\end{equation}
	and
	\begin{equation}\label{Phir}
		\Phi(x^0,u;r):=U(x^0,u;r)-\frac32W(x^0,u;r).
	\end{equation}
	The functions \(U\), \(W\), and \(\Phi\) defined in \eqref{Ur}, \eqref{Wr}, and \eqref{Phir} are locally absolutely continuous on \((0,\delta)\). For every \(0<\sigma_1<\sigma_2<\delta\),
	\begin{equation}\label{Phir1}
		\begin{split}
			&\Phi(x^0,u;\sigma_2)-\Phi(x^0,u;\sigma_1)\\
			&\quad=2\int_{\sigma_1}^{\sigma_2}r^{-n-1}
			\int_{\partial B_r(x^0)}
			\frac{\left(\nabla u\cdot\nu-\frac32\frac ur\right)^2}
			{H(|\nabla u|^2;x_n)}\,\dH{n-1}\,\dx{r}\\
			&\qquad\quad+
			\sum_{i=1}^4\int_{\sigma_1}^{\sigma_2}
			r^{-n-2}K_i(x^0,u;r)\,\dx{r},
		\end{split}
	\end{equation}
	where \(\nu=(x-x^0)/r\) on \(\partial B_r(x^0)\), and
	\begin{equation}\label{K1r}
		\begin{split}
			K_1(x^0,u;r):=\int_{B_r(x^0)}\Bigg[
			&\frac{|\nabla u|^2}{H(|\nabla u|^2;x_n)}
			-F(|\nabla u|^2;x_n)\\
			&+\left(\frac{x_n}{H_0}-\lambda(x_n)\right)
			\chi_{\{u>0\}}\Bigg]\,\dx{x},
		\end{split}
	\end{equation}
	\begin{equation}\label{K2r}
		\begin{split}
			K_2(x^0,u;r):={}&\int_{B_r(x^0)}
			\Bigl[\partial_2F(|\nabla u|^2;x_n)\\
			&+\Bigl(\lambda'(x_n)-\frac1{H_0}\Bigr)\chi_{\{u>0\}}\Bigr]
			x_n\,\dx{x},
		\end{split}
	\end{equation}
	\begin{equation}\label{K3r}
		K_3(x^0,u;r):=3\int_{\partial B_r(x^0)}
		\left(\frac1{H(|\nabla u|^2;x_n)}-\frac1{H_0}\right)
		u\nabla u\cdot\nu\,\dH{n-1},
	\end{equation}
	and
	\begin{equation}\label{K4r}
		K_4(x^0,u;r):=\frac9{2r}\int_{\partial B_r(x^0)}
		\left(\frac1{H_0}-\frac1{H(|\nabla u|^2;x_n)}\right)
		u^2\,\dH{n-1}.
	\end{equation}
\end{proposition}

\begin{proof}
	Fix \(x^0\), and omit \(x^0\) and \(u\) from the notation for \(U\), \(W\), \(\Phi\), and \(K_i\). Set
	\begin{equation}\label{eq.notations}
		\begin{aligned}
			a(x)&:=\frac1{H(|\nabla u|^2;x_n)},\qquad a_0:=\frac1{H_0},\\
			e(x)&:=F(|\nabla u|^2;x_n)+\lambda(x_n)\chi_{\{u>0\}},\\
			b(x)&:=\partial_2F(|\nabla u|^2;x_n)+\lambda'(x_n)\chi_{\{u>0\}}.
		\end{aligned}
	\end{equation}
	By \eqref{eq.subsonic-admissibility}, on every compact \(K\subset\subset\Omega\) the state \((|\nabla u|^2,x_n)\) lies in a compact subset of the fixed rectangle \eqref{eq.cR}. It follows from Lemma \ref{lem.preliminary} that \(H\) is \(C^2\) and bounded there and \(\nabla u\) is locally bounded. In particular, the continuous representative of \(u\) is locally Lipschitz. The coarea formula therefore shows that \(r\mapsto U(x^0,u;r)\) is locally absolutely continuous, and radial parametrization of the sphere shows the same for \(r\mapsto W(x^0,u;r)\). 

	Since \(x_n^0=0\), \eqref{eq.pohozaev-energy-identity} yields
	\[
		\begin{split}
			&nJ_F(u;B_r(x^0))-rJ_F(u;\partial B_r(x^0))\\
			&\quad=2\int_{\partial B_r(x^0)}au\nabla u\cdot\nu\,\dH{n-1}
			-2r\int_{\partial B_r(x^0)}a(\nabla u\cdot\nu)^2\,\dH{n-1}\\
			&\qquad\quad-\int_{B_r(x^0)}b x_n\,\dx{x}.
		\end{split}
	\]
	By \eqref{K1r} and \eqref{K2r},
	\[
		K_1(r)+K_2(r)=\int_{B_r(x^0)}a|\nabla u|^2\,\dx{x}-J_F(u;B_r(x^0))+\int_{B_r(x^0)}b x_n\,\dx{x}.
	\]
	Combining these equalities with \eqref{eq.energy-identity}, we obtain
	\begin{equation}\label{eq.Weiss-energy-balance}
		\begin{split}
			&(n+1)J_F(u;B_r(x^0))-rJ_F(u;\partial B_r(x^0))\\
			&\quad=3\int_{\partial B_r(x^0)}au\nabla u\cdot\nu\,\dH{n-1}
			-2r\int_{\partial B_r(x^0)}a(\nabla u\cdot\nu)^2\,\dH{n-1}\\
			&\quad-K_1(r)-K_2(r).
		\end{split}
	\end{equation}
	Differentiating \eqref{Ur} and using \eqref{eq.Weiss-energy-balance} gives
	\begin{equation*}
		\begin{split}
			U'(r)={}&2r^{-n-1}\int_{\partial B_r(x^0)}
			a(\nabla u\cdot\nu)^2\,\dH{n-1}\\
			&-3r^{-n-2}\int_{\partial B_r(x^0)}
			au\nabla u\cdot\nu\,\dH{n-1}
			+r^{-n-2}\bigl(K_1(r)+K_2(r)\bigr).
		\end{split}
	\end{equation*}
	On the other hand,
	\[
		\begin{aligned}
			&\frac{\dx{}}{\dx{r}}\left( \int_{\partial B_r(x^0)}u^2\,\dH{n-1} \right)\\
			&=\frac{n-1}{r}\int_{\partial B_r(x^0)}u^2\,\dH{n-1}+2\int_{\partial B_r(x^0)}u\nabla u\cdot\nu\,\dH{n-1}.
		\end{aligned}
	\]
	Consequently,
	\begin{equation*}
		W'(r)=\frac{r^{-n-3}}{H_0}
		\left[2r\int_{\partial B_r(x^0)}u\nabla u\cdot\nu\,\dH{n-1}
		-3\int_{\partial B_r(x^0)}u^2\,\dH{n-1}\right].
	\end{equation*}
	Subtracting \(3W'(r)/2\) from \(U'(r)\), we find
	\[
		\begin{split}
			\Phi'(r)={}&2r^{-n-1}\int_{\partial B_r(x^0)}
			a(\nabla u\cdot\nu)^2\,\dH{n-1}\\
			&-3r^{-n-2}\int_{\partial B_r(x^0)}
			(a+a_0)u\nabla u\cdot\nu\,\dH{n-1}\\
			&+\frac92a_0r^{-n-3}\int_{\partial B_r(x^0)}u^2\,\dH{n-1}
			+r^{-n-2}\bigl(K_1(r)+K_2(r)\bigr).
		\end{split}
	\]
	It then follows from the definition of \(K_3\) and \(K_4\) that
	\begin{equation}\label{Phi1r}
		\Phi'(r)=2r^{-n-1}\int_{\partial B_r(x^0)}
		\frac{\left(\nabla u\cdot\nu-\frac32\frac ur\right)^2}
		{H(|\nabla u|^2;x_n)}\,\dH{n-1}
		+r^{-n-2}\sum_{i=1}^4K_i(r).
	\end{equation}
	Note that each term is integrable on compact subintervals of \((0,\delta)\). Since \(\Phi\) is locally absolutely continuous, integration of \eqref{Phi1r} gives \eqref{Phir1} for every \(0<\sigma_1<\sigma_2<\delta\).
\end{proof}

\begin{remark}
	Since \(\nabla u=0\) a.e. outside \(\{u>0\}\), the volume terms can be written entirely over \(\{u>0\}\). Integration by parts in the first variable, together with \(\lambda(s)=2s/H_0-F(s;s)\), gives
	\begin{equation}\label{K11}
		K_1(x^0,u;r)=\int_{B_r(x^0)\cap\{u>0\}}
		\int_{x_n}^{|\nabla u|^2}
		\tau\,\partial_1\!\left(\frac1H\right)(\tau;x_n)\,\dx{\tau}\,\dx{x}.
	\end{equation}
	Differentiating \(\lambda(s)=2s/H_0-F(s;s)\) with respect to \(s\) yields
	\[
		\lambda'(s)=\frac1{H_0}-\partial_2F(s;s).
	\]
	Thus
	\begin{equation}\label{K22}
		\begin{split}
			K_2(x^0,u;r)&=\int_{B_r(x^0)\cap\{u>0\}}x_n\int_{x_n}^{|\nabla u|^2}\partial_2\!\left(\frac1H\right)(\tau;x_n)\,\dx{\tau}\,\dx{x}\\
			&=\frac1{H_0^2}\int_{B_r(x^0)\cap\{u>0\}}x_n\left(H(|\nabla u|^2;x_n)-H_0\right)\,\dx{x},
		\end{split}
	\end{equation}
	where the second equality follows from \eqref{eq.signed-height-coefficient-identity}. In the first line, \(\partial_2\) differentiates the height variable, not \(\tau\). Note that all horizontal integration segments in \eqref{K11}--\eqref{K22} lie in the coefficient rectangle, since \(0<x_n<s_0<t_0\) in \(\{u>0\}\).
\end{remark}

\begin{remark}\label{rem.Weiss-boundary-errors-combined}
	The terms \(K_1,K_2\) occur in \(U'\). The terms \(K_3,K_4\) record the difference between the variable weight \(1/H\) and the constant weight \(1/H_0\) when completing the square in \(U'-3W'/2\). In particular, they satisfy
	\begin{equation}\label{eq.Weiss-boundary-errors-combined}
		\begin{split}
			&K_3(x^0,u;r)+K_4(x^0,u;r)\\
			&\quad=3\int_{\partial B_r(x^0)}
			\left(\frac1{H(|\nabla u|^2;x_n)}-\frac1{H_0}\right)
			u\left(\nabla u\cdot\nu-\frac32\frac ur\right)\,\dH{n-1}\\
			&\quad:=3\int_{\partial B_r(x^0)}
			(a-a_0)u\delta_r\,\dH{n-1},
		\end{split}
	\end{equation}
	where \(a(x)\) and \(a_0\) are defined in \eqref{eq.notations}, and \(\delta_r\) is defined by
	\begin{equation}\label{eq.delta_r}
		\delta_r=\nabla u\cdot\nu-\frac{3u}{2r}\quad\text{ on }\partial B_r(x^0).
	\end{equation}
\end{remark}

\begin{remark}\label{rem.Weiss-positive-volume-remainder}
	Fix \(t\ge0\), \(s\ge0\) in the coefficient rectangle, write \(H=H(t;s)\), and use \(\Lambda(t;s)=2t/H-F(t;s)\) for simplicity. A direct calculation gives
	\[
		\begin{split}
			\frac tH-F(t;s)-\lambda(s)+\frac{sH}{H_0^2}
			&=\Lambda(t;s)-\lambda(s)-\frac tH+\frac{sH}{H_0^2}\\
			&=\frac{p(H_0)-p(H)+Hh(H)}{gH_0^2}\\
			&=\frac1{gH_0^2}\int_{H_0}^{H}h(\rho)\,\dx{\rho}.
		\end{split}
	\]
	where the second equality uses \eqref{eq.Lambda-pressure-representation} and \eqref{eq.normalized-bernoulli}; the third identity follows from \(h(H_0)=0\) and \(h'(\rho)=p'(\rho)/\rho\). Since \(h\) is strictly increasing, we know that the last expression is nonnegative, whether \(H\ge H_0\) or \(H<H_0\). Thus it follows from 
	\eqref{K1r} and \eqref{K2r} that
	\begin{equation}\label{eq.Weiss-positive-volume-remainder}
		\begin{split}
			&K_1(x^0,u;r)+K_2(x^0,u;r)\\
			&\qquad=\frac1{gH_0^2}\int_{B_r(x^0)\cap\{u>0\}}\int_{H_0}^{H(|\nabla u|^2;x_n)}h(\rho)\,\dx{\rho}\,\dx{x}\ge0.
		\end{split}
	\end{equation}
	Moreover, \(h'\) is bounded above and below by positive constants, we have 
	\[
		C_1(H-H_0)^2\le\int_{H_0}^{H}h(\rho)\,\dx{\rho}\le C_2(H-H_0)^2.
	\]
	This together with \(H(s;s)=H_0\) and the two-sided bounds on \(-\partial_1H\) 
	imply that \(|H(t;s)-H_0|\) is comparable to \(|t-s|\), this yields the estimate
	\begin{equation*}
		\begin{aligned}
			&c\int_{B_r(x^0)\cap\{u>0\}}(|\nabla u|^2-x_n)^2\,\dx{x}\\
			&\le K_1(x^0,u;r)+K_2(x^0,u;r)\\
			&\le C\int_{B_r(x^0)\cap\{u>0\}}(|\nabla u|^2-x_n)^2\,\dx{x}.
		\end{aligned}
	\end{equation*}
\end{remark}

\begin{remark}
	For \(R>0\), set
	\begin{equation*}
		u_R(x):=\frac{u(x^0+Rx)}{R^{3/2}},
		\qquad x\in\Omega_R:=\frac{\Omega-x^0}{R},
	\end{equation*}
	where \(x_n^0=0\), and define
	\begin{equation*}
		\begin{split}
			H_R(t;s)&:=H(Rt;Rs),\\
			F_R(t;s)&:=\frac1R F(Rt;Rs)
			=\int_0^t\frac1{H_R(\tau;s)}\,\dx{\tau},\\
			\lambda_R(s)&:=\frac1R\lambda(Rs).
		\end{split}
	\end{equation*}
	These coefficients are defined for \(0\le t<t_0/R\) and \(|s|<s_0/R\). Define \(J_{F_R}\) by \eqref{energy} with \((F_R,\lambda_R)\) in place of \((F,\lambda)\). The relations \(\nabla u(x^0+Rx)=R^{1/2}\nabla u_R(x)\) and \(x_n^0=0\) give
	\begin{equation*}
		r^{-n-1}J_F(u;B_r(x^0))=(r/R)^{-n-1}J_{F_R}(u_R;B_{r/R}).
	\end{equation*}
	Rescaling the interior equation and \eqref{eq.domain-variation-formula} shows that \(u_R\) satisfies the corresponding equation and stationarity identity with these rescaled coefficients.
\end{remark}

\begin{remark}
	As \(R\to0^+\),
	\[
		H_R(t;s)\longrightarrow H_0,\qquad
		F_R(t;s)\longrightarrow\frac t{H_0},\qquad
		\lambda_R(s)\longrightarrow\frac s{H_0},
	\]
	locally uniformly on bounded sets. Thus the limiting functional is
	\begin{equation*}
		\frac1{H_0}\int_A
		\left(|\nabla v|^2+x_n\chi_{\{v>0\}}\right)\,\dx{x},
		\qquad A\subset\subset\R^n.
	\end{equation*}
	We remark that whether a blow-up limit is again a subsonic variational solution of this functional requires convergence of the phases. More precisely, for a vanishing sequence \(R_j\to0^+\) as \(j\to\infty\), assume
	\begin{gather*}
		u_{R_j}\longrightarrow u_0
		\quad\text{strongly in }W^{1,2}_{\mathrm{loc}}(\R^n),\\
		\chi_{\{u_{R_j}>0\}}\longrightarrow\chi_{\{u_0>0\}}
		\quad\text{in }L^1_{\mathrm{loc}}(\R^n).
	\end{gather*}
	Then we obtain
	\[
		J_{F_{R_j}}(u_{R_j};A)\longrightarrow
		\frac1{H_0}\int_A
		\left(|\nabla u_0|^2+x_n\chi_{\{u_0>0\}}\right)\,\dx{x},
	\]
	for every bounded open \(A\subset\subset\R^n\). The domain-variation identity also passes to
	\[
		\int_{\R^n}\left[
		\left(|\nabla u_0|^2+x_n\chi_{\{u_0>0\}}\right)
		\operatorname{div}\phi
		-2\nabla u_0\cdot D\phi\,\nabla u_0
		+\chi_{\{u_0>0\}}\phi_n\right]\,\dx{x}=0
	\]
	for every \(\phi\in C_c^1(\R^n;\R^n)\). Without identifying the phase limit, the same calculation yields a weak-\(*\) limit \(\chi_0\) of \(\chi_{\{u_{R_j}>0\}}\), with \(0\le\chi_0\le1\), in place of \(\chi_{\{u_0>0\}}\).
\end{remark}

\begin{remark}
	In the incompressible case \cite{MR2810856}, all four remainders in \eqref{Phir1} vanish. In the quasilinear problem, \(K_1+K_2\ge0\) by \eqref{eq.Weiss-positive-volume-remainder}, while \eqref{eq.Weiss-boundary-errors-combined} gives no sign for \(K_3+K_4\). Thus \eqref{Phir1} alone does not establish monotonicity of \(\Phi\). However, for any fixed \(\rho\in(0,\delta)\), if we define
	\[
		\Phi(x^0,u;r)
		-\int_\rho^r s^{-n-2}\sum_{i=1}^4K_i(x^0,u;s)\,\dx{s}
	\]
	Then it follows from Proposition \ref{prop:wsm} that for a.e. \(r\in(0,\delta)\),
	\[
		\frac{\dx}{\dx{r}}\left(  \Phi(x^0,u;r)
		-\int_\rho^r s^{-n-2}\sum_{i=1}^4K_i(x^0,u;s)\,\dx{s}
		\right)\ge 0.
	\]
\end{remark}

\begin{remark}
	For the incompressible model \(H\equiv H_0\), one has
	\[
		F(t;s)=\frac t{H_0},\qquad\lambda(s)=\frac s{H_0},
	\]
	and all four terms \(K_i\) vanish. Thus \eqref{Phir1} reduces to the boundary case in \cite[Theorem 3.5(ii)]{MR2810856}, up to the constant factor \(1/H_0\). In this case \(\Phi\) is nondecreasing without any correction term, and its constancy on an interval of radii is equivalent to \(3/2\)-homogeneity on the corresponding annulus.
\end{remark}

\section{The blow-up analysis}

With the aid of monotonicity formula \eqref{Phir1}, we are able to conduct a blow-up analysis at stagnation points on the hyperplane \(\Pi\). 

\begin{lemma}\label{lem.blow-up limits}
	Let \(u\) be a subsonic variational solution of \eqref{eq.quaslinear-free-boundary-problem2}, let \(x^0\in\Omega\cap\Pi\) be a stagnation point. Assume that there exists a constant \(C_0>0\) such that 
	\begin{equation}\label{eq.fd-growth}
		|\nabla u|^2\le C_0x_n^+\quad\text{ locally a.e. in }\Omega.
	\end{equation}
	Then the following statements hold.
	\begin{enumerate}[label=(\roman*)]
		\item The function 
		\begin{equation}\label{eq.bu-quadratic-monotonicity}
			r\longmapsto\Phi(x^0,u;r)+Cr^2 \quad\text{is nondecreasing for }0<r<r_0
		\end{equation}
		for some uniform constant \(C>0\). Moreover, the limit \(\Phi(x^0,u;0^+)\) exists and is finite.

		\item Given any sequence \(r_m\to 0^+\), the functions
		\begin{equation}\label{eq.bls}
			u_m(x):=\frac{u(x^0+r_mx)}{r_m^{3/2}},
			\qquad x\in\Omega_m:=\frac{\Omega-x^0}{r_m},
		\end{equation}
		have a subsequence converging locally uniformly and strongly in \(W^{1,2}_{\mathrm{loc}}(\R^n)\) to a function \(u_0\).

		\item Every limit function \(u_0\) obtained in \textup{(ii)} is locally Lipschitz, nonnegative, zero on \(\{x_n\le0\}\), harmonic in \(\{u_0>0\}\), and satisfies
		\begin{equation}\label{eq.bu-limit-growth-homogeneity}
			|\nabla u_0|^2\le C_0x_n^+\quad\text{a.e.},\qquad
			u_0(\lambda x)=\lambda^{3/2}u_0(x)\quad(\lambda>0).
		\end{equation}
		In addition, \(\Delta u_0\) is a nonnegative Radon measure. If a sequence in \eqref{eq.bls} converges weakly in \(W^{1,2}_{\mathrm{loc}}\), its convergence to that limit is locally uniform and strong in \(W^{1,2}_{\mathrm{loc}}(\R^n)\).
	\end{enumerate}
\end{lemma}

\begin{proof}
	After a horizontal translation, we may assume \(x^0=0\), and we work in a ball centered on \(\Pi\) in which the growth condition \eqref{eq.fd-growth} holds. Absolute continuity of \(u\) along almost every vertical line, together with the continuity of \(u\), gives
	\begin{equation}\label{eq.bu-pointwise-growth}
		0\le u(x)\le\int_0^{x_n^+}\sqrt{C_0s}\,\dx{s}
		=\frac23\sqrt{C_0}(x_n^+)^{3/2}.
	\end{equation}
	We divide the proof into four steps.

	\emph{Step I: estimates of additional terms in monotonicity formula.}
	On the fixed coefficient rectangle the derivatives \(\partial_1(1/H)\) and \(\partial_s(1/H)\) are bounded, and we have
	\[
		\begin{split}
			K_1(r)&=\int_{B_r\cap\{u>0\}}\int_{x_n}^{|\nabla u|^2}
			\tau\,\partial_1(1/H)(\tau;x_n)\,\dx{\tau}\,\dx{x},\\
			K_2(r)&=\int_{B_r\cap\{u>0\}}x_n\int_{x_n}^{|\nabla u|^2}
			\left.\partial_s(1/H)(\tau;s)\right|_{s=x_n}\,\dx{\tau}\,\dx{x}.
		\end{split}
	\]
	Here \(\partial_s\) differentiates the height variable with \(\tau\) held fixed.
	Consequently, we deduce from \eqref{eq.fd-growth} that
	\[
		|K_1(r)|+|K_2(r)|\le Cr^{n+2}.
	\]
	Since \(H(x_n;x_n)=H_0\) in the positive phase, it follows from 
	the mean value theorem
	and \eqref{eq.fd-growth} that
	\begin{equation*}
		|a-a_0|\le C\bigl||\nabla u|^2-x_n\bigr|\le Cx_n
		\quad\text{a.e. in }B_r\cap\{u>0\},
	\end{equation*}
	where \(a(x)\) and \(a_0\) are as in \eqref{eq.notations}. On the zero set, the factors \(u\) and \(\nabla u\) in the boundary errors vanish a.e., so \eqref{K3r} and \eqref{K4r} yield
	\[
		\begin{split}
			|K_3(r)|&\le Cr\int_{\partial B_r}|u||\nabla u|\,\dH{n-1}
			\le Cr^{n+2},\\
			|K_4(r)|&\le C\int_{\partial B_r}u^2\,\dH{n-1}\le Cr^{n+2}.
		\end{split}
	\]
	In particular,
	\begin{equation}\label{bound}
		r^{-n-2}\sum_{i=1}^4|K_i(r)|\le C\quad\text{for a.e. }r\in(0,r_0).
	\end{equation}
	Together with \(F(t;s)\le Ct\) and \(0\le\lambda(s)\le Cs\) for
	\(s\ge0\), \eqref{eq.bu-pointwise-growth} gives \(|\Phi(r)|\le C\).

	For the sharper correction we recall \eqref{eq.Weiss-positive-volume-remainder} and \eqref{eq.Weiss-boundary-errors-combined} (see Remarks~\ref{rem.Weiss-positive-volume-remainder} and~\ref{rem.Weiss-boundary-errors-combined}). In the notation of \eqref{eq.notations} and \eqref{eq.delta_r}, we obtain
	\[
		K_1(r)+K_2(r)\ge0,\qquad K_3(r)+K_4(r)=3\int_{\partial B_r}(a-a_0)u\delta_r\,\dH{n-1}.
	\]
	Young's inequality in \eqref{Phi1r} yields
	\begin{equation*}
		\begin{split}
			\Phi'(r)
			&\ge r^{-n-1}\int_{\partial B_r(x^0)}a\delta_r^2\,\dH{n-1}
			+r^{-n-2}\bigl(K_1(r)+K_2(r)\bigr)\\
			&\qquad-\frac94r^{-n-3}
			\int_{\partial B_r(x^0)}\frac{(a-a_0)^2}{a}u^2\,\dH{n-1}\\
			&\ge r^{-n-1}\int_{\partial B_r(x^0)}a\delta_r^2\,\dH{n-1}-Cr
		\end{split}
	\end{equation*}
	Indeed, the surface integral in the error term is bounded by
	\(Cr^{n+4}\). This proves \eqref{eq.bu-quadratic-monotonicity};
	boundedness then gives the existence and finiteness of the density.

	\emph{Step II: compactness.}
	Using the growth condition \eqref{eq.fd-growth}, we deduce that 
	\[
		|\nabla u_m(x)|^2\le C_0x_n^+,\qquad 0\le u_m(x)\le\frac23\sqrt{C_0}(x_n^+)^{3/2} \quad\text{on }B_R,
	\]
	for every fixed \(R>0\) and all large \(m\).
	It then follows from the Arzel\`a--Ascoli theorem that a subsequence converges locally uniformly and weakly in \(W^{1,2}_{\mathrm{loc}}(\R^n)\) to some blow-up limit \(u_0\). The local uniform convergence implies that \(u_0\) is nonnegative and vanishes on \(\{x_n\le0\}\). The weak \(W^{1,2}_{\mathrm{loc}}\) convergence also gives \(|\nabla u_0|^2\le C_0x_n^+\) a.e. in \(\R^n\). Now we define
	\begin{equation}\label{eq.bu-rescaled-coefficients}
		\begin{aligned}
			a_m(x)&:=\frac1{H(r_m|\nabla u_m(x)|^2;r_mx_n)},\\
			F_m(t;s)&:=r_m^{-1}F(r_mt;r_ms),&
			\lambda_m(s)&:=r_m^{-1}\lambda(r_ms).
		\end{aligned}
	\end{equation}
	Note that the arguments of \(H\) have size at most \(C_Rr_m\) on \(B_R\),
	therefore we have
	\begin{equation}\label{Hst}
		\|a_m-a_0\|_{L^\infty(B_R)}\le C_Rr_m\to0.
	\end{equation}

	\emph{Step III: the limiting equation and strong convergence.}
	Let \(\eta\in C_c^\infty(\{u_0>0\})\). Since \(u_m\) converges locally uniformly to \(u_0\), we have \(\operatorname{supp}\eta\subset\{u_m>0\}\) for \(m\) large enough. Consequently,
	\[
		\int a_m\nabla u_m\cdot\nabla\eta\,\dx{x}=0
		\longrightarrow a_0\int\nabla u_0\cdot\nabla\eta\,\dx{x}.
	\]
	This implies that \(u_0\) is harmonic in \(\{u_0>0\}\).

	For arbitrary \(\zeta\in C_c^1(\R^n)\), it follows from \eqref{eq.energy-test-identity} that
	\begin{equation}\label{eq.bu-rescaled-energy}
		\int a_m|\nabla u_m|^2\zeta\,\dx{x}
		=-\int a_mu_m\nabla u_m\cdot\nabla\zeta\,\dx{x}.
	\end{equation}
	Using the fact that \(u_0\ge 0\) is locally Lipschitz and harmonic in \(\{u_0>0\}\), 
	we can test this equation with \(\zeta(u_0-\varepsilon)_+\) and then passing to the limit as \(\varepsilon\to0^+\), which yields
	\begin{equation}\label{eq.bu-limit-energy}
		\int |\nabla u_0|^2\zeta\,\dx{x}
		=-\int u_0\nabla u_0\cdot\nabla\zeta\,\dx{x},
	\end{equation}
	where we use \(\nabla u_0=0\) a.e. on \(\{u_0=0\}\). By \eqref{Hst}, local uniform convergence, and weak \(L^2\)-gradient convergence, the right side of \eqref{eq.bu-rescaled-energy} tends to the right side of \eqref{eq.bu-limit-energy} multiplied by \(a_0\). Hence, for every nonnegative \(\zeta\in C_c^1(\R^n)\),
	\[
		\int \zeta|\nabla u_m|^2\,\dx{x}
		\longrightarrow\int \zeta|\nabla u_0|^2\,\dx{x}.
	\]
	It then follows from the weak convergence of \(\nabla u_m\) to \(\nabla u_0\) that
	\(\nabla u_m\) converges to \(\nabla u_0\) strongly on every compact subset; strong convergence of the functions themselves follows from local uniform convergence.

	The original Radon measure is nonnegative, so
	\[
		\operatorname{div}(a_m\nabla u_m)\ge0 \quad\text{in }\mathcal D'(\Omega_m).
	\]
	Passing to the limit against nonnegative smooth test functions gives \(a_0\Delta u_0\ge0\), and a positive distribution is a Radon measure.

	\emph{Step IV: homogeneity of \(u_0\).}
	Integrating the equation \eqref{Phi1r} over \((r_m\sigma_1,r_m\sigma_2)\), with \(0<\sigma_1<\sigma_2\), we have
	\begin{equation*}
		\begin{split}
			&2\int_{B_{\sigma_2}\setminus B_{\sigma_1}}
			|x|^{-n-3}a_m(x)
			\left(x\cdot\nabla u_m-\frac32u_m\right)^2\,\dx{x}\\
			&\quad=\Phi(r_m\sigma_2)-\Phi(r_m\sigma_1)
			-\int_{r_m\sigma_1}^{r_m\sigma_2}
			t^{-n-2}\sum_{i=1}^4K_i(t)\,\dx{t}.
		\end{split}
	\end{equation*}
	It follows from Step I and \eqref{bound} that the right hand side tends to zero as \(m\to\infty\). The uniform ellipticity and weak lower semicontinuity also imply
	\[
		x\cdot\nabla u_0=\frac32u_0
		\quad\text{ a.e. on }B_{\sigma_2}\setminus B_{\sigma_1}.
	\]
	This yields the desired \(3/2\)-homogeneity, and
	local continuity gives \eqref{eq.bu-limit-growth-homogeneity}. Finally, the same argument applies to a sequence as in \eqref{eq.bls} that converges weakly in \(W^{1,2}_{\mathrm{loc}}\), which gives the strong \(W^{1,2}_{\mathrm{loc}}\) convergence.  This completes the proof.
\end{proof}

\begin{remark}
	It follows from \eqref{bound} that the additional term
	\[
		t^{-n-2}K_i(x^0,u;t),\qquad i=1,\ldots,4,
	\]
	is integrable at zero. Consequently, if we define
	\begin{equation*}
		\tilde\Phi(x^0,u;r):=\Phi(x^0,u;r)-\sum_{i=1}^4\int_0^r t^{-n-2}K_i(x^0,u;t)\,\dx{t},
	\end{equation*}
	then it is easy to see that 
	\begin{equation*}
		\tilde\Phi'(r)=2r^{-n-1}\int_{\partial B_r(x^0)}a\delta_r^2\,\dH{n-1}\ge0,\quad\tilde\Phi(x^0,u;0^+)=\Phi(x^0,u;0^+).
	\end{equation*}
\end{remark}

\subsection{Weighted densities}

The positivity indicators need not converge to \(\chi_{\{u_0>0\}}\), even when the functions converge strongly in \(W^{1,2}_{\mathrm{loc}}\). Until the phase has been identified, we keep \(\chi_0\) as an independent limiting quantity.

\begin{lemma}\label{Lemma: densities at the stagnation points}
	Let \(u\) be a subsonic variational solution of \eqref{eq.quaslinear-free-boundary-problem2}, let \(x^0\in\Omega\cap\Pi\) be a stagnation point. Assume the growth condition \eqref{eq.fd-growth}. Then the
	following properties hold.

	\begin{enumerate}[label=(\roman*)]
		\item At every \(x^0\in\Omega\cap\Pi\),
		\begin{equation}\label{eq.bu-weighted-density}
			\Phi(x^0,u;0^+)=\frac1{H_0}\lim_{r\to0^+}r^{-n-1}
			\int_{B_r(x^0)}x_n^+\chi_{\{u>0\}}\dx{x}.
		\end{equation}
		In particular,
		\begin{equation}\label{eq.bu-density-range}
			0\le\Phi(x^0,u;0^+)\le\Theta_0,
			\qquad
			\Theta_0:=\frac1{H_0}\int_{B_1}x_n^+\dx{x}
			=\frac{\omega_{n-1}}{(n+1)H_0},
		\end{equation}
		where \(\omega_{n-1}\) is the volume of the unit ball in \(\R^{n-1}\). If \(\Phi(x^0,u;0^+)=0\), then every natural blow-up at \(x^0\) is zero.

	\item The function \(x^0\mapsto\Phi(x^0,u;0^+)\) is upper semicontinuous on
	\(\Omega\cap\Pi\). Its full weighted density level set is
	relatively closed in \(\Omega\) and is contained in \(S_u\).

	\item Let \(u_m\) be a convergent natural blow-up sequence at \(x^0\),
	with the rescaled coefficients in \eqref{eq.bu-rescaled-coefficients}.
	After passing to a subsequence,
	\begin{equation}\label{eq.bu-phase-weak-limit}
		\chi_{\{u_m>0\}}\stackrel{*}{\rightharpoonup}\chi_0
		\quad\text{in }L^\infty_{\mathrm{loc}}(\R^n),
		\qquad 0\le\chi_0\le1,
	\end{equation}
	where \(\chi_0=0\) a.e. on \(\{x_n\le0\}\) and \(\chi_0=1\)
	a.e. on \(\{u_0>0\}\). The pair \((u_0,\chi_0)\) satisfies
	\begin{equation}\label{pu0}
		\int_{\R^n}
		\left[
		(|\nabla u_0|^2+x_n\chi_0)\operatorname{div}\phi
		-2\nabla u_0\cdot D\phi\,\nabla u_0
		+\chi_0\phi_n
		\right]\,\dx{x}=0
	\end{equation}
	for every \(\phi\in C_c^1(\R^n;\R^n)\). 
	\item For \(z\in\Pi\), the function
	\begin{equation*}
		\begin{aligned}
			\Phi_0(z;\rho)&:=\frac{\rho^{-n-1}}{H_0}
			\int_{B_\rho(z)}(|\nabla u_0|^2+x_n\chi_0)\dx{x}\\
			&-\frac32\frac{\rho^{-n-2}}{H_0}
			\int_{\partial B_\rho(z)}u_0^2\,\dH{n-1}.
		\end{aligned}
	\end{equation*}
	is nondecreasing in \(\rho\) and has a finite limit at zero. Moreover, for every \(\rho>0\), we have 
	\[
		\Phi_0(0;\rho)=\Phi(x^0,u;0^+),
	\]
	and
	\begin{equation}\label{eq.bu-moving-center-stability}
		\limsup_{m\to\infty}\Phi(x^0+r_mz_m,u;0^+)\le\Phi_0(z;0^+).
	\end{equation}
	for \(z_m\in\Pi\) and \(z_m\to z\).
	\end{enumerate}
\end{lemma}

\begin{proof}
The proof is divided into four steps.

\emph{Step I.}
We begin with the limiting phase. Since \(0\le\chi_{\{u_m>0\}}\le1\), the sequence is bounded in
\(L^\infty_{\mathrm{loc}}(\R^n)\). Thus there exists a further subsequence 
and a function \(\chi_0\in L^\infty_{\mathrm{loc}}(\R^n)\) such that 
\(0\le\chi_0\le1\) a.e. Besides, \(\chi_0\) satisfies \eqref{eq.bu-phase-weak-limit}.
For every \(m\), \(u_m\equiv0\) in \(\Omega_m\cap\{x_n\le0\}\), so \(\chi_{\{u_m>0\}}=0\) there. Every compact subset of \(\{x_n\le0\}\) lies in \(\Omega_m\) for all large \(m\), and weak-\(*\) convergence gives \(\chi_0=0\) a.e. in \(\{x_n\le0\}\).
Since \(u_m\to u_0\) locally uniformly, \(u_m>0\) on each compact subset of \(\{u_0>0\}\) for all sufficiently large \(m\), hence \(\chi_0=1\) a.e. in \(\{u_0>0\}\).

We now compute the stationary identity for \(u_m\) and \(\chi_m\).
Fix a test function \(\bar\phi\in C_c^1(\R^n;\R^n)\) and take
\(\phi(y)=r_m\bar\phi((y-x^0)/r_m)\) in
\eqref{eq.domain-variation-formula}. It follows from a simple computation that
\[
	\begin{aligned}
		D_y\phi(y)&=D_x\bar\phi(x),&
		\operatorname{div}_y\phi(y)&=\operatorname{div}_x\bar\phi(x),\\
		\phi_n(y)&=r_m\bar\phi_n(x),&
		\nabla u(y)&=r_m^{1/2}\nabla u_m(x).
	\end{aligned}
\]
Since \(x_n^0=0\), we infer from equations in
\eqref{eq.bu-rescaled-coefficients} that
\[
	\begin{aligned}
		F(r_mt;r_ms)&=r_mF_m(t;s),&
		\lambda(r_ms)&=r_m\lambda_m(s),\\
		\partial_1F(r_mt;r_ms)&=\partial_1F_m(t;s),&
		\lambda'(r_ms)&=\lambda_m'(s),\\
		\partial_2F(r_mt;r_ms)&=\partial_2F_m(t;s).&&
	\end{aligned}
\]
Dividing by \(r_m^{n+1}\) yields
\[
	\begin{aligned}
		0={}&\int_{\Omega_m}\bigl(F_m(|\nabla u_m|^2;x_n)
		       +\lambda_m(x_n)\chi_{\{u_m>0\}}\bigr)
		       \operatorname{div}\bar\phi\,\dx{x}\\
		&-2\int_{\Omega_m}\partial_1F_m(|\nabla u_m|^2;x_n)
		       \nabla u_m\cdot D\bar\phi\,\nabla u_m\,\dx{x}\\
		&+\int_{\Omega_m}\bigl(\partial_2F_m(|\nabla u_m|^2;x_n)
		       +\lambda_m'(x_n)\chi_{\{u_m>0\}}\bigr)
		       \bar\phi_n\,\dx{x},
	\end{aligned}
\]
where the integrals are confined to the support of \(\bar\phi\).
For bounded \(t\ge0\) and bounded signed \(s\), we have
\begin{equation}\label{eq.bu-coefficient-expansions}
	\begin{aligned}
		|F_m(t;s)-a_0t|&\le Cr_m(t+|s|)t,&
		|\partial_1F_m(t;s)-a_0|&\le Cr_m(t+|s|),\\
		|\partial_2F_m(t;s)|&\le Cr_mt,&
		|\lambda_m(s)-a_0s|&\le Cr_ms^2,\\
		|\lambda_m'(s)-a_0|&\le Cr_m|s|.&&
	\end{aligned}
\end{equation}
On the support of any fixed test field the
arguments \((|\nabla u_m|^2,x_n)\) are uniformly bounded, so strong
\(W^{1,2}\) convergence and \eqref{eq.bu-phase-weak-limit} yield the identity \eqref{pu0} as \(m\to\infty\).

\emph{Step II.}
For a.e. \(R>0\), it follows from \eqref{eq.bu-limit-energy} and the 
homogeneity of \(u_0\) that
\[
	\int_{B_R}|\nabla u_0|^2\,\dx{x}
	=\int_{\partial B_R}u_0\nabla u_0\cdot\nu\,\dH{n-1}
	=\frac{3}{2R}\int_{\partial B_R}u_0^2\,\dH{n-1},
\]
for every \(R>0\). Moreover, a rescaling yields that
\begin{equation}\label{Phirm}
	\begin{split}
		\Phi(x^0,u;r_mR)
		={}&R^{-n-1}\int_{B_R}
		\bigl(F_m(|\nabla u_m|^2;x_n)
		+\lambda_m(x_n)\chi_{\{u_m>0\}}\bigr)\,\dx{x}\\
		&-\frac32\frac{R^{-n-2}}{H_0}
		\int_{\partial B_R}u_m^2\,\dH{n-1}.
	\end{split}
\end{equation}
Combining \eqref{eq.bu-coefficient-expansions} with the convergence properties of \(u_m\) and \(\chi_{\{u_m>0\}}\), we have
\begin{equation}\label{eq.bu-pair-density}
	\Phi(x^0,u;0^+)=\Phi_0(0;R)
	=\frac{1}{H_0}R^{-n-1}\int_{B_R}x_n\chi_0\,\dx{x}.
\end{equation}
Fix \(R=1\), then \(\int_{B_1}x_n\chi_0\,\dx{x}=H_0\Phi(x^0,u;0^+)\) 
for every phase limit \(\chi_0\). Since every vanishing
sequence of radii has such a subsequence, setting \(R=1\)
proves the full limit in \eqref{eq.bu-weighted-density},
and the bounds \eqref{eq.bu-density-range} follow. If this density is
zero, then \(x_n\chi_0=0\) a.e. in \(B_1\); since \(\chi_0=1\) on
\(\{u_0>0\}\), it follows from the continuity that \(u_0=0\) on \(B_1\), and homogeneity
gives \(u_0\equiv0\).

\emph{Step III.}
Fix \(x^0\in\Omega\cap\Pi\) and \(\varepsilon>0\). Then we may choose a fixed and sufficiently small \(r>0\) such that
\[
	\Phi(x^0,u;r)+Cr^2
	\le\Phi(x^0,u;0^+)+\frac\varepsilon2.
\]
At this fixed radius the function \(y\mapsto\Phi(y,u;r)\) is continuous. Also note that
\[
	\int_{\partial B_r(y)}u^2\,\dH{n-1}
	=r^{n-1}\int_{\partial B_1}u^2(y+r\omega)\,\dH{n-1}(\omega)
\]
is continuous by the local uniform continuity of \(u\). Hence, for
\(y\in\Pi\) sufficiently close to \(x^0\), we have
\[
	\Phi(y,u;0^+)
	\le\Phi(y,u;r)+Cr^2
	\le\Phi(x^0,u;r)+Cr^2+\frac\varepsilon2
	\le\Phi(x^0,u;0^+)+\varepsilon.
\]
This proves upper semicontinuity, with the radius fixed before the
center varies. Since \(\Phi\le\Theta_0\), the full
weighted density level set is relatively closed. Every point of this
level set belongs to the set \(S_u\).

\emph{Step IV.}
Since \(u_0\) satisfies the equation \eqref{pu0}, we may apply a similar argument as in the proof of Proposition \ref{prop:wsm}, which gives
\begin{equation*}
	\frac{\dx{}}{\dx{\rho}}\Phi_0(z;\rho)
	=\frac{2}{H_0}\rho^{-n-1}
	\int_{\partial B_\rho(z)}
	\left(\nabla u_0\cdot\nu-\frac{3u_0}{2\rho}\right)^2\,\dH{n-1}\ge0.
\end{equation*}
It then follows from \eqref{eq.bu-limit-growth-homogeneity} and \(0\le\chi_0\le1\) that \(\Phi_0(z;\rho)\) is bounded near zero, which gives that the limit is finite.

Let \(\Phi_m(z_m,u_m;\rho)\) denote the right side of \eqref{Phirm} with
\(B_R\) replaced by \(B_\rho(z_m)\). Then we have
\[
	\Phi_m(z_m,u_m;\rho)
	=\Phi(x^0+r_mz_m,u;r_m\rho).
\]
For every fixed \(\rho>0\), we claim that 
\begin{equation}\label{eq.bu-fixed-radius-convergence}
	\Phi_m(z_m,u_m;\rho)\to\Phi_0(z;\rho).
\end{equation}
Indeed, the kinetic densities converge strongly in \(L^1\), the boundary
functions converge locally uniformly, and the phase integral converges
by weak-star convergence. It then follows from the uniform correction at the
original scale that
\[
	\Phi(x^0+r_mz_m,u;0^+)
	\le\Phi_m(z_m,u_m;\rho)+Cr_m^2\rho^2.
\]
Taking the limsup in \(m\) and then letting \(\rho\to0^+\) provides
\eqref{eq.bu-moving-center-stability}.
\end{proof}

\begin{remark}
	Let us remark that \((u_0,\chi_0)\) is not automatically a variational solution with phase \(\chi_{\{u_0>0\}}\). For example, the pair \((0,c\chi_{\{x_n>0\}})\), for any \(0\le c\le1\), satisfies the displayed stationary-pair identity with
	\[
		c\int_{\{x_n>0\}}\operatorname{div}(x_n\phi)\,\dx{x}=0.
	\]
\end{remark}

\subsection{Two-dimensional blow-up limits}

Let \(x^0\in\Omega\cap\Pi\) be a stagnation point, it follows from Lemma~\ref{Lemma: weak v.s variational} and the growth condition \eqref{eq.fd-growth} that
\begin{equation}\label{eq.bu-weighted-perimeter-hypothesis}
	\begin{gathered}
		\chi_{\{u>0\}}\in BV_{\mathrm{loc}}(\Omega\cap\{x_n>0\}),\\
		r^{1/2-n}\int_{B_r(x^0)\cap\{x_n>0\}}\sqrt{x_n}\,\dx{|D\chi_{\{u>0\}}|}\le C_1\qquad(0<r<r_0).
	\end{gathered}
\end{equation}

\begin{remark}
	Let us remark that the estimates in \eqref{eq.bu-weighted-perimeter-hypothesis} is a direct corollary of Lemma \ref{Lemma: weak v.s variational}.
	Here \(|D\chi_{\{u>0\}}|\) denotes the total variation Radon
	measure on \(\Omega\cap\{x_n>0\}\) in the second estimates of \eqref{eq.bu-weighted-perimeter-hypothesis}. 
	
	Note that the estimate in \eqref{eq.bu-weighted-perimeter-hypothesis} yields
	weighted perimeter control up to the stagnation plane \(\Pi\), but we cannot obtain
	the unweighted BV regularity across \(\Pi\).
	Indeed, let \(K\subset B_r(x^0)\cap\{x_n\ge\varepsilon\}\) be compact
	and let \(0<r<r_0\), then
	\[
		\begin{aligned}
			&|D\chi_{\{u>0\}}|(K)\\
			&\le\varepsilon^{-1/2}
			\int_{B_r(x^0)\cap\{x_n>0\}}\sqrt{x_n}\,
			\dx{|D\chi_{\{u>0\}}|}\\
			&\le C_1\varepsilon^{-1/2}r^{n-1/2}.
		\end{aligned}
	\]
	The distinction between weighted control and unweighted perimeter
	at the stagnation line \(\Pi\) also appears for incompressible waves
	with vorticity in~\cite[Remark 3.3(ii) and Theorem 4.5]{MR2995099}.

	A geometric example also shows why the weight cannot be discarded.
	The set
	\[
		A=\R\times\bigcup_{k\ge1}(2^{-2k-1},2^{-2k})
	\]
	is locally of finite perimeter in \(\{x_2>0\}\), but its
	infinitely many horizontal interfaces have infinite unweighted
	perimeter in every ball centered at the origin. Their heights
	are \(2^{-j}\), \(j\ge2\), and their lengths in \(B_r\)
	are at most \(2r\). Consequently,
	\[
		\int_{B_r\cap\{x_2>0\}}\sqrt{x_2}\,\dx{|D\chi_A|}
		\le2r\sum_{\substack{j\ge2\\2^{-j}<r}}2^{-j/2}
		\le Cr^{3/2}.
	\]
\end{remark}

\begin{lemma}
Let \(u\) be a subsonic variational solution of \eqref{eq.quaslinear-free-boundary-problem2}, let \(x^0\in\Omega\cap\Pi\) be a stagnation point.
Assume \eqref{eq.fd-growth}, then for every blow-up sequence in \eqref{eq.bls}, there is 
a subsequence so that
\begin{equation}\label{eq.bu-phase-strong-limit}
	\chi_{\{u_m>0\}}\to\chi_0
	\quad\text{in }L^1_{\mathrm{loc}}(\R^n),
	\qquad \chi_0\in\{0,1\}\quad\text{a.e.}
\end{equation}
The limiting pair \((u_0,\chi_0)\) has all the properties stated in 
Lemma~\ref{Lemma: densities at the stagnation points}.
\end{lemma}

\begin{proof}
Fix \(R>0\) and \(0<\varepsilon<R\). It follows from
\eqref{eq.bu-weighted-perimeter-hypothesis} that
\[
	\begin{aligned}
		\int_{B_R\cap\{x_n>0\}}\sqrt{x_n}\,\dx{|D\chi_{\{u_m>0\}}|}
		&=r_m^{1/2-n}
		\int_{B_{Rr_m}(x^0)\cap\{x_n>0\}}\sqrt{x_n}\,
		\dx{|D\chi_{\{u>0\}}|}\\
		&\le C_1R^{n-1/2}.
	\end{aligned}
\]
This implies
\[
	|D\chi_{\{u_m>0\}}|(B_R\cap\{x_n>\varepsilon\})
	\le C_1\varepsilon^{-1/2}R^{n-1/2}.
\]
For each fixed truncation \(\varepsilon>0\), the estimate above and
\(0\le\chi_{\{u_m>0\}}\le1\) give a uniform bound for \(\chi_m\) in
\(BV(B_R\cap\{x_n>\varepsilon\})\). It is easy to check that the set is bounded and
Lipschitz, it then follows from \cite[Theorem 5.5]{MR3409135}
that there is a subsequence converging strongly in
\(L^1(B_R\cap\{x_n>\varepsilon\})\). Applying the same argument with
\(\varepsilon\) replaced by \(1/j\) and \(R\) by \(k\), we may extract a 
further subsequence and a function
\(\chi_0\in L^1_{\mathrm{loc}}(\{x_n>0\})\) which satisfies
\(\chi_{\{u_m>0\}}\to\chi_0\) in
\(L^1(B_R\cap\{x_n>\varepsilon\})\) for every \(R>0\) and
\(\varepsilon>0\).
We then extend \(\chi_0\) by zero on \(\{x_n\le0\}\).
Let us remark that it remains to control the strip near \(\Pi\). 
Recall that \(u\equiv 0\) on \(\{x_n\le 0\}\), we have by a direct calculation that
\[
	|B_R\cap\{0<x_n\le\varepsilon\}|\le C_R\varepsilon.
\]
therefore the \(L^1(B_R)\) difference between \(\chi_m\) and \(\chi_0\) then splits into this strip and its complement. 
Letting \(m\to\infty\) first and then \(\varepsilon\to0^+\) proves \eqref{eq.bu-phase-strong-limit} on every compact subset.
\end{proof}

Let \(n=2\) and set 
\begin{equation*}
	\mathcal C_{\mathrm{st}}
	:=\{(\varrho\cos\theta,\varrho\sin\theta):
	\varrho>0,\ \pi/6<\theta<5\pi/6\},
\end{equation*}
where we use polar coordinates \(x=(\varrho\cos\theta,\varrho\sin\theta)\) for \(0\le\theta<2\pi\). Define the Stokes profile as
\begin{equation}\label{eq.bu-Stokes-profile}
	u_{\mathrm{st}}(\varrho,\theta)
	:=\begin{cases}
		\displaystyle\frac{\sqrt2}{3}\varrho^{3/2}
		\cos\!\left(\frac32(\theta-\pi/2)\right),
		&\pi/6<\theta<5\pi/6,\\[1mm]
		0, &\text{otherwise}.
	\end{cases}
\end{equation}

\begin{proposition}[Two-dimensional classification]\label{Proposition: 2-dimensional case}
Let \(n=2\), let \(u\) be a subsonic variational solution of \eqref{eq.quaslinear-free-boundary-problem2}, and let \(x^0\in\Omega\cap\Pi\) be a stagnation point. Assume the growth condition \eqref{eq.fd-growth}. Then
\begin{equation}\label{eq.bu-three-densities}
	\Phi(x^0,u;0^+)\in
	\left\{0,\ \frac{\sqrt3}{3H_0},\ \frac{2}{3H_0}\right\}.
\end{equation}
For the blow-up sequence \(u_r(x)=r^{-3/2}u(x^0+rx)\), the
corresponding limits are
\begin{equation}\label{eq.bu-three-pairs}
	(u_r,\chi_{\{u_r>0\}})\longrightarrow
	\begin{cases}
		(0,0),&\Phi(x^0,u;0^+)=0,\\[1mm]
		(u_{\mathrm{st}},\chi_{\mathcal C_{\mathrm{st}}}),
		&\Phi(x^0,u;0^+)=\sqrt3/(3H_0),\\[1mm]
		(0,\chi_{\{x_2>0\}}),&\Phi(x^0,u;0^+)=2/(3H_0).
	\end{cases}
\end{equation}
The functions converge locally uniformly and strongly in
\(W^{1,2}_{\mathrm{loc}}(\R^2)\), and the phase indicators converge
strongly in \(L^1_{\mathrm{loc}}(\R^2)\).
\end{proposition}

\begin{proof}
Choose an arbitrary sequence of radii. The preceding lemmas give a
subsequence converging to a stationary pair \((u_0,\chi_0)\) with
\(\chi_0\) a characteristic function, which we now classify directly. We divide the proof into four steps.

\emph{Step I.}
Homogeneity and local Lipschitz regularity give
\(u_0(\varrho,\theta)=\varrho^{3/2}q(\theta)\) with \(q\ge0\) and
\(q=0\) on \([\pi,2\pi]\); the function \(q\) is Lipschitz on the
circle. On every component of \(\{q>0\}\), harmonicity gives
\(q''+\frac94q=0\), so every positive component has zero endpoint
values and length \(2\pi/3\). There can be at most one such interval
in \((0,\pi)\). We first note that a positive component cannot reach
\(\theta=0\) or \(\theta=\pi\). Indeed, if a component began at
\(\theta=0\), a direct calculation shows that \(\gamma=2\pi/3\) and
\(q(\theta)=c\sin(3\theta/2)\) on \((0,\gamma)\). Consequently, we have
\[
	\frac{q(\theta)}{\theta}
	=\frac{3c}{2}\,\frac{\sin(3\theta/2)}{3\theta/2}
	\longrightarrow\frac{3c}{2}>0,
\]
as \(\theta\to0^+\),
On the other hand, it follows from
\(0\le u_0(x)\le\frac23\sqrt{C_0}(x_2^+)^{3/2}\) that 
\[
	0\le q(\theta)\le C(\sin\theta)^{3/2}\le C\theta^{3/2},
\]
which implies \(q(\theta)/\theta\le C\theta^{1/2}\to0\) and gives a contradiction.
The case that a component ended at \(\theta=\pi\) is similar so we omit the details.
Therefore every positive component satisfies
\(0<\alpha<\gamma<\pi\). Hence, if \(u_0\not\equiv0\), there exist
\(0<\alpha<\gamma<\pi\) with \(\gamma-\alpha=\frac{2\pi}{3}\) such
that
\begin{equation}\label{eq.bu-angular-general}
	q(\theta)=c\sin\!\left(\frac32(\theta-\alpha)\right)
	\quad(\alpha<\theta<\gamma),\qquad c>0,
\end{equation}
and \(q=0\) elsewhere in the upper semicircle.

\emph{Step II.}
Observe that the equation \eqref{pu0}
reduces to 
\[
	\int\chi_0\operatorname{div}(x_2\phi)\,\dx{x}=0
\]
on any connected open set contained in
\(\{x_2>0\}\cap\operatorname{int}(\{u_0=0\})\). Take \(\phi=\psi/x_2\) for \(x_2>0\), 
we see that \(\nabla\chi_0=0\) in the sense of distributions. Consequently, 
\(\chi_0\in\{0,1\}\) is constant on each such connected set. If \(u_0\not\equiv 0\), then
\(\chi_0=c_\pm\) on the two exterior sectors for some constants, while \(\chi_0=1\) on the positive sector. 

\emph{Step III.}
Let \(p\ne0\) be a point on the ray \(\theta=\alpha\) (the other case is similar). Choose \(r>0\) so small that
\(B_r(p)\subset\{x_2>0\}\) and \(B_r(p)\) does not intersect the other ray \(\theta=\gamma\). Let \(\nu_u\) be the normal pointing out of the
positive sector, and define
\[
	\mathbf T_0=(|\nabla u_0|^2+x_2\chi_0)\operatorname{Id}
	-2\nabla u_0\otimes\nabla u_0.
\]
In each open side a direct calculation gives \(\operatorname{div}\mathbf T_0=\chi_0\mathbf e_2\),
so integrating \eqref{pu0} over the two sides shows that
\(\mathbf T_0\nu_u\) is continuous across the ray. Hence we obtain the jump condition
\begin{equation}\label{eq.bu-ray-jump}
	|\nabla u_0|^2=(1-c_\pm)x_2
	\quad\text{on the corresponding ray}.
\end{equation}
It follows from \eqref{eq.bu-angular-general} that \(c_-=c_+=0\). 
Since \(|\nabla u_0|^2=(9/4)c^2\varrho\)
in the positive sector, we may deduce from \eqref{eq.bu-ray-jump} that
\(\frac94c^2=\sin\alpha=\sin\gamma\). Since
\(0<\alpha<\gamma<\pi\) and \(\gamma-\alpha=2\pi/3\), we have
\(\alpha=\frac\pi6\), \(\gamma=\frac{5\pi}6\), and
\(c=\frac{\sqrt2}{3}\). Thus the nonzero pair is exactly
\((u_{\mathrm{st}},\chi_{\mathcal C_{\mathrm{st}}})\).

\emph{Step IV.}
If \(u_0\equiv0\), we follow the argument of Step II that either \(\chi_0=0\) or
\(\chi_0=\chi_{\{x_2>0\}}\) a.e. It then follows from \eqref{eq.bu-pair-density}
that \(0\),
\[
	\frac1{H_0}\int_0^1\varrho^2\,\dx{\varrho}
	\int_{\pi/6}^{5\pi/6}\sin\theta\,\dx{\theta}
	=\frac{\sqrt3}{3H_0},
\]
or
\[
	\frac1{H_0}\int_0^1\varrho^2\,\dx{\varrho}
	\int_0^\pi\sin\theta\,\dx{\theta}
	=\frac2{3H_0},
\]
respectively. This proves \eqref{eq.bu-three-pairs}.
\end{proof}

\begin{remark}
	We call \(\sqrt3/(3H_0)\) the \emph{Stokes density}, zero the \emph{zero density}, and \(\Theta_0=2/(3H_0)\) the \emph{full weighted density}.

	At the full weighted density the natural function limit is zero although the limiting phase fills the upper half-plane. This is a candidate for the horizontally flat alternative.
\end{remark}

\section{Non-degenerate stagnation points}

Throughout this section, \(u\) is a subsonic variational solution
satisfying \eqref{eq.fd-growth}. Recall the notation for the stagnation set, i.e.,
\begin{equation*}
	S_u=\Omega\cap\partial\{u>0\}\cap\Pi,
\end{equation*}
where \(\Pi\) is the stagnation hyperplane defined in Definition \ref{def.stagnation-hyperplane}.

\begin{definition}[Non-degenerate and degenerate stagnation points]
\label{Definition: stgnation points}
	A point \(x^0\in S_u\) is \emph{non-degenerate} if
	\begin{equation*}
		\frac{u(x^0+r_m x)}{r_m^{3/2}}\to u_0(x)\not\equiv0,
	\end{equation*}
	locally uniformly and strongly in \(W^{1,2}_{\mathrm{loc}}(\R^n)\)
	for some sequence \(r_m\to 0^+\). Otherwise, \(x^0\) is \emph{degenerate}.
\end{definition}

\begin{remark}
	By natural compactness, non-degeneracy is equivalent to
	\[
		\limsup_{r\to0^+}r^{-3/2}\sup_{B_r(x^0)}u>0.
	\]
\end{remark}

\begin{remark}
	In dimension two, under the standing growth condition \eqref{eq.fd-growth}, Proposition \ref{Proposition: 2-dimensional case} identifies non-degeneracy with Stokes density. Degenerate points have either zero density or full weighted density \(\frac{2}{3H_0}\).
\end{remark}

\begin{lemma}\label{Lemma:measure1}
Let \(n=2\) and let \(u\) be a subsonic variational solution of
\eqref{eq.quaslinear-free-boundary-problem2} satisfying
\eqref{eq.fd-growth}, and let \(x^0\in S_u\) be a non-degenerate
point. Then the following assertions hold.

\begin{enumerate}[label=(\roman*)]

\item \(\Phi(x^0,u;0^+)=\sqrt3/(3H_0)\), and
\(u_r\to u_{\mathrm{st}}\) locally uniformly and strongly in
\(W^{1,2}_{\mathrm{loc}}(\R^2)\) for the full family
\(u_r(x)=r^{-3/2}u(x^0+rx)\).

\item The weighted phase density and the area fraction satisfy
\begin{equation}\label{eq.bu-Stokes-density-limits}
	\begin{aligned}
		&\lim_{r\to0^+}r^{-3}\int_{B_r(x^0)}x_2^+\chi_{\{u>0\}}\,\dx{x}
		=\frac{\sqrt3}{3},\\
		&\lim_{r\to0^+}\frac{|B_r(x^0)\cap\{u>0\}|}{|B_r(x^0)|}=\frac13.
	\end{aligned}
\end{equation}
In particular, for some \(r_{x^0}>0\),
\begin{equation}\label{nden1}
	\frac{\displaystyle\int_{B_r(x^0)}x_2^+\chi_{\{u>0\}}\,\dx{x}}
	{r\,\mathcal L^2(B_r(x^0))}\ge\frac{\sqrt3}{6\pi}
	\qquad(0<r<r_{x^0}).
\end{equation}
For every fixed \(R>0\),
\begin{equation}\label{mes1}
	\mathcal L^2\!\left(B_R\cap\bigl(
	\{x\in\R^2:u(x^0+rx)>0\}\,\triangle\,\mathcal C_{\mathrm{st}}
	\bigr)\right)\to0,
\end{equation}

\item Let \(\mu:=Lu\) be the nonnegative Radon measure. Define its
natural rescaling, acting on compact Borel sets in the rescaled
domain, by
\begin{equation*}
	\tilde\mu_r(A):=H_0r^{-3/2}\mu(x^0+rA).
\end{equation*}
Then
\begin{equation}\label{eq.bu-Stokes-flux-limit}
	\tilde\mu_r\stackrel{*}{\rightharpoonup}\Delta u_{\mathrm{st}}
	=\sqrt{x_2}\,\cH^1\mres
	\left(\ell_{\pi/6}\cup\ell_{5\pi/6}\right)
	\quad\text{locally in }\R^2,
\end{equation}
where \(\ell_\alpha:=\{(\varrho\cos\alpha,\varrho\sin\alpha):
\varrho>0\}\). In particular,
\begin{equation}\label{eq.bu-total-flux-limit}
	r^{-3/2}\mu(B_r(x^0))\longrightarrow\frac{2\sqrt2}{3H_0}.
\end{equation}

Set
\[
	\mathcal V_{\mathrm{st}}:=\{
	(\cos(\pi/6),\sin(\pi/6)),
	(\cos(5\pi/6),\sin(5\pi/6))\},
\]
and, for \(\varepsilon>0\), let
\[
	\mathcal C_\varepsilon:=\left\{x\ne0:
	\operatorname{dist}_{\mathbb S^1}
	\left(\frac{x}{|x|},\mathcal V_{\mathrm{st}}\right)<\varepsilon\right\},
\]
where \(\operatorname{dist}_{\mathbb S^1}\) is angular
distance. Then
\begin{equation}\label{mes2}
	r^{-3/2}\mu\left(x^0+r(B_1\setminus\mathcal C_\varepsilon)\right)\to 0.
\end{equation}
\end{enumerate}
\end{lemma}

\begin{proof}
We first remark that assertion (i) and \eqref{mes1} are direct conclusions 
of Proposition~\ref{Proposition: 2-dimensional case}. The first limit
in \eqref{eq.bu-Stokes-density-limits} follows directly from
\eqref{eq.bu-weighted-density}, and the second follows from \(\chi_{\{u_r>0\}}\to\chi_{\mathcal C_{\mathrm{st}}}\) and a simple calculation that \(\mathcal L^2(B_1\cap\mathcal C_{\mathrm{st}})
=\pi/3\). Also, it is easy to check the ratio in
\eqref{nden1} converges to \(\sqrt3/(3\pi)\).

To prove (iii), let us define
\[
	\hat a_r(x):=\frac{H_0}{H(r|\nabla u_r(x)|^2;rx_2)}.
\]
Changing variables gives
\[
	\begin{aligned}
		\tilde\mu_r&=\operatorname{div}(\hat a_r\nabla u_r),\\
		\langle\tilde\mu_r,\eta\rangle&=-\int\hat a_r\nabla u_r\cdot
		\nabla\eta\,\dx{x}\quad(\eta\in C_c^1(\R^2)),
	\end{aligned}
\]
Then \(\hat a_r\to1\) uniformly on
compact sets by freezing the coefficient, so the strong \(W^{1,2}\) convergence of \(u_r\) implies that
\[
	\langle\tilde\mu_r,\eta\rangle
	\longrightarrow-\int\nabla u_{\mathrm{st}}\cdot\nabla\eta\,\dx{x}
	=\langle\Delta u_{\mathrm{st}},\eta\rangle.
\]
Fix a compact set \(K\subset\R^2\) and a cutoff
\(\zeta\in C_c^\infty(\R^2)\) which satisfies \(0\le\zeta\le1\) and \(\zeta=1\) on a
neighborhood of \(K\). Since \(\tilde\mu_r\ge0\), we have 
\[
	\begin{aligned}
		\tilde\mu_r(K)&\le\langle\tilde\mu_r,\zeta\rangle\\
		&=-\int\hat a_r\nabla u_r\cdot\nabla\zeta\,\dx{x}\\
		&\le\|\hat a_r\|_{L^\infty}
		\|\nabla u_r\|_{L^2(\operatorname{supp}\nabla\zeta)}
		\|\nabla\zeta\|_{L^2}
		\le C_K.
	\end{aligned}
\]
Hence the masses of \(\tilde\mu_r\) are locally uniformly bounded.
Now let \(\phi\in C_c(\R^2)\) be a test function and choose
\(\phi_\varepsilon\in C_c^\infty(\R^2)\) which satisfies
\(\|\phi-\phi_\varepsilon\|_\infty\le\varepsilon\) and 
\(\operatorname{supp}\phi_\varepsilon\) contained in a fixed compact set
\(K'\). Then
\[
	\bigl|\langle\tilde\mu_r-\Delta u_{\mathrm{st}},\phi\rangle\bigr|
	\le\varepsilon\bigl(\tilde\mu_r(K')+\Delta u_{\mathrm{st}}(K')\bigr)
	+\bigl|\langle\tilde\mu_r-\Delta u_{\mathrm{st}},
	\phi_\varepsilon\rangle\bigr|.
\]
Letting \(r\to0^+\) first and then \(\varepsilon\to0^+\) gives
\(\langle\tilde\mu_r,\phi\rangle\to
\langle\Delta u_{\mathrm{st}},\phi\rangle\) for every
\(\phi\in C_c(\R^2)\).

Recall that \(u_{\mathrm{st}}\) is harmonic in \(\mathcal C_{\mathrm{st}}\) and vanishes on \(\partial\mathcal C_{\mathrm{st}}\setminus\{0\}\), On either ray of \(\mathcal C_{\mathrm{st}}\), a direct calculation shows that its outward normal derivative is
\(-\sqrt{x_2}\). Integrating over the sector with a disk of radius
\(\rho\) removed yields
\eqref{eq.bu-Stokes-flux-limit}. Hence \(\tilde\mu_r(B_1)\to\Delta u_{\mathrm{st}}(B_1)\),
and
\[
	\frac1{H_0}\Delta u_{\mathrm{st}}(B_1)
	=\frac2{H_0}\int_0^1\sqrt{\frac\varrho2}\,\dx{\varrho}
	=\frac{2\sqrt2}{3H_0}.
\]
Then \eqref{eq.bu-total-flux-limit} follows from dividing the scaled measure by \(H_0\).
Let us define the set 
\[
	K_\varepsilon:=\{0\}\cup\left\{0<|x|\le1:
	\operatorname{dist}_{\mathbb S^1}
	\left(\frac{x}{|x|},\mathcal V_{\mathrm{st}}\right)\ge\varepsilon\right\}.
\]
Then \(K_\varepsilon\) is compact in \(B_2\), 
contains \(B_1\setminus\mathcal C_\varepsilon\),
and has zero limiting measure. Thus we infer from the weak-star convergence that
\(\limsup_{r\to0^+}\tilde\mu_r(K_\varepsilon)\le0\). This proves
\eqref{mes2}.
\end{proof}

\begin{remark}
	Let us remark that equations \eqref{mes1} and \eqref{mes2} concern phase volume and flux concentration. However, they do \emph{not} assert that the entire topological free boundary converges in Hausdorff distance to the two Stokes rays. The domination in Lemma~\ref{Lemma: weak v.s variational} shows that \eqref{mes2} also yields the corresponding weighted-perimeter concentration for subsonic variational solutions. The later geometric argument uses the additional local curve hypothesis.
\end{remark}

\begin{proposition}[Isolation of non-degenerate stagnation points]
\label{Proposition: iso-stp}
Let \(n=2\) and let \(u\) be a subsonic variational solution of
\eqref{eq.quaslinear-free-boundary-problem2} satisfying
\eqref{eq.fd-growth}. If \(x^0\in S_u\) is a non-degenerate stagnation point, then \(x^0\) is the \emph{only}
non-degenerate stagnation point in that neighborhood.
\end{proposition}

\begin{proof}
If \(x^0\) were not isolated among the non-degenerate stagnation points, there would be non-degenerate stagnation points \(x^m\ne x^0\) with \(x^m\to x^0\). It then follows from Proposition \ref{Proposition: 2-dimensional case} that
\begin{equation}\label{eq.bu-density-gap-at-centers}
	\Phi(x^m,u;0^+)\ge\frac{\sqrt3}{3H_0}.
\end{equation}
Write \(r_m=|x^m-x^0|\). Then along a subsequence we have
 \(z:=\frac{x^m-x^0}{r_m}\) is identically \((1,0)\) or
\((-1,0)\). Moreover, the blow-ups
\(u_m(x)=r_m^{-3/2}u(x^0+r_mx)\) converge to \(u_{\mathrm{st}}\)
locally uniformly and strongly in \(W^{1,2}_{\mathrm{loc}}\), and
\(\chi_m\) converge strongly to \(\chi_{\mathcal C_{\mathrm{st}}}\).

Choose \(0<\rho<1/4\). Since \(\overline{B_\rho(z)}\) is
disjoint from \(\overline{\mathcal C_{\mathrm{st}}}\), so \(u_0=\chi_0=0\) on \(B_\rho(z)\). Moreover, it
follows from \eqref{eq.bu-fixed-radius-convergence} that
\[
	\Phi_m(z,u_m;\rho)=\Phi(x^m,u;r_m\rho)\longrightarrow0
\]
Combining this with
\eqref{eq.bu-density-gap-at-centers} implies 
\[
	0<\frac{\sqrt3}{3H_0}\le\Phi(x^m,u;0^+)
	\le\Phi(x^m,u;r_m\rho)+Cr_m^2\rho^2 \longrightarrow0,
\]
which yields a contradiction.
\end{proof}

As a direct corollary, we have

\begin{corollary}\label{cor.bu-Stokes-accumulation}
Let \(n=2\) and let \(u\) be a subsonic variational solution of
\eqref{eq.quaslinear-free-boundary-problem2} satisfying
\eqref{eq.fd-growth}. The non-degenerate stagnation set is at most
countable. Every accumulation point in \(\Omega\) of distinct
non-degenerate stagnation points is a full weighted density
degenerate stagnation point.
\end{corollary}

\begin{proof}
If a non-degenerate points \(x^m\) converge to
\(y\in\Omega\), the relative closedness of
\(\partial\{u>0\}\cap\Pi\) implies \(y\in S_u\), and we deduce from the upper
semicontinuity that
\[
	\Phi(y,u;0^+)\ge\limsup_m\Phi(x^m,u;0^+)
	=\frac{\sqrt3}{3H_0}.
\]
The isolation proposition excludes Stokes density at \(y\), so the
three possible density values force
\(\Phi(y,u;0^+)=2/(3H_0)\), which implies \(u_0=0\).
\end{proof}

\begin{remark}
	The preceding corollary shows that every possible interior accumulation of Stokes points lies in the full weighted density set, which is the subject of the following frequency analysis.
\end{remark}

\section{Degenerate stagnation points}\label{sec.degenerate-stagnation}

\subsection{Nonexistence of zero-density points}

In this section, we investigate the degenerate stagnation points.First recall the notation \(\Pi=\{x_2=0\}\), \(S_u=\Omega\cap\partial\{u>0\}\cap\Pi\), and set
\begin{equation*}
	S_u^{\rm cusp}:=\{x^0\in S_u:\Phi(x^0,u;0^+)=0\}.
\end{equation*}
The density formula \eqref{eq.bu-weighted-density} reads
\begin{equation}\label{eq.cusp-density-input}
	\Phi(x^0,u;0^+)=\frac1{H_0}\lim_{r\to0^+}r^{-3}\int_{B_r(x^0)}x_2^+\chi_{\{u>0\}}\,\dx{x},\qquad x^0\in\Omega\cap\Pi.
\end{equation}
Fix \(x^0\in S_u\) and \(R>0\) such that \(B_R(x^0)\subset\subset\Omega\). Let \(C_0\) be a constant for which \eqref{eq.fd-growth} holds on this ball, and choose \(H_{\min}>0\) such that
\[
	H(|\nabla u(x)|^2;x_2)\ge H_{\min}\qquad\text{ for a.e. }x\in B_R(x^0).
\]
Set
\[
	\mathsf L:=\frac{H_0}{H_{\min}}\sqrt{C_0}.
\]
Using the coefficient \(a=1/H(|\nabla u|^2;x_2)\) already defined in \eqref{eq.notations}, normalize the flux by
\begin{equation}\label{eq.cusp-normalized-flux}
	\mathbf a:=H_0a\nabla u,\qquad\hat\mu:=\operatorname{div}\mathbf a=H_0\mu.
\end{equation}
On the chosen ball, the growth bound and \(\nabla u=0\) a.e. on \(\{u=0\}\) imply
\begin{equation}\label{eq.cusp-flux-upper}
	|\mathbf a|\le\mathsf L\sqrt{x_2^+}\,\chi_{\{u>0\}}\quad\text{a.e.}
\end{equation}

\begin{lemma}\label{lem.cusp-flux}
Let \(n=2\) and let \(u\) be a subsonic variational solution of \eqref{eq.quaslinear-free-boundary-problem2}. Assume that \(u\) satisfies the growth condition \eqref{eq.fd-growth}. Then the normalized flux \(\hat\mu\) defined in \eqref{eq.cusp-normalized-flux} is a nonnegative Radon measure.
Moreover, on \(B_R(x^0)\cap\{x_2>0\}\), we have
\begin{equation}\label{eq.cusp-flux-perimeter}
	\chi_{\{u>0\}}\in BV_{\mathrm{loc}}(\Omega\cap\{x_2>0\}),
	\qquad
	\sqrt{x_2}\,\dx{|D\chi_{\{u>0\}}|}
	\le 2\sqrt{C_0}\,\dx{\hat\mu}
\end{equation}
as an inequality of measures. Define
\begin{equation}\label{eq.cusp-weighted-mass}
	m(r):=\int_{B_r(x^0)}\sqrt{x_2^+}\,
	\chi_{\{u>0\}}\,\dx{x},
\end{equation}
which is absolutely continuous on \((0,R)\) and satisfies,
\begin{equation}\label{eq.cusp-perimeter-flux-bound}
	\int_{B_r(x^0)\cap\{x_2>0\}}
	\sqrt{x_2}\,\dx{|D\chi_{\{u>0\}}|}
	\le 2\sqrt{C_0}\,\hat\mu(B_r(x^0))
	\le 2\sqrt{C_0}\mathsf L\,m'(r).
\end{equation}
for almost every \(r\in(0,R)\).
\end{lemma}

\begin{proof}
Remark~\ref{rem.flux-measure} and
Lemma~\ref{Lemma: weak v.s variational} give positivity of
\(\hat\mu\), the local BV property, and
\eqref{eq.cusp-flux-perimeter}. Using the coarea formula, we have
for a.e. \(r\),
\begin{equation}\label{eq.cusp-coarea}
	m'(r)=\int_{\partial B_r(x^0)}\sqrt{x_2^+}\,
	\chi_{\{u>0\}}\,\dH{1}.
\end{equation}
To compute the spherical flux, take the radial cutoff
\(\eta_h\) that equals one on \(B_r(x^0)\), decreases linearly to
zero on \(B_{r+h}(x^0)\setminus B_r(x^0)\), and vanishes outside
\(B_{r+h}(x^0)\). Smooth approximation gives
\[
	\int\eta_h\,\dx{\hat\mu}
	=\frac1h\int_r^{r+h}\int_{\partial B_s(x^0)}
	\mathbf a\cdot\nu\,\dH{1}\,\dx{s}.
\]
There exists a set \(E\subset(0,R)\) with
\(\cL^1((0,R)\setminus E)=0\) such that, for every \(r\in E\),
\[
	\hat\mu\bigl(\partial B_r(x^0)\bigr)=0
\]
and \(r\) is a Lebesgue point of
\(s\mapsto\int_{\partial B_s(x^0)}\mathbf a\cdot\nu\,\dH{1}\).  
Passing
to the limit as \(h\to0^+\) and using
\eqref{eq.cusp-flux-upper}, we obtain
\begin{equation}\label{eq.cusp-flux-identity}
	\hat\mu(B_r(x^0))
	=\int_{\partial B_r(x^0)}\mathbf a\cdot\nu\,\dH{1}
	\le\mathsf Lm'(r).
\end{equation}
Then \eqref{eq.cusp-perimeter-flux-bound} follows from \eqref{eq.cusp-flux-perimeter} and \eqref{eq.cusp-flux-identity}.
\end{proof}

\begin{lemma}\label{lem.cusp-weighted-isoperimetry}
Assume that \(A\subset\subset\{x_2>0\}\) is a set of finite perimeter. Then we have
\begin{equation}\label{eq.cusp-isoperimetric}
	P_{\mathrm{w}}(A)\ge c_{\mathrm{iso}}\,m_{\mathrm{w}}^{3/5}(A),
	\qquad c_{\mathrm{iso}}=\left(\frac{\pi^2}{2}\right)^{1/5}.
\end{equation}
Here in \eqref{eq.cusp-isoperimetric}, 
\[
	m_{\mathrm{w}}(A):=\int_A\sqrt{x_2}\,\dx{x},\qquad
	P_{\mathrm{w}}(A):=\int_{\partial^*A}\sqrt{x_2}\,\dH{1},
\]
where \(\partial^*A\) denotes the reduced boundary of the set \(A\). 
\end{lemma}

\begin{proof}

By assumption, the set \(A\) is bounded away from \(\{x_2=0\}\), so 
\(\sqrt{x_2}\) is a smooth function on a neighborhood of \(\overline A\). 
Choose a test function 
\(\zeta\in C_c^\infty(\{x_2>0\})\) with \(\zeta\equiv1\) on a
neighborhood of \(\overline A\), so that
\(V:=\zeta\sqrt{x_2}\,\mathbf e_2\) is a compactly supported
\(C^1\) field. Applying the Gauss--Green formula on \(\R^2\) to \(V\)
yields
\[
	\int_A\operatorname{div}V\,\dx{x}
	=\int_{\partial^*A}V\cdot\nu_A\,\dH{1}.
\]
Since
\[
	\operatorname{div}V
	=\frac{\zeta}{2\sqrt{x_2}}+\sqrt{x_2}\,\partial_2\zeta.
\]
Observe that \(\zeta=1\), \(\partial_2\zeta=0\) on a neighborhood of
\(\overline A\), while \(\zeta=1\) on \(\partial^*A\), we have
\begin{equation}\label{eq.cusp-negative-moment}
	\begin{aligned}
		\frac12\int_A x_2^{-1/2}\,\dx{x}
		&=\int_{\partial^*A}\sqrt{x_2}\,\nu_{A,2}\,\dH{1}\\
		&\le\int_{\partial^*A}\sqrt{x_2}\,|\nu_{A,2}|\,\dH{1}
		\le P_{\mathrm{w}}(A).
	\end{aligned}
\end{equation} 
Extend \(f:=\sqrt{x_2}\chi_A\) by zero to \(\R^2\). The BV product
rule and \eqref{eq.cusp-negative-moment} give
\[
	|Df|(\R^2)\le P_{\mathrm{w}}(A)
	+\frac12\int_Ax_2^{-1/2}\,\dx{x}
	\le2P_{\mathrm{w}}(A).
\]
The two-dimensional Sobolev embedding for BV functions \cite{MR3409135} provides
\[
	\|f\|_{L^2(\R^2)}\le\frac1{2\sqrt\pi}|Df|(\R^2),
\]
and hence
\begin{equation*}
	\int_Ax_2\,\dx{x}\le\frac1\pi P_{\mathrm{w}}^2(A).
\end{equation*}
H\"older's inequality, with exponents \(3/2\) and \(3\), now yields
\[
	\begin{split}
		m_{\mathrm{w}}(A)
		&=\int_A(x_2)^{2/3}(x_2^{-1/2})^{1/3}\,\dx{x}\\
		&\le\left(\int_Ax_2\,\dx{x}\right)^{2/3}
		\left(\int_Ax_2^{-1/2}\,\dx{x}\right)^{1/3}\\
		&\le 2^{1/3}\pi^{-2/3}P_{\mathrm{w}}^{5/3}(A),
	\end{split}
\]
and rearranging proves \eqref{eq.cusp-isoperimetric}. The exponent
\(3/5\) is determined by the scaling of weighted perimeter and
volume, whose homogeneities are \(3/2\) and \(5/2\), respectively.
\end{proof}

The flux estimate and the weighted isoperimetric inequality give
a positive lower bound for the phase density at every stagnation point. 

\begin{theorem}\label{thm.cusp-exclusion}
Let \(n=2\) and let \(u\) be a subsonic variational solution of \eqref{eq.quaslinear-free-boundary-problem2}. Let
\(x^0\in S_u\) be a stagnation point with \(B_R(x^0)\subset\subset\Omega\). Assume in addition that 
\eqref{eq.fd-growth} holds on \(B_R(x^0)\). Then for every \(0<r<R\),
\begin{equation}\label{eq.cusp-positive-density}
	\int_{B_r(x^0)}x_2^+\chi_{\{u>0\}}\,\dx{x}\ge c_{\mathrm{ph}}r^3.
\end{equation}
Consequently, we have 
\begin{equation}\label{eq.cusp-positive-Weiss}
	\Phi(x^0,u;0^+)\ge\frac{c_{\mathrm{ph}}}{H_0}>0,
	\qquad S_u^{\mathrm{cusp}}=\varnothing.
\end{equation}
In particular, the constant \(c_{\mathrm{ph}}>0\) in \eqref{eq.cusp-positive-density} and \eqref{eq.cusp-positive-Weiss} can be chosen as 
\[
	c_{\mathrm{ph}}:=\frac1\pi
	\left(\frac{2c_{\mathrm{iso}}}{5(1+2\sqrt{C_0}\mathsf L)}\right)^5,
\]
where \(c_{\mathrm{iso}}\) is given in
Lemma~\ref{lem.cusp-weighted-isoperimetry}, and 
\begin{equation*}
	\begin{gathered}
		\mathsf L:=\frac{H_0}{H_{\min}}\sqrt{C_0},\\
		H(|\nabla u(x)|^2;x_2)\ge H_{\min}
		\quad\text{for a.e. }x\in B_R(x^0),
	\end{gathered}
\end{equation*}
\end{theorem}

\begin{remark}
	The local bound \(H\ge H_0/2\) allows the uniform choice \(\mathsf L=2\sqrt{C_0}\).
\end{remark}

\begin{proof}[Proof of Theorem~\ref{thm.cusp-exclusion}]
After a horizontal translation we may assume without loss of generality that 
\(x^0=0\). Recall \(m(r)\) defined in \eqref{eq.cusp-weighted-mass}. Note that 
every ball about the origin contains an open
positive-phase set at positive height, this is because \(0\in\partial\{u>0\}\) and \(\{u>0\}\) is an open set. It follows that
\[
	m(r)>0
\]
for every \(0<r<R\). Moreover, it is easy to check that 
\begin{equation*}
	\lim_{r\to0^+}m(r)=0.
\end{equation*}
Let \(r\in(0,R)\) be such that, for \(\mathcal H^1\)-a.e.
\(x\in\partial B_r(x^0)\cap\{x_2>0\}\), the limit
\[
	T_r\chi_{\{u>0\}}(x):=\lim_{\rho\to0^+}
	\frac{1}{\cL^2(B_\rho(x)\cap B_r(x^0))}
	\int_{B_\rho(x)\cap B_r(x^0)}\chi_{\{u>0\}}(y)\,\dx{y}
\]
exists and agrees with \(\chi_{\{u>0\}}(x)\), such that
\(|D\chi_{\{u>0\}}|(\partial B_r(x^0)\cap\{x_2>0\})=0\), and such that
the identities \eqref{eq.cusp-coarea} and \eqref{eq.cusp-flux-identity}
hold at \(r\). We note that these conditions hold for a.e. \(r\). Since
\(\sqrt{x_2^+}\) vanishes on \(\{x_2\le0\}\), only the upper arc enters
\eqref{eq.cusp-coarea}.

Let us now define 
\[
	\mathscr P(r):=\int_{B_r(x^0)\cap\{x_2>0\}}
	\sqrt{x_2}\,\dx{|D\chi_{\{u>0\}}|}.
\]
Observe that the horizontal line \(\{x_2=\varepsilon\}\) carries no
\(|D\chi_{\{u>0\}}|\)-measure for a.e. \(\varepsilon>0\), we will call such
\(\varepsilon\) admissible. For an admissible \(\varepsilon\), consider
the truncation
\[
	A_{r,\varepsilon}:=B_r(x^0)\cap\{u>0\}\cap\{x_2>\varepsilon\}.
\]
It is easy to check that \(A_{r,\varepsilon}\) is a set of finite
perimeter compactly contained in \(\{x_2>0\}\). Then it follows from the 
\(BV\) truncation
formula and a direct computation that
\[
	\begin{split}
		P_{\mathrm{w}}(A_{r,\varepsilon})\le{}&
		\int_{B_r(x^0)\cap\{x_2>\varepsilon\}}
		\sqrt{x_2}\,\dx{|D\chi_{\{u>0\}}|}\\
		&+\int_{\partial B_r(x^0)\cap\{x_2>\varepsilon\}}
		\sqrt{x_2}\,\chi_{\{u>0\}}\,\dH{1}\\
		&+\int_{B_r(x^0)\cap\{x_2=\varepsilon\}}
		\sqrt{x_2}\,\dH{1}.
	\end{split}
\]
Observe that \(\sqrt{x_2}=\sqrt{\varepsilon}\) and \(\chi_{\{u>0\}}\le1\)
hold on \(x_2=\varepsilon\). Also note that 
\(B_r(x^0)\cap\{x_2=\varepsilon\}\) is a segment of length at most
\(2r\). Thus, we obtain
\[
	\begin{aligned}
		P_{\mathrm{w}}(A_{r,\varepsilon})\le{}&
		\int_{B_r\cap\{x_2>\varepsilon\}}
		\sqrt{x_2}\,\dx{|D\chi_{\{u>0\}}|}\\
		&+\int_{\partial B_r\cap\{x_2>\varepsilon\}}
		\sqrt{x_2}\chi_{\{u>0\}}\,\dH{1}+2r\sqrt\varepsilon.
	\end{aligned}
\]
Passing to the limit as \(\varepsilon\to0^+\) through admissible \(\varepsilon\), 
we may deduce from the monotone convergence theorem that
\[
	\int_{B_r\cap\{x_2>\varepsilon\}}
	\sqrt{x_2}\,\dx{|D\chi_{\{u>0\}}|}
	\longrightarrow
	\int_{B_r\cap\{x_2>0\}}
	\sqrt{x_2}\,\dx{|D\chi_{\{u>0\}}|}
	=\mathscr P(r).
\]
Using \eqref{eq.cusp-coarea} and the monotone convergence once again, we have
\[
	\int_{\partial B_r\cap\{x_2>\varepsilon\}}
	\sqrt{x_2}\chi_{\{u>0\}}\,\dH{1}
	\longrightarrow
	\int_{\partial B_r\cap\{x_2>0\}}\sqrt{x_2}\,
	\chi_{\{u>0\}}\,\dH{1}
	=m'(r).
\]
It is easy to see that \(m_{\mathrm{w}}(A_{r,\varepsilon})\to m(r)\). Taking
\(\varepsilon\to0^+\) in \eqref{eq.cusp-isoperimetric} gives
\[
	c_{\mathrm{iso}}m^{3/5}(r)\le\mathscr P(r)+m'(r).
\]
It follows from \eqref{eq.cusp-flux-perimeter} and
\eqref{eq.cusp-flux-identity} that
\[
	\mathscr P(r)\le2\sqrt{C_0}\,\hat\mu(B_r(x^0))
	\le2\sqrt{C_0}\mathsf L\,m'(r),
\]
and hence
\[
	c_{\mathrm{iso}}m^{3/5}(r)
	\le(1+2\sqrt{C_0}\mathsf L)\,m'(r).
\]
On every interval \([\rho,r]\subset(0,R)\), the function \(m(r)\) is
positive and absolutely continuous, so the chain rule gives
\[
	\bigl(m^{2/5}(s)\bigr)'=\frac25m^{-3/5}(s)m'(s)
	\ge\frac{2c_{\mathrm{iso}}}{5(1+2\sqrt{C_0}\mathsf L)}
	\quad\text{a.e.}
\]
Integrating and then letting \(\rho\to0^+\) proves, for every
\(0<r<R\),
\begin{equation}\label{eq.cusp-half-moment-lower}
	m(r)\ge
	\left(\frac{2c_{\mathrm{iso}}}{5(1+2\sqrt{C_0}\mathsf L)}\right)^{5/2}
	r^{5/2}.
\end{equation}
We now apply the Cauchy's inequality, which gives
\[
	m^2(r)\le |B_r\cap\{u>0\}|
	\int_{B_r}x_2^+\chi_{\{u>0\}}\,\dx{x}
	\le\pi r^2\int_{B_r}x_2^+\chi_{\{u>0\}}\,\dx{x}.
\]
Combining this estimate with the lower bound
\eqref{eq.cusp-half-moment-lower} gives
\eqref{eq.cusp-positive-density}. Passing to the density limit in
\eqref{eq.cusp-density-input} proves \eqref{eq.cusp-positive-Weiss}.
\end{proof}

\begin{remark}
Let us remark that the estimate \eqref{eq.cusp-positive-density} also implies
the lower bound
\[
	|B_r(x^0)\cap\{u>0\}|\ge c_{\mathrm{ph}}r^2.
\]
Moreover, set
\(h_r:=\sup\{x_2:x\in B_r(x^0),\ u(x)>0\}\).
It follows from \eqref{eq.cusp-half-moment-lower} and
\[
	m(r)\le2r\int_0^{h_r}\sqrt{s}\,\dx{s}=\frac43rh_r^{3/2},
\]
that there is a
constant \(c_h>0\), depending only on \(C_0\) and \(\mathsf L\), such
that \(h_r\ge c_hr\).
\end{remark}

\begin{corollary}\label{cor.fd-density-alternatives}
Let \(n=2\) and let \(u\) be a subsonic variational solution of \eqref{eq.quaslinear-free-boundary-problem2}. Assume that \(u\) satisfies the growth condition \eqref{eq.fd-growth}. Then for each \(x^0\in S_u\) there holds
\[
	\Phi(x^0,u;0^+)\in\left\{\frac{\sqrt3}{3H_0},\frac2{3H_0}\right\}.
\]
\end{corollary}

\begin{proof}
The weighted-perimeter estimate of
Lemma~\ref{Lemma: weak v.s variational} allows us to apply
Proposition~\ref{Proposition: 2-dimensional case} at every stagnation
point. Of the three densities in \eqref{eq.bu-three-densities},
Theorem~\ref{thm.cusp-exclusion} excludes zero.
\end{proof}

It remains to study the full weighted density alternative.  
We develop its
frequency analysis for subsonic variational solutions.
For subsonic weak solutions,
we prove the nonexistence of it 
under
the local curve hypothesis and under the
finite-component hypothesis introduced in
Section~\ref{sec.full-density-exclusion}.

\subsection{Preliminary identities at full weighted density points}

Let \(n\ge2\) and let \(u\) be a
subsonic variational solution of \eqref{eq.quaslinear-free-boundary-problem2}. 
Assume that \(u\) satisfies the growth condition \eqref{eq.fd-growth}. We work in the fixed coefficient rectangle
\(\cR=(-t_0,t_0)\times(-s_0,s_0)\) (cf. equation \eqref{eq.cR}), where
\(H(0;0)=H_0>0\). For every \((t,s)\in\cR\) with \(t\ge0\) and every
\(0\le\tau\le t\), we have \((\tau,s)\in\cR\). It follows from Lemma~\ref{lem.preliminary} that
\[
	c\le H\le C,\qquad -C\le\partial_1H\le-c<0,\qquad 0<c\le\partial_2H\le C.
\]
We also recall the following quantities:
\begin{equation*}
	F(t;s)=\int_0^t\frac{\dx{\tau}}{H(\tau;s)},\qquad\lambda(s)=\frac1{H_0^2}\int_0^s H(0;\sigma)\,\dx{\sigma},
\end{equation*}
and
\begin{equation}\label{eq.fd-constitutive-identities}
	\begin{aligned}
		\partial_2F(t;s)&=\frac{H(t;s)-H(0;s)}{H_0^2},\\
		\lambda'(s)&=\frac{H(0;s)}{H_0^2},\qquad\lambda''(s)=\frac{\partial_2H(0;s)}{H_0^2}>0,\\
		\lambda(0)&=0,\qquad \lambda(s)=\frac{s}{H_0}+O(s^2).
	\end{aligned}
\end{equation}
Then we have  
\[
	\partial_2F(t;s)\le0\quad\text{ for }t\ge 0,\qquad|\partial_2F(t;s)|\le Ct,
\]
and 
\[
	cs\le\lambda(s)\le Cs\quad\text{ for small }s\ge 0.
\]
Let \(u\) be a subsonic variational solution, then by definition we have
\(u\in W^{1,2}_{\mathrm{loc}}(\Omega)\cap C^0(\Omega)\),
\(u\ge0\), \(u=0\) on \(\Omega\cap\{x_n\le0\}\), and the equation 
\begin{equation}\label{eq.fd-interior-equation}
	\operatorname{div}\left(\frac{\nabla u}{H(|\nabla u|^2;x_n)}\right)=0
	\quad\hbox{in }\Omega\cap\{u>0\}
\end{equation}
holds in the classical sense. Using \eqref{eq.F-height-identity}, we may rewrite \eqref{eq.domain-variation-formula} as
\begin{equation*}
	\begin{split}
		0={}&\int_\Omega
		\bigl(F(|\nabla u|^2;x_n)+\lambda(x_n)\chi_{\{u>0\}}\bigr)
		\operatorname{div}\phi\,\dx{x}\\
		&-2\int_\Omega\frac{\nabla u\cdot D\phi\nabla u}
		{H(|\nabla u|^2;x_n)}\,\dx{x}\\
		&+\int_\Omega\bigl(\partial_2F(|\nabla u|^2;x_n)
		+\lambda'(x_n)\chi_{\{u>0\}}\bigr)\phi_n\,\dx{x}
	\end{split}
\end{equation*}
for every \(\phi\in C_c^1(\Omega;\R^n)\). We define the set of full weighted
density points as
\begin{equation}\label{eq.fd-sets}
	\Sigma_u:=\{x\in S_u:\Phi(x,u;0^+)=\Theta_0\},\qquad\Theta_0:=\frac1{H_0}\int_{B_1}x_n^+\,\dx{x},
\end{equation}
where \(S_u=\Omega\cap\partial\{u>0\}\cap\{x_n=0\}\) is the set of stagnation points and \(\Theta_0=\omega_{n-1}/((n+1)H_0)\) where \(\omega_{n-1}\) the volume of the unit ball in \(\R^{n-1}\). Note that a simple calculation shows that \(\Theta_0=2/(3H_0)\) if \(n=2\).

Let \(x^0\in\Sigma_u\) be a fixed full weighted density point, 
and we translate it horizontally to the origin,
and write \(B_r=B_r(0)\) for simplicity. Choose \(r_0>0\) such that \(B_{r_0}\subset\subset\Omega\) and the following hold on \(B_{r_0}\):
(i) the growth condition \eqref{eq.fd-growth} is satisfied with a
single constant \(C_0\);
(ii) the coefficient states \((|\nabla u(x)|^2,x_n)\in\cR\). 

Recall \eqref{eq.notations}, we write \(a_0=1/H_0\). Besides, we abbreviate the
Weiss monotonicity formula introduced in Proposition~\ref{prop:wsm} as
\begin{equation}\label{eq.fd-original-Weiss}
	\begin{aligned}
		W(r)&:=a_0r^{-n-2}\int_{\partial B_r}u^2\,\dH{n-1},\\
		\Phi(r)&:=r^{-n-1}J_F(u;B_r)-\frac32W(r).
	\end{aligned}
\end{equation}
In this fashion, we have \(W(r)=W(0,u;r)\) and \(\Phi(r)=\Phi(0,u;r)\).

\begin{remark}
	Unless otherwise indicated, constants in this section depend only on the fixed local coefficient bounds, \(C_0\), and \(n\); small admissible radii may also depend on \(\beta\). Constants in the height and compactness estimates may additionally depend on a fixed outer-scale frequency bound.
\end{remark}

\begin{lemma}\label{lem.fd-flux-testing}
Let \(u\) be a
subsonic variational solution of \eqref{eq.quaslinear-free-boundary-problem2}. Define the flux measure
\[
	\mu=\operatorname{div}(a\nabla u).
\]
Then \(\mu\) is a nonnegative
Radon measure in \(B_{r_0}\), and \(\mu\) is supported on
\(B_{r_0}\cap\partial\{u>0\}\). Moreover, the following identities
hold:
\begin{enumerate}
	\item For every \(\zeta\in C_c^1(B_{r_0})\)
	\begin{equation}\label{eq.fd-product-test}
		\int a|\nabla u|^2\zeta\,\dx{x}
		=-\int au\nabla u\cdot\nabla\zeta\,\dx{x},
	\end{equation}
	\item For a.e. \(r<r_0\),
	\begin{equation}\label{eq.fd-base-energy}
		\int_{B_r}a|\nabla u|^2\,\dx{x}
		=\int_{\partial B_r}au\nabla u\cdot\nu\,\dH{n-1},
	\end{equation}
	\item For a.e. \(r<r_0\), 
	\begin{equation}\label{eq.fd-base-Pohozaev}
		\begin{split}
			r\frac{\dx{}}{\dx{r}}J_F(u;B_r)={}&nJ_F(u;B_r)
			-2\int_{B_r}a|\nabla u|^2\,\dx{x}\\
			&+2r\int_{\partial B_r}a(\nabla u\cdot\nu)^2\,\dH{n-1}\\
			&+\int_{B_r}x_n\bigl(\partial_2F(|\nabla u|^2;x_n)
			+\lambda'(x_n)\chi_{\{u>0\}}\bigr)\,\dx{x}.
		\end{split}
	\end{equation}
\end{enumerate}
\end{lemma}

\begin{proof}
Let \(\theta\in C^\infty(\R)\) be a nondecreasing function which satisfies  
\(\theta=0\) on
\((-\infty,1]\) and \(\theta=1\) on \([2,\infty)\). Let \(\zeta\in C_c^\infty(B_{r_0})\) be a nonnegative test function. Then it is easy to check that the function
\(\zeta\theta(ku)\) is compactly supported in \(\{u>0\}\) for \(k\ge 1\). Therefore we can use \(\zeta\theta(ku)\) as a test function for the equation \eqref{eq.fd-interior-equation}. Integration by parts gives
\[
	0=\int a\nabla u\cdot\nabla(\zeta\theta(ku))
	=\int a\theta(ku)\nabla u\cdot\nabla\zeta
	+k\int a\zeta\theta'(ku)|\nabla u|^2,
\]
we can rewrite this as 
\[
	-\int a\theta(ku)\nabla u\cdot\nabla\zeta
	=k\int a\zeta\theta'(ku)|\nabla u|^2\ge0.
\]
Since \(\theta(ku)\to\mathbf 1_{\{u>0\}}\) pointwise as \(k\to\infty\)
and \(0\le\theta(ku)\le1\), we deduce from the dominated convergence theorem and 
\(\nabla u=0\) a.e. on \(\{u=0\}\) that
\[
	\langle\mu,\zeta\rangle
	=-\int a\nabla u\cdot\nabla\zeta
	=\lim_{k\to\infty}\left(
	-\int a\theta(ku)\nabla u\cdot\nabla\zeta\right)
	\ge0
\]
for every nonnegative \(\zeta\in C_c^\infty(B_{r_0})\). Thus \(\mu\) is
a nonnegative distribution, hence a nonnegative Radon measure.

Moreover, it follows from \eqref{eq.fd-interior-equation} that \(\mu=0\) in \(\{u>0\}\), and \(\mu=0\) in the interior of
\(\{u=0\}\). This implies that
\(\operatorname{supp}\mu\subset B_{r_0}\cap\partial\{u>0\}\).

We still take a nonnegative test function
\(\zeta\in C_c^1(B_{r_0})\). For \(\varepsilon>0\), it is easy to check that 
the function
\(\zeta(u-\varepsilon)_+\) is supported in
\(\{u\ge\varepsilon\}\), hence is also an admissible test
function for \eqref{eq.fd-interior-equation}. Since
\(\nabla(u-\varepsilon)_+=\nabla u\) on \(\{u>\varepsilon\}\) and
\(\nabla(u-\varepsilon)_+=0\) elsewhere, we obtain
\[
	\int_{\{u>\varepsilon\}}a|\nabla u|^2\zeta\,\dx{x}
	=-\int a(u-\varepsilon)_+\nabla u\cdot\nabla\zeta\,\dx{x}.
\]
It follows from the monotone convergence theorem that the left-hand side 
increases to
\(\int_{\{u>0\}}a|\nabla u|^2\zeta\)
as \(\varepsilon\to0^+\). It follows from \((u-\varepsilon)_+\to u\) pointwise, \(0\le(u-\varepsilon)_+\le u\), 
the local boundedness of \(a\) and \(u\), and 
\(\nabla u\in L^2_{\mathrm{loc}}\) that 
\[
	\bigl|a(u-\varepsilon)_+\nabla u\cdot\nabla\zeta\bigr|
	\le C|\nabla u|\,|\nabla\zeta|\in L^1(B_{r_0}).
\]
Therefore it follows from the dominated convergence theorem
that 
\[
	-\int a(u-\varepsilon)_+\nabla u\cdot\nabla\zeta\,\dx{x}\to-\int au\nabla u\cdot\nabla\zeta,
\]
as \(\varepsilon\to 0^+\).

Using again \(\nabla u=0\) a.e. on \(\{u=0\}\), we may replace
\(\{u>0\}\) by \(B_{r_0}\) in the left-hand side, which gives
\begin{equation}\label{eq.fd-product-test-nonneg}
	\int a|\nabla u|^2\zeta\,\dx{x}
	=-\int au\nabla u\cdot\nabla\zeta\,\dx{x},
	\qquad \zeta\in C_c^1(B_{r_0}),\ \zeta\ge0.
\end{equation}
It should be noted that the argument above requires the sign \(\zeta\ge0\). The identity \eqref{eq.fd-product-test-nonneg} extends to signed test functions in the following way. Given \(\zeta\in C_c^1(B_{r_0})\), choose
\(\varphi\in C_c^\infty(B_{r_0})\) with \(\varphi\ge|\zeta|\). It follows that
\[
	\zeta=(\varphi+\zeta)-\varphi,
\]
where \(\varphi+\zeta\ge0\) and \(\varphi\ge0\) both belong to
\(C_c^1(B_{r_0})\). Applying \eqref{eq.fd-product-test-nonneg} to \(\varphi+\zeta\) and to
\(\varphi\), and subtracting the two resulting identities, gives
\[
	\int a|\nabla u|^2\zeta\,\dx{x}
	=-\int au\nabla u\cdot\nabla\zeta\,\dx{x}.
\]
Hence \eqref{eq.fd-product-test} holds for every
\(\zeta\in C_c^1(B_{r_0})\).

The energy identity \eqref{eq.energy-identity} gives
\eqref{eq.fd-base-energy}. To obtain \eqref{eq.fd-base-Pohozaev},
we use \eqref{eq.JF-boundary} in the form
\[
	J_F(u;\partial B_r)=\frac{\dx{}}{\dx{r}}J_F(u;B_r).
\]
Taking \(x_n^0=0\) in \eqref{eq.pohozaev-energy-identity}
then yields \eqref{eq.fd-base-Pohozaev}.
\end{proof}

We now introduce
\begin{equation}\label{eq.fd-coefficients}
	b(s):=\frac1{H(0;s)},\qquad\text{ and }\qquad d(x):=a(x)-b(x_n),
\end{equation}
where \(a(x):=1/H(|\nabla u(x)|^2;x_n)\) is defined in \eqref{eq.notations}. Set
\begin{equation*}
	\mathcal A(t;s):=\frac{t}{H(t;s)}-F(t;s)=\int_0^t\tau\,\partial_1\!\left(\frac1H\right)(\tau;s)\,\dx{\tau}.
\end{equation*}
It follows from Lemma~\ref{lem.preliminary} and the growth condition \eqref{eq.fd-growth}
that \(\partial_1H<0\), \(\partial_2H>0\), and 
\begin{equation}\label{eq.fd-elementary-coefficients}
	\begin{gathered}
		c\le a,b\le C,\quad b'(s)<0,\quad 0\le d(x)\le C|\nabla u(x)|^2,\\
		ct^2\le\mathcal A(t;s)\le Ct^2,\\
		|a(x)-a_0|+|b(x_n)-a_0|\le Cr
		\quad\text{for a.e. }x\in B_r(x^0).
	\end{gathered}
\end{equation}

\begin{remark}
	The last estimate in \eqref{eq.fd-elementary-coefficients} follows from \eqref{eq.fd-growth} and smoothness of the local branch. Note that \(d=0\) and \(\mathcal A(|\nabla u|^2;x_n)=0\) on the set \(\{u=0\}\). For a.e. \(x'\in\R^{n-1}\) and every \(x_n>0\) such that \(\{x'\}\times[0,x_n]\subset B_{r_0}(x^0)\), the function
	\(s\mapsto u(x',s)\) is absolutely continuous on \([0,x_n]\), and its
	trace at \(s=0\) vanishes. Hence it follows from \eqref{eq.fd-growth} that
	\[
		u(x',x_n)
		=\int_0^{x_n}\partial_n u(x',s)\,\dx{s}
		\le\sqrt{C_0}\int_0^{x_n}s^{1/2}\,\dx{s}
		=\frac23\sqrt{C_0}\,x_n^{3/2}.
	\]
	Since \(u\) is a subsonic variational solution, it is continuous. Using the continuity of the function \(x_n\), the estimate above extends from a.e. \(x'\) to every point as follows
	\[
		0\le u(x)\le\frac23\sqrt{C_0}(x_n^+)^{3/2}
		\quad\text{in }B_{r_0}(x^0).
	\]
\end{remark}
For \(0<r<r_0\), introduce the following quantities
\begin{equation}\label{eq.fd-radial-quantities}
	\begin{aligned}
		I(r)&:=\int_{B_r}a|\nabla u|^2\,\dx{x},\qquad
		P(r):=\int_{B_r}\mathcal A(|\nabla u|^2;x_n)\,\dx{x},\\
		J(r)&:=\int_{\partial B_r}b(x_n)u^2\,\dH{n-1},\qquad T(r):=r^{-n-2}J(r),\\
		G(r)&:=\int_{B_r}\lambda(x_n^+)(1-\chi_{\{u>0\}})\,\dx{x},\qquad\mathcal B_\lambda(r):=r^{-n-1}\int_{B_r}\lambda(x_n^+)\,\dx{x}.
	\end{aligned}
\end{equation}
Furthermore, it follows from the coarea formula and the bound
\(ct^2\le\mathcal A(t;s)\le Ct^2\) in
\eqref{eq.fd-elementary-coefficients} that
\begin{equation}\label{eq.fd-P-monotone}
	P'(r)=\int_{\partial B_r(x^0)}\mathcal A(|\nabla u|^2;x_n)\,\dH{n-1}\ge0
	\qquad\text{for a.e. }r\in(0,r_0).
\end{equation}

\begin{remark}
	The function \(b(s)\) defined in \eqref{eq.fd-coefficients} is the 
	reciprocal density of the zero-velocity hydrostatic state at height \(s\). 
	We address that its dependence on
	height alone permits us to differentiate \(J(r)\) given in \eqref{eq.fd-radial-quantities} using first
	derivatives of \(u\). The expansion of \(\lambda\) in \eqref{eq.lambda} gives
	\[
		\lambda(s)=\frac{s}{H_0}+O(s^2),
	\]
	and
	\begin{equation}\label{eq.fd-background-limit}
		\mathcal B_\lambda(r)
		=\frac{1}{H_0}r^{-n-1}\int_{B_r}x_n^+\,\dx{x}+O(r)
		=\Theta_0+O(r),
	\end{equation}
	where the last equality follows from \(\Theta_0
	=\frac1{H_0}\int_{B_1}x_n^+\,\dx{x}\) in \eqref{eq.fd-sets}.
\end{remark}

The renormalized Weiss monotonicity formula is
\begin{equation}\label{eq.fd-renormalized-Weiss}
	\begin{aligned}
		\mathcal M(r)&:=r^{-n-1}\bigl(I(r)-G(r)\bigr)-\frac32T(r)\\
		&=\Phi(r)-\mathcal B_\lambda(r)+r^{-n-1}P(r)+\frac32\bigl(W(r)-T(r)\bigr),
	\end{aligned}
\end{equation}
where the second equality follows from the identity
\[
	I(r)-G(r)=J_F(u;B_r)-r^{n+1}\mathcal B_\lambda(r)+P(r).
\]

\begin{remark}
	The growth condition \eqref{eq.fd-growth} and coefficient bounds in \eqref{eq.fd-elementary-coefficients} imply
	\[
		\begin{aligned}
			&I(r)\le Cr^{n+1},\qquad &&P(r)\le CrI(r),\\
			&0\le W(r),T(r)\le C,\qquad&&|W(r)-T(r)|\le CrW(r).
		\end{aligned}
	\]
	Indeed, \(\mathcal A(|\nabla u|^2;x_n)\le C|\nabla u|^4 \le Cr|\nabla u|^2\), while \(u^2\le Cr^3\) on \(\partial B_r\). Since \(\Phi(r)\to\Theta_0\) at the fixed point \(x^0\in\Sigma_u\), \eqref{eq.fd-background-limit} and \eqref{eq.fd-renormalized-Weiss} yield
	\begin{equation}\label{eq.fd-M-zero}
		\lim_{r\to0^+}\mathcal M(r)=0,\qquad 0\le T(r)\le C.
	\end{equation}
\end{remark}

\begin{lemma}\label{lem.fd-height-positive}
Let \(u\) be a
subsonic variational solution of \eqref{eq.quaslinear-free-boundary-problem2}, let \(0\in\Sigma_u\). 
Then the functions
\begin{equation}\label{eq.fd-D-V}
	D(r):=\frac{rI(r)}{J(r)},\qquad
	V(r):=\frac{rG(r)}{J(r)}\ge0.
\end{equation}
are well defined and locally absolutely continuous. Moreover,
\begin{equation}\label{eq.fd-Weiss-frequency-relation}
	\mathcal M(r)=T(r)\left(D(r)-\frac32-V(r)\right).
\end{equation}
\end{lemma}

\begin{proof}
We first prove that \(J(r)>0\) for all \(0<r<r_0\).
If \(J(r)=0\), it follows from the continuity of \(u\) and positivity of \(b\) that
\(u=0\) on \(\partial B_r\). Note that the support of
\((u-\varepsilon)_+\) in \(B_r\) is compactly contained in
\(B_r\cap\{u>0\}\) for \(\varepsilon>0\). Therefore we may test the 
equation \eqref{eq.fd-interior-equation} with this
function, which gives
\[
	\int_{\{u>\varepsilon\}}a|\nabla u|^2\,\dx{x}=0,
\]
and hence \(\nabla u=0\) a.e. on \(\{u>\varepsilon\}\). Passing to the limit as 
\(\varepsilon\to0^+\) gives \(u=0\) in \(B_r\), which gives a contradiction to
\(x^0\in\partial\{u>0\}\).

By the coarea formula, the volume integrals are locally absolutely
continuous. Writing
\[
	J(r)=r^{n-1}\int_{\partial B_1}
	b(r\omega_n)\,u^2(r\omega)\,\dH{n-1}(\omega),
\]
we see that \(J\) is locally absolutely continuous as well. Since \(J>0\), the
quotients \(D=rI/J\) and \(V=rG/J\) inherit this property on every
compact subinterval of \((0,r_0)\). 

It then follows from \(D=rI/J\) and \(V=rG/J\) that
\(I=JD/r\) and \(G=JV/r\). Substituting into
\eqref{eq.fd-renormalized-Weiss}, we compute directly to find that
\[
	\begin{aligned}
		\mathcal M(r)&=r^{-n-1}(I(r)-G(r))-\frac32T(r)\\
		&=r^{-n-2}J(r)(D(r)-V(r))-\frac32r^{-n-2}J(r)\\
		&=T(r)\left[D(r)-\frac32-V(r)\right].
	\end{aligned}
\]
\end{proof}

\subsection{The modified Weiss monotonicity formula}

In this section, we derive the modified Weiss type monotonicity formula. Define 
\begin{equation}\label{eq.fd-volume-remainder}
	\begin{split}
		R_{\rm kin}(r):={}&-nP(r)+\int_{B_r\cap\{u>0\}}x_n\partial_2F(|\nabla u|^2;x_n)\,\dx{x}\\
		&-\int_{B_r\cap\{u=0,\ x_n>0\}}\bigl(x_n\lambda'(x_n)-\lambda(x_n)\bigr)\,\dx{x},
	\end{split}
\end{equation}
where \(\{u=0,x_n>0\}\) denotes the set \(\{u=0\}\cap\{x_n>0\}\). Recall the definition of \(P(r)\) in \eqref{eq.fd-radial-quantities} and
the bound \(ct^2\le\mathcal A(t;s)\le Ct^2\) in
\eqref{eq.fd-elementary-coefficients}, we obtain \(P(r)\ge0\). Since
\(\lambda''(s)=H_0^{-2}\partial_2H(0;s)>0\) and \(\lambda(0)=0\), we have
\[
	s\lambda'(s)-\lambda(s)
	=\int_0^s\tau\lambda''(\tau)\,\dx{\tau}
	=s^2\int_0^1\theta\lambda''(\theta s)\,\dx{\theta}\ge0.
\]
Since \(\lambda''\) is bounded and \(\lambda(s)\ge cs\) for small
\(s\ge0\), we may deduce that \(s\lambda'(s)-\lambda(s)\le Cs^2\le Cs\lambda(s)\).
It then follows from \(\partial_2F(t;s)\le0\) that all three terms
in \eqref{eq.fd-volume-remainder} are nonpositive, which implies
\(R_{\rm kin}(r)\le0\).

We now derive a lower bound for \(R_{\rm kin}\) in \eqref{eq.fd-volume-remainder}. 
First recall the estimates \(\mathcal A(t;s)\le Ct^2\),
\(|\partial_2F(t;s)|\le Ct\) and
\(s\lambda'(s)-\lambda(s)\le Cs\lambda(s)\). 
Since \(t=|\nabla u|^2\le C_0x_n^+\le C_0r\) on \(B_r\), we have 
\begin{equation}\label{eq.fd-remainder-bounds}
	P(r)\le CrI(r),
	\qquad
	-\int_{B_r\cap\{u>0\}}x_n\partial_2F(|\nabla u|^2;x_n)\,\dx{x}
	\le CrI(r).
\end{equation}
It then follows from \(x_n\le r\) and the definition of \(G(r)\) in \eqref{eq.fd-radial-quantities} that
\[
	-\int_{B_r\cap\{u=0,\,x_n>0\}}
	\bigl(x_n\lambda'(x_n)-\lambda(x_n)\bigr)\,\dx{x}
	\ge -CrG(r).
\]
Substituting these three bounds into \eqref{eq.fd-volume-remainder}
gives \(R_{\rm kin}(r)\ge -Cr(I(r)+G(r))\). This together with
\(R_{\rm kin}(r)\le0\) implies that
\begin{equation}\label{eq.fd-remainder-relative}
	-Cr\bigl(I(r)+G(r)\bigr)\le R_{\rm kin}(r)\le0.
\end{equation}
For a.e. \(r\), define the following two boundary errors:
\begin{equation}\label{eq.fd-boundary-errors}
	\begin{aligned}
		\xi(r)&:=\frac r{J(r)}\int_{\partial B_r}d(x)\,u\nabla u\cdot\nu\,\dH{n-1},\\
		\vartheta(r)&:=\frac1{J(r)}\int_{\partial B_r}x_nb'(x_n)u^2\,\dH{n-1},
	\end{aligned}
\end{equation}
where \(d(x)\) and \(b(x_n)\) are given in \eqref{eq.fd-coefficients}. It follows from the energy identity \eqref{eq.energy-identity} that
\begin{equation}\label{eq.fd-D-minus-xi}
	\frac r{J(r)}\int_{\partial B_r}b(x_n)\,u\nabla u\cdot\nu\,\dH{n-1}=D(r)-\xi(r).
\end{equation}
Moreover, we deduce from the sign of \(b'(x_n)\) that
\begin{equation}\label{eq.fd-vartheta-bound}
	-Cr\le\vartheta(r)\le0.
\end{equation}
Differentiating \(J(r)\) in \eqref{eq.fd-radial-quantities} with respect to \(r\) gives
\begin{equation}\label{eq.fd-J-derivative}
	J'(r)=\frac{n-1}{r}J(r)+2\int_{\partial B_r}b(x_n)\,u\nabla u\cdot\nu\,\dH{n-1}+\frac{\vartheta(r)}rJ(r).
\end{equation}
Since \(T(r)=r^{-n-2}J(r)\), we may rewrite this as
\begin{equation}\label{eq.fd-height-derivative}
	\frac{T'(r)}{T(r)}=\frac2r\left(D(r)-\frac32-\xi(r)\right)+\frac{\vartheta(r)}r.
\end{equation}
We then deduce the following
formulas for the derivatives of \(I-G\) and \(\mathcal M\) defined in \eqref{eq.fd-renormalized-Weiss}.

\begin{lemma}
Let \(u\) be a
subsonic variational solution of \eqref{eq.quaslinear-free-boundary-problem2}, let \(0\in\Sigma_u\). Then for almost every sufficiently small \(r>0\),
\begin{equation}\label{eq.fd-L-identity}
	\begin{split}
		&r(I(r)-G(r))'-(n-2)(I(r)-G(r))\\
		={}&2r\int_{\partial B_r}a(x)(\nabla u\cdot\nu)^2\,\dH{n-1}
		+rP'(r)-3G(r)+R_{\mathrm{kin}}(r),
	\end{split}
\end{equation}
where \(R_{\mathrm{kin}}(r)\) is defined in \eqref{eq.fd-volume-remainder}. 
Consequently,
\begin{equation}\label{eq.fd-M-identity}
	\begin{split}
		\mathcal M'(r)={}&2r^{-n-1}\int_{\partial B_r}a(x)\delta_r^2\,\dH{n-1}
		+3r^{-n-2}\int_{\partial B_r}d(x)\,u\delta_r\,\dH{n-1}\\
		&+r^{-n-1}P'(r)+r^{-n-2}R_{\mathrm{kin}}(r)
		-\frac32\frac{\vartheta(r)}rT(r),
	\end{split}
\end{equation}
where \(\delta_r:=\nabla u\cdot\nu-\frac{3u}{2r}\) is defined
in~\eqref{eq.delta_r}.
\end{lemma}

\begin{proof}
Since \(\lambda(0)=0\), we apply the divergence theorem to the function
\(x\lambda(x_n^+)\), which gives
\[
	r\frac{\dx{}}{\dx{r}}\int_{B_r}\lambda(x_n^+)\,\dx{x}
	=n\int_{B_r}\lambda(x_n^+)\,\dx{x}
	+\int_{B_r\cap\{x_n>0\}}x_n\lambda'(x_n)\,\dx{x}.
\]
Subtracting this identity from \eqref{eq.fd-base-Pohozaev} and using
\[
	J_F(u;B_r)-\int_{B_r}\lambda(x_n^+)\,\dx{x}=I(r)-P(r)-G(r),
\]
we obtain
\begin{equation}\label{eq.fd-L-precursor}
	\begin{split}
		&r(I(r)-P(r)-G(r))'\\
		={}&n(I(r)-P(r)-G(r))-2I(r)
		+2r\int_{\partial B_r}a(x)(\nabla u\cdot\nu)^2\,\dH{n-1}\\
		&+\int_{B_r\cap\{u>0\}}x_n\partial_2F(|\nabla u|^2;x_n)\,\dx{x}
		-\int_{B_r\cap\{u=0,x_n>0\}}x_n\lambda'(x_n)\,\dx{x}.
	\end{split}
\end{equation}
It follows from the definition of \(G(r)\) in \eqref{eq.fd-radial-quantities} and
\(\lambda(0)=0\) that
\[
	G(r)=\int_{B_r\cap\{u=0,x_n>0\}}\lambda(x_n)\,\dx{x}.
\]
The last two terms in \eqref{eq.fd-L-precursor} combine with
\(R_{\mathrm{kin}}\) as
\[
	\begin{aligned}
		&\int_{B_r\cap\{u>0\}}x_n\partial_2F(|\nabla u|^2;x_n)\,\dx{x}
		-\int_{B_r\cap\{u=0,x_n>0\}}x_n\lambda'(x_n)\,\dx{x}\\
		={}&R_{\mathrm{kin}}(r)+nP(r)-G(r).
	\end{aligned}
\]
Inserting this into \eqref{eq.fd-L-precursor} yields
\eqref{eq.fd-L-identity}. Now differentiating 
\[
	\mathcal M(r)=r^{-n-1}(I(r)-G(r))-\frac{3T(r)}2
\] 
and using \eqref{eq.fd-L-identity}, we obtain 
\[
	\begin{aligned}
		\mathcal M'(r)={}&-3r^{-n-2}I(r)
		+2r^{-n-1}\int_{\partial B_r}a(x)(\nabla u\cdot\nu)^2\,\dH{n-1}\\
		&+r^{-n-1}P'(r)+r^{-n-2}R_{\mathrm{kin}}(r)-\frac32T'(r).
	\end{aligned}
\]
To rewrite the boundary term, note that
\[
	(\nabla u\cdot\nu)^2
	=\delta_r^2+\frac{3u}{r}\delta_r+\frac{9u^2}{4r^2}.
\]
Multiplying by \(a\) and integrating over \(\partial B_r\) gives
\[
	\begin{aligned}
		\int_{\partial B_r}a(\nabla u\cdot\nu)^2\,\dH{n-1}
		={}&\int_{\partial B_r}a\delta_r^2\,\dH{n-1}\\
		&+\frac{3}{r}\int_{\partial B_r}au\delta_r\,\dH{n-1}
		+\frac{9}{4r^2}\int_{\partial B_r}au^2\,\dH{n-1}.
	\end{aligned}
\]
It follows from \eqref{eq.fd-base-energy} that
\[
	\frac{3}{r}\int_{\partial B_r}au\delta_r\,\dH{n-1}
	=\frac{3}{r}I(r)-\frac{9}{2r^2}\int_{\partial B_r}au^2\,\dH{n-1},
\]
and hence
\[
	\int_{\partial B_r}a(\nabla u\cdot\nu)^2\,\dH{n-1}
	=\int_{\partial B_r}a\delta_r^2\,\dH{n-1}
	+\frac{3}{r}I(r)
	-\frac{9}{4r^2}\int_{\partial B_r}au^2\,\dH{n-1}.
\]
Using
\[
	\int_{\partial B_r}a(x)u^2\,\dH{n-1}
	=J(r)+\int_{\partial B_r}d(x)\,u^2\,\dH{n-1},
\]
together with \(I(r)=J(r)D(r)/r\) and
\eqref{eq.fd-height-derivative}, the terms containing \(D(r)\)
cancel, and the remaining terms are
\[
	\begin{aligned}
		&\frac{3\xi(r)T(r)}{r}
		-\frac32\frac{\vartheta(r)}{r}T(r)
		-\frac92r^{-n-3}\int_{\partial B_r}d(x)\,u^2\,\dH{n-1}\\
		&=3r^{-n-2}\int_{\partial B_r}d(x)\,u\delta_r\,\dH{n-1}
		-\frac32\frac{\vartheta(r)}{r}T(r).
	\end{aligned}
\]
Substituting this back gives \eqref{eq.fd-M-identity}.
\end{proof}

Fix \(0<\beta<1\), put \(\eta(r)=r^\beta\), let us define the
modified Weiss functional and the modified frequency formula by
\begin{equation}\label{eq.fd-Psi}
	\begin{aligned}
		\Psi_\beta(r)&:=\mathcal M(r)+r^\beta T(r),\\
		\mathsf H_\beta(r)&:=\frac{\Psi_\beta(r)}{T(r)}
		=D(r)-\frac32-V(r)+r^\beta.
	\end{aligned}
\end{equation}
We can also write \(\Psi_\beta(r)\) as
\begin{equation*}
	\begin{split}
		\Psi_\beta(r)={}&\Phi(r)-\Phi(0^+)-\bigl(\mathcal B_\lambda(r)-\Theta_0\bigr)+r^{-n-1}P(r)\\
		&+\frac32\bigl(W(r)-T(r)\bigr)+r^\beta T(r).
	\end{split}
\end{equation*}

\begin{proposition}[Modified Weiss-type monotonicity formula]\label{prop.fd-modified-Weiss}
Let \(u\) be a
subsonic variational solution of \eqref{eq.quaslinear-free-boundary-problem2}, 
let \(0\in\Sigma_u\). Assume that \(u\) satisfies the growth condition \eqref{eq.fd-growth}
on \(B_{r_0}\) with a constant
\(C_0\). Then for each \(0<\beta<1\), there are \(C>0\) and
\(r_\beta\in(0,r_0)\) such that
\begin{equation}\label{eq.fd-modified-Weiss}
	\begin{aligned}
		\Psi_\beta'(r)\ge{}&\left(2r^{\beta-1}-C\right)\Psi_\beta(r)\\
		&+r^{-n-1}\int_{\partial B_r}a(x)\delta_r^2\,\dH{n-1}
		+\frac12r^{-n-1}P'(r),
	\end{aligned}
\end{equation}
for almost every \(r\in(0,r_\beta)\). In particular,
\begin{equation}\label{eq.fd-Psi-monotone}
	r\longmapsto\exp\left(Cr-\frac{2r^\beta}{\beta}\right)
	\Psi_\beta(r)
	\quad\text{is nondecreasing}.
\end{equation}
and consequently \(\Psi_\beta(r)\ge0\) and
\(\mathsf H_\beta(r)\ge0\).
\end{proposition}

\begin{proof}
Recall \eqref{eq.fd-Psi}, we have \(\Psi_\beta(r)=\mathcal M(r)+\eta(r)T(r)\). It then follows from \eqref{eq.fd-height-derivative}, and \eqref{eq.fd-Weiss-frequency-relation} that
\[
	\eta(r)T'(r)-\frac{2\eta(r)}{r}\mathcal M(r)
	=\frac{2\eta(r)}{r}(V(r)-\xi(r))T(r)
	+\frac{\eta(r)\vartheta(r)}{r}T(r),
\]
where \(\xi(r)\) and \(\vartheta(r)\) are defined in
\eqref{eq.fd-boundary-errors}. Writing \(\nabla u\cdot\nu=\delta_r+3u/(2r)\) in the term containing
\(\xi(r)\) (cf. equation \eqref{eq.fd-boundary-errors}), we obtain the identity
\[
	\begin{aligned}
		&3r^{-n-2}\int_{\partial B_r}d(x)\,u\delta_r\,\dH{n-1}
		-\frac{2\eta(r)\xi(r)}{r}T(r)\\
		={}&(3-2\eta(r))r^{-n-2}\int_{\partial B_r}d(x)\,u\delta_r\,\dH{n-1}\\
		&-3\eta(r)r^{-n-3}\int_{\partial B_r}d(x)u^2\,\dH{n-1}.
	\end{aligned}
\]
Combining \eqref{eq.fd-M-identity} with the two identities above, we
obtain 
\begin{equation}\label{eq.fd-Psi-derivative}
	\begin{split}
		&\Psi_\beta'(r)-\frac{2\eta(r)}{r}\Psi_\beta(r)\\
		={}&2r^{-n-1}\int_{\partial B_r}a(x)\delta_r^2\,\dH{n-1}\\
		&+(3-2\eta(r))r^{-n-2}\int_{\partial B_r}d(x)\,u\delta_r\,\dH{n-1}\\
		&+r^{-n-1}P'(r)+r^{-n-2}R_{\mathrm{kin}}(r)
		+\frac{2\eta}{r}V(r)T(r)\\
		&-3\eta(r)r^{-n-3}\int_{\partial B_r}d(x)u^2\,\dH{n-1}\\
		&+\left(\eta'(r)-\frac{2\eta^2(r)}{r}\right)T(r)
		+\left(\eta(r)-\frac32\right)\frac{\vartheta(r)}{r}T(r).
	\end{split}
\end{equation}
Since \(\eta(r)\le1\) and
\(\vartheta(r)\le0\), we deduce that the last term in \eqref{eq.fd-Psi-derivative} is 
nonnegative. It follows from \(d(x)\le C|\nabla u|^2\), \(a(x)\ge c\),
\(u^2\le Cr^3\), and \(\mathcal A\ge c|\nabla u|^4\) on
\(\partial B_r\) that
\[
	\frac{d^2(x)u^2}{a(x)}\le Cr^3\mathcal A(|\nabla u|^2;x_n).
\]
Using Young's inequality, we have
\begin{equation}\label{eq.fd-Young-delta}
	\begin{split}
		&\left|(3-2\eta(r))r^{-n-2}
		\int_{\partial B_r}d(x)\,u\delta_r\,\dH{n-1}\right|\\
		\le{}& r^{-n-1}\int_{\partial B_r}a(x)\delta_r^2\,\dH{n-1}
		+Cr^{-n-3}\int_{\partial B_r}\frac{d^2(x)u^2}{a(x)}\,\dH{n-1}\\
		\le{}& r^{-n-1}\int_{\partial B_r}a(x)\delta_r^2\,\dH{n-1}
		+Cr\cdot r^{-n-1}P'(r).
	\end{split}
\end{equation}
Using \(d(x)\le Crb(x_n)\), we have
\begin{equation}\label{eq.fd-d-u2-bound}
	3\eta(r)r^{-n-3}\int_{\partial B_r}d(x)u^2\,\dH{n-1}
	\le C\eta(r)T(r).
\end{equation}
Recall the volume remainder \(R_{\rm kin}\) defined in \eqref{eq.fd-volume-remainder}, 
we infer from \eqref{eq.fd-remainder-relative} that
\begin{equation}\label{eq.fd-remainder-lower}
	\begin{aligned}
		r^{-n-2}R_{\mathrm{kin}}(r)&\ge-C_1(D(r)+V(r))T(r)\\
		&=-C_1\Psi_\beta(r)-C_1(3/2+2V(r)-\eta(r))T(r).
	\end{aligned}
\end{equation}
Combining \eqref{eq.fd-Psi-derivative} with
\eqref{eq.fd-Young-delta}, \eqref{eq.fd-d-u2-bound}, and
\eqref{eq.fd-remainder-lower} together, and choosing \(r_\beta>0\) small enough
that \(1-Cr\ge1/2\), we conclude that there are constants \(C_2,C_3\), depending only on
\(C_0,C_1,n\), such that for a.e. \(0<r<r_\beta\),
\begin{equation}\label{eq.fd-Psi-lower}
	\begin{split}
		&\Psi_\beta'(r)-\left(\frac{2\eta(r)}{r}-C_1\right)\Psi_\beta(r)\\
		&\quad\ge r^{-n-1}\int_{\partial B_r}a(x)\delta_r^2\,\dH{n-1}
		+\frac12r^{-n-1}P'(r)\\
		&\quad+\left(\frac{2\eta(r)}{r}-C_2\right)V(r)T(r)\\
		&\quad+\left(\eta'(r)-\frac{2\eta^2(r)}{r}-C_3(1+\eta(r))\right)T(r).
	\end{split}
\end{equation}
It follows from \eqref{eq.fd-P-monotone} that \(r^{-n-1}P'(r)\ge 0\).
Since \(V(r),T(r)\ge0\), the last two terms in
\eqref{eq.fd-Psi-lower} are nonnegative as soon as their coefficients
are. The first coefficient satisfies
\[
	\frac{2\eta(r)}{r}=2r^{\beta-1}\longrightarrow+\infty,
\]
while the second one satisfies
\[
	\eta'(r)-\frac{2\eta^2(r)}{r}-C_3(1+\eta(r))
	=r^{\beta-1}\bigl[\beta-2r^\beta-C_3r^{1-\beta}(1+r^\beta)\bigr],
\]
and since \(r^{1-\beta}\to0\), we have that the bracket tends to \(\beta>0\) as 
\(r\to0^+\). Hence both coefficients are nonnegative for all
sufficiently small \(r\), and the last two terms in
\eqref{eq.fd-Psi-lower} may be discarded.
Discarding the last two terms in \eqref{eq.fd-Psi-lower} and using
\eqref{eq.fd-P-monotone}, we obtain \eqref{eq.fd-modified-Weiss}. Note that we also
obtain the inequality
\begin{equation}\label{eq.fd-Psi-differential}
	\Psi_\beta'(r)\ge\left(2r^{\beta-1}-C_1\right)\Psi_\beta(r)
	\qquad\text{for a.e. }0<r<r_\beta.
\end{equation}
Multiplying \eqref{eq.fd-Psi-differential} by
\(\exp(C_1r-2r^\beta/\beta)\) gives
\[
	\frac{\dx{}}{\dx{r}}
	\left[\exp\left(C_1r-\frac{2r^\beta}{\beta}\right)\Psi_\beta(r)\right]
	\ge0,
\]
which proves \eqref{eq.fd-Psi-monotone}.

Thanks to \eqref{eq.fd-M-zero}, we have \(\mathcal M(r)\to0\) and
\(0\le T(r)\le C\). Since \(\eta(r)=r^\beta\to0\) as \(r\to 0^+\), we may
deduce from \eqref{eq.fd-Psi} that \(\Psi_\beta(r)\to0\) as \(r\to0^+\).
Integrating \eqref{eq.fd-Psi-differential} from \(\varepsilon\) to
\(r\) yields
\[
	\exp\left(Cr-\frac{2r^\beta}{\beta}\right)\Psi_\beta(r)
	\ge\exp\left(C\varepsilon-\frac{2\varepsilon^\beta}{\beta}\right)
	\Psi_\beta(\varepsilon),
\]
and letting \(\varepsilon\to0^+\) gives \(\Psi_\beta(r)\ge0\).
Finally, the nonnegativity of \(\mathsf H_\beta(r)\) follows from
\eqref{eq.fd-Psi} and the positivity \(T(r)>0\).
\end{proof}

Let us remark that the modified Weiss-type monotonicity formula for \(\Psi_\beta(r)\) also holds for variable centers.

\begin{lemma}\label{lem.fd-uniform-moving-centers}
Let \(u\) be a subsonic variational solution of
\eqref{eq.quaslinear-free-boundary-problem2}, let \(K\subset\Sigma_u\) be
a compact subset. Assume that the growth condition \eqref{eq.fd-growth} holds 
in a neighborhood of \(K\) with a common constant \(C_0\). Then for every \(0<\beta<1\),
there exist \(r_\beta\in(0,r_0)\) and \(C>0\), independent of \(y\in K\),
such that 
\begin{equation*}
	\begin{aligned}
		\Psi_{\beta,y}'(r)\ge{}&
		 (2r^{\beta-1}-C)\Psi_{\beta,y}(r)
		 +r^{-n-1}\!\int_{\partial B_r(y)}
		     a(x)\delta_{r,y}^{2}\,\dH{n-1}\\
		 &+\frac12r^{-n-1}P_y'(r)
	\end{aligned}
\end{equation*}
for every \(y\in K\) and for a.e. \(r\in(0,r_\beta)\), where
\(\delta_{r,y}=\nabla u\cdot(x-y)/r-3u/(2r)\). Decreasing
\(r_\beta\) if necessary, one also has
\[
	\Psi_{\beta,y}(r)\ge0,\qquad
	D_y(r)\ge\frac32-r^\beta,
\]
and
\begin{equation}\label{eq.fd-uniform-moving-lower}
	\frac{r\int_{B_r(y)}|\nabla u|^2\,\dx{x}}
	{\int_{\partial B_r(y)}u^2\,\dH{n-1}}
	\ge(1-Cr)\left(\frac32-r^\beta\right)
\end{equation}
for all \(y\in K\) and \(0<r<r_\beta\).
\end{lemma}

\begin{remark}
In Lemma \ref{lem.fd-uniform-moving-centers}, we note that the subscript \(y\)
represents that the equations in
\eqref{eq.fd-radial-quantities}, \eqref{eq.fd-renormalized-Weiss},
Lemma~\ref{lem.fd-height-positive}, and \eqref{eq.fd-Psi}
are evaluated on balls centered at \(y\). 
\end{remark}

\begin{proof}
Let \(K\subset\Sigma_u\) be compact. Choose \(0<\rho<1\) such that \(B_{2\rho}(y)\subset\subset\Omega\) for
all \(y\in K\), the growth condition \eqref{eq.fd-growth} holds on
\(\bigcup_{y\in K}B_{2\rho}(y)\) with a single constant \(C_0\), and
\([0,2C_0\rho]\times[-2\rho,2\rho]\subset\subset\cR\). As in the proof of Proposition~\ref{prop.fd-modified-Weiss}, the
vertical integration implies \eqref{eq.bu-pointwise-growth}. Combining the
growth condition \eqref{eq.fd-growth} with the coefficient estimates in
Lemma~\ref{lem.preliminary}, we obtain \eqref{eq.fd-elementary-coefficients}
and \eqref{eq.fd-remainder-relative}, with constants independent of
\(y\in K\). A similar argument as in Lemma~\ref{lem.fd-height-positive} shows 
\(J_y>0\) and the local absolute continuity of equations on \((0,\rho)\) 
defined in \eqref{eq.fd-radial-quantities}, \eqref{eq.fd-renormalized-Weiss},
Lemma~\ref{lem.fd-height-positive}, and \eqref{eq.fd-Psi}.

Repeating the calculation of Proposition~\ref{prop.fd-modified-Weiss}, 
there are constants \(A_i\ge1\) that are independent
of \(y\) such that for a.e. \(r<\rho\),
\begin{equation}\label{eq.fd-uniform-Weiss}
	\begin{aligned}
		&\Psi_{\beta,y}'(r)-\bigl(2r^{\beta-1}-A_1\bigr)\Psi_{\beta,y}(r)\\
		&\quad\ge
		 r^{-n-1}\int_{\partial B_r(y)}a(x)\,\delta_{r,y}(x)^2\,\dH{n-1}\\
		&\qquad +(1-A_4r)r^{-n-1}P_y'(r)
		 +\bigl(2r^{\beta-1}-A_2\bigr)V_y(r)T_y(r)\\
		&\qquad +r^{\beta-1}
		 \bigl[\beta-2r^\beta-A_3r^{1-\beta}(1+r^\beta)\bigr]T_y(r).
	\end{aligned}
\end{equation}
For
\[
	r_\beta:=\min\Bigg\{\rho,\ \frac1{2A_4},\
	 A_2^{-1/(1-\beta)},\left(\frac\beta4\right)^{1/\beta},\
	 \left(\frac\beta{4A_3}\right)^{1/(1-\beta)}\Bigg\}
\]
and \(r<r_\beta\), the coefficients of \(P_y'\), \(V_yT_y\) and \(T_y\)
in \eqref{eq.fd-uniform-Weiss} are nonnegative, so discarding the 
last two terms gives the asserted differential inequality with \(C=A_1\).

The bounds \eqref{eq.fd-elementary-coefficients} and the pointwise
estimate \(u\le C(x_n^+)^{3/2}\) give \(T_y\le C\),
\(r^{-n-1}P_y\le Cr\) and \(|W_y-T_y|\le Cr\). Note that these estimates
are uniformly in \(y\). Since \(\mathcal B_\lambda\) defined in
\eqref{eq.fd-radial-quantities} is invariant under horizontal
translations, we infer from \eqref{eq.fd-background-limit},
\eqref{eq.fd-renormalized-Weiss}, and the uniform bounds
\(T_y\le C\), \(r^{-n-1}P_y\le Cr\) and \(|W_y-T_y|\le Cr\) that
\begin{equation}\label{eq.fd-Psi-expansion}
	\Psi_{\beta,y}(r)=\Phi(y,u;r)-\Theta_0+O(r+r^\beta).
\end{equation}
Since \(\Phi(y,u;0^+)=\Theta_0\) for \(y\in\Sigma_u\), the right-hand
side of \eqref{eq.fd-Psi-expansion} tends to zero as \(r\to0^+\), and
so does \(\Psi_{\beta,y}(r)\). The integrating factor
\(\exp(A_1r-2r^\beta/\beta)\) makes
\[
	r\longmapsto \exp\left(A_1r-\frac{2r^\beta}{\beta}\right)
	\Psi_{\beta,y}(r)
\]
nondecreasing, and this function tends to zero at the origin. It 
follows that \(\Psi_{\beta,y}(r)\ge0\). Thanks to \eqref{eq.fd-Psi} and
\eqref{eq.fd-Weiss-frequency-relation}, we have
\[
	D_y(r)=\frac32+V_y(r)-r^\beta+\frac{\Psi_{\beta,y}(r)}{T_y(r)}.
\]
Since \(\Psi_{\beta,y}\ge0\), \(V_y\ge0\) and \(T_y>0\), we obtain
\[
	D_y(r)\ge\frac32-r^\beta.
\]
Now recall \(a\) and \(b\) defined in \eqref{eq.notations} and
\eqref{eq.fd-coefficients}, respectively. It follows from 
\(|a-a_0|+|b-a_0|\le Cr\) in
\eqref{eq.fd-elementary-coefficients} that \(a\le a_0+Cr\) and
\(b\ge a_0-Cr\) hold uniformly in \(y\). After reducing \(r_\beta\) so that
\(a_0-Cr>0\), we obtain
\[
	\frac{r\int_{B_r(y)}|\nabla u|^2\,\dx{x}}
	{\int_{\partial B_r(y)}u^2\,\dH{n-1}}
	\ge\frac{a_0-Cr}{a_0+Cr}D_y(r)
	\ge(1-Cr)\left(\frac32-r^\beta\right),
\]
which proves \eqref{eq.fd-uniform-moving-lower}.
\end{proof}

\subsection{The Weiss-type modified frequency inequality}
Let \(u\) be a
subsonic variational solution of \eqref{eq.quaslinear-free-boundary-problem2}. 
Assume that \(u\) satisfies the growth condition \eqref{eq.fd-growth}.
For a.e. \(r\), we introduce the functions
\begin{equation}\label{eq.fd-Q-S}
	\begin{aligned}
		\mathcal Q(r)&:=\frac1{J(r)}\int_{\partial B_r}a(x)\left[r\nabla u\cdot\nu-D(r)u\right]^2\,\dH{n-1},\\
		\mathcal S(r)&:=\frac{r^2P'(r)}{J(r)}\ge0.
	\end{aligned}
\end{equation}

The next proposition controls the modified frequency \(\mathsf H_\beta\)
defined in \eqref{eq.fd-Psi} and the quotient \(V\) defined in
\eqref{eq.fd-D-V}.

\begin{proposition}[Modified frequency inequality]\label{prop.fd-frequency}
Let \(u\) be a subsonic variational solution of
\eqref{eq.quaslinear-free-boundary-problem2} with \(0\in\Sigma_u\), and
assume that \eqref{eq.fd-growth} holds on \(B_{r_0}\) with a constant
\(C_0\). There are a constant \(C>0\) and, after a possible reduction,
a radius \(r_\beta>0\) such that
\begin{equation}\label{eq.fd-frequency-inequality}
	\mathsf H_\beta'(r)+C\mathsf H_\beta(r)
	\ge\frac{\mathcal Q(r)}{r}
	+\frac{V(r)\mathsf H_\beta(r)}{r}
	+\frac{V^2(r)}{r}
\end{equation}
for a.e. \(r\in(0,r_\beta)\). Thus
\begin{equation}\label{eq.fd-Hbeta-monotone}
	r\longmapsto e^{Cr}\mathsf H_\beta(r)
	\quad\text{is nonnegative and nondecreasing}.
\end{equation}
In particular, the limit \(\mathsf H(0^+):=\mathsf H_\beta(0^+)\)
exists, is finite and nonnegative, and is independent of the choice
of \(\beta\).
\end{proposition}

\begin{proof}
By Proposition~\ref{prop.fd-modified-Weiss}, \(\mathsf H_\beta(r)\ge0\),
where \(\mathsf H_\beta\) is defined in \eqref{eq.fd-Psi}. Recall also
that \(V(r)\ge0\) by \eqref{eq.fd-D-V}. Thanks to \eqref{eq.fd-Psi},
\(\mathsf H_\beta(r)=D(r)-V(r)-\frac32+\eta(r)\). Hence, for all
sufficiently small \(r>0\),
\[
	0<\frac32-\eta(r)\le D(r)-V(r)\le D(r).
\]
Let us introduce the boundary errors
\begin{equation}\label{eq.fd-e-gamma}
	\begin{aligned}
		e(r)&:=\frac1{J(r)}\int_{\partial B_r}d(x)u^2\,\dH{n-1},\\
		\gamma(r)&:=\frac1{J(r)}\int_{\partial B_r}d(x)\,u(r\nabla u\cdot\nu-Du)\,\dH{n-1}.
	\end{aligned}
\end{equation}
It follows that \(0\le e(r)\le Cr\) and
\begin{equation}\label{eq.fd-square-algebra-proof-1}
	\begin{aligned}
		&\frac1{J(r)}\int_{\partial B_r}a(x)u^2\,\dH{n-1}=1+e(r),\\
		&\xi(r)=D(r)e(r)+\gamma(r),\\
		&D(r)-\xi(r)=(1-e(r))D(r)-\gamma(r).
	\end{aligned}
\end{equation}
Thanks to the energy identity in \eqref{eq.fd-base-energy}, we have
\[
	\frac{r}{J(r)}\int_{\partial B_r}a(x)u(\nabla u\cdot\nu)\,\dH{n-1}
	=D(r).
\]
It then follows from the definition of \(\mathcal Q(r)\) in \eqref{eq.fd-Q-S} and a 
direct calculation that
\begin{equation*}
	\begin{split}
		\frac1{J(r)}\int_{\partial B_r}a(x)(r\nabla u\cdot\nu)^2\,\dH{n-1}
		&=\mathcal Q(r)+2D^2(r)-D^2(r)(1+e(r))\\
		&=\mathcal Q(r)+(1-e(r))D^2(r).
	\end{split}
\end{equation*}
Consequently,
\begin{equation*}
	\begin{split}
		&\frac1{J(r)}\int_{\partial B_r}a(x)(r\nabla u\cdot\nu)^2\,\dH{n-1}
		-(D(r)-V(r))(D(r)-\xi(r))\\
		={}&\mathcal Q(r)+(1-e(r))D(r)V(r)
		+(D(r)-V(r))\gamma(r),
	\end{split}
\end{equation*}
where \(\xi(r)\) is defined in \eqref{eq.fd-boundary-errors}.
It follows from \eqref{eq.fd-D-V} that
\[
	D(r)-V(r)=\frac{r(I(r)-G(r))}{J(r)}.
\]
A direct computation gives
\begin{equation}\label{eq.fd-DV-quotient}
	r(D(r)-V(r))'
	=\frac{r^2(I(r)-G(r))'}{J(r)}
	+(D(r)-V(r))\left(1-\frac{rJ'(r)}{J(r)}\right).
\end{equation}
It follows from \eqref{eq.fd-J-derivative} and
\eqref{eq.fd-D-minus-xi} that
\begin{equation}\label{eq.fd-J-log-derivative}
	\frac{rJ'(r)}{J(r)}
	=n-1+2(D(r)-\xi(r))+\vartheta(r).
\end{equation}
Substituting \eqref{eq.fd-L-identity} and using
\eqref{eq.fd-J-log-derivative}, we obtain
\begin{equation}\label{eq.fd-DV-derivative}
	\begin{split}
		r(D(r)-V(r))'
		={}&\frac2{J(r)}\int_{\partial B_r}a(x)(r\nabla u\cdot\nu)^2\,\dH{n-1}\\
		&-2(D(r)-V(r))(D(r)-\xi(r))\\
		&+\mathcal S(r)-3V(r)+\frac{rR_{\mathrm{kin}}(r)}{J(r)}
		-\vartheta(r)(D(r)-V(r)).
	\end{split}
\end{equation}
Therefore it follows from \eqref{eq.fd-Psi}, \(\eta(r)=r^\beta\) and a direct calculation
that
\begin{equation}\label{eq.rHbeta-prime}
	\begin{split}
		r\mathsf H_\beta'(r)={}&2\mathcal Q(r)
		+2(1-e(r))D(r)V(r)
		+2(D(r)-V(r))\gamma(r)\\
		&+\mathcal S(r)-3V(r)+\frac{rR_{\mathrm{kin}}(r)}{J(r)}
		-\vartheta(r)(D(r)-V(r))+r\eta'(r).
	\end{split}
\end{equation}
By the definition of \(\gamma\) in \eqref{eq.fd-e-gamma} and
\(u=\frac1{D(r)}\bigl[r\nabla u\cdot\nu-(r\nabla u\cdot\nu-D(r)u)\bigr]\),
we have
\begin{equation}\label{eq.fd-gamma-split}
	\begin{split}
		&(D(r)-V(r))\gamma(r)\\
		={}&\frac{D(r)-V(r)}{D(r)J(r)}
		\int_{\partial B_r}d(x)(r\nabla u\cdot\nu)
		(r\nabla u\cdot\nu-D(r)u)\,\dH{n-1}\\
		&-\frac{D(r)-V(r)}{D(r)J(r)}
		\int_{\partial B_r}d(x)(r\nabla u\cdot\nu-D(r)u)^2\,\dH{n-1}.
	\end{split}
\end{equation}
Note that \(0<(D(r)-V(r))/D(r)\le1\). By \eqref{eq.fd-Q-S},
\[
	\mathcal Q(r)=\frac1{J(r)}\int_{\partial B_r}a(x)
	\bigl(r\nabla u\cdot\nu-D(r)u\bigr)^2\,\dH{n-1}.
\]
For the second term in \eqref{eq.fd-gamma-split}, since
\(0<(D(r)-V(r))/D(r)\le1\) and \(d(x)/a(x)\le Cr\),
\[
	\begin{aligned}
		&-\frac{D(r)-V(r)}{D(r)J(r)}
		\int_{\partial B_r}d(x)
		\bigl(r\nabla u\cdot\nu-D(r)u\bigr)^2\,\dH{n-1}\\
		&\quad\ge-\frac{Cr}{J(r)}\int_{\partial B_r}a(x)
		\bigl(r\nabla u\cdot\nu-D(r)u\bigr)^2\,\dH{n-1}
		=-Cr\mathcal Q(r).
	\end{aligned}
\]
For the first term, we apply the Cauchy-Schwarz inequality and obtain
\begin{equation}\label{eq.fd-CS-gamma}
	\begin{aligned}
		&\left|\frac{D(r)-V(r)}{D(r)J(r)}
		\int_{\partial B_r}d(x)(r\nabla u\cdot\nu)
		\bigl(r\nabla u\cdot\nu-D(r)u\bigr)\,\dH{n-1}\right|\\
		&\quad\le
		\left(\frac1{J(r)}\int_{\partial B_r}
		\frac{d^2(x)(r\nabla u\cdot\nu)^2}{a(x)}\,\dH{n-1}\right)^{1/2}
		\mathcal Q(r)^{1/2}.
	\end{aligned}
\end{equation}
The growth condition \eqref{eq.fd-growth} and
\(\mathcal A\ge c|\nabla u|^4\) give
\[
	\frac{d^2(x)(r\nabla u\cdot\nu)^2}{a(x)}
	\le Cr^2|\nabla u|^6
	\le Cr^3\mathcal A(|\nabla u|^2;x_n).
\]
As \(P'(r)=\int_{\partial B_r}\mathcal A(|\nabla u|^2;x_n)\,\dH{n-1}\)
by \eqref{eq.fd-P-monotone} and
\(\mathcal S(r)=r^2P'(r)/J(r)\) by \eqref{eq.fd-Q-S}, we obtain
\[
	\frac1{J(r)}\int_{\partial B_r}
	\frac{d^2(x)(r\nabla u\cdot\nu)^2}{a(x)}\,\dH{n-1}
	\le Cr\,\mathcal S(r),
\]
so the first factor in \eqref{eq.fd-CS-gamma} is bounded by
\(\sqrt{Cr\mathcal S(r)}\). Therefore, we obtain
\[
	(D(r)-V(r))\gamma(r)
	\ge-\sqrt{Cr\mathcal S(r)\mathcal Q(r)}-Cr\mathcal Q(r).
\]
It follows from the Young inequality that
\[
	2\sqrt{Cr\mathcal S(r)\mathcal Q(r)}
	\le\frac12\mathcal S(r)+2Cr\mathcal Q(r)
\] 
Thus, we get
\[
	\begin{aligned}
		&2\mathcal Q(r)+2(D(r)-V(r))\gamma(r)+\mathcal S(r)\\
		&\ge(2-4Cr)\mathcal Q(r)+\tfrac12\mathcal S(r)
		\ge\mathcal Q(r)+\tfrac12\mathcal S(r).
	\end{aligned}
\]
This together with the inequality in 
\eqref{eq.rHbeta-prime}, and the estimates
\(-\vartheta(D-V)\ge0\), \(e\le Cr\), and
\(\frac{rR_{\mathrm{kin}}}{J}\ge-Cr(D+V)\), yields
\[
	\begin{aligned}
		r\mathsf H_\beta'(r)&\ge\mathcal Q(r)+\frac12\mathcal S(r)\\ 
		&+2(1-Cr)D(r)V(r)-3V(r)-Cr(D(r)+V(r))+r\eta'(r).
	\end{aligned}
\]
Inserting \(D(r)=\mathsf H_\beta(r)+\frac32+V(r)-\eta(r)\) and
enlarging \(C\) if necessary yields
\[
	\begin{split}
		r\mathsf H_\beta'(r)\ge{}&\mathcal Q(r)+\frac12\mathcal S(r)
		+(2-Cr)V(r)\mathsf H_\beta(r)+(2-Cr)V^2(r)\\
		&-C(\eta(r)+r)V(r)-Cr\mathsf H_\beta(r)-Cr+r\eta'(r).
	\end{split}
\]
Choosing \(Cr\le1/2\) and using
\[
	C(\eta(r)+r)V(r)\le\tfrac12V^2(r)+C'(\eta(r)+r)^2,
\]
we get
\[
	(2-Cr)V^2(r)-C(\eta(r)+r)V(r)
	\ge V^2(r)-C'(\eta(r)+r)^2,
\]
and
\[
	(2-Cr)V(r)\mathsf H_\beta(r)\ge V(r)\mathsf H_\beta(r).
\]
Since
\begin{equation}\label{eq.fd-scalar-error}
	r\eta'(r)-Cr-C'(\eta(r)+r)^2
	=r^\beta\Bigl[\beta-Cr^{1-\beta}
	-C'\bigl(r^\beta+2r+r^{2-\beta}\bigr)\Bigr]
\end{equation}
and the bracket tends to \(\beta>0\) as \(r\to0^+\), the scalar error
\(r\eta'(r)-Cr-C'(\eta(r)+r)^2\) is nonnegative for all sufficiently
small \(r\). Discarding the positive term \(\mathcal S(r)/2\) and
dividing by \(r\) proves \eqref{eq.fd-frequency-inequality}.

Now multiplying the inequality \eqref{eq.fd-frequency-inequality} by \(e^{Cr}\) on both sides, we have
\[
	\frac{\dx{}}{\dx{r}}[e^{Cr}\mathsf H_\beta(r)]
	\ge e^{Cr}\frac{\mathcal Q(r)+V(r)\mathsf H_\beta(r)+V^2(r)}{r}\ge0.
\]
It follows that the function \(e^{Cr}\mathsf H_\beta(r)\) is nonnegative and
nondecreasing on \((0,r_\beta)\). Thus it has a finite limit as \(r\to0^+\). Finally,
Since
\[
	\mathsf H_{\beta_1}(r)-\mathsf H_{\beta_2}(r)=r^{\beta_1}-r^{\beta_2}\to0,
\]
we see that the limit is independent of \(\beta\).
\end{proof}

As a direct corollary of the modified frequency inequality established in 
Proposition \ref{prop.fd-frequency}, we obtain

\begin{corollary}\label{cor.fd-frequency-limits}
Under the assumptions of Proposition~\ref{prop.fd-frequency}, let
\(x^0\in\Sigma_u\) be the fixed center. Then
\begin{enumerate}
	\item 
	\begin{equation}\label{eq.fd-D-V-limits}
		V(r)\longrightarrow0,\qquad
		D(r)\longrightarrow N:=\frac32+\mathsf H(0^+)
		\in[3/2,\infty).
	\end{equation}
	\item 
	\begin{equation}\label{eq.fd-tail-integrability}
		\int_0^{r_1}\frac{\mathcal Q(r)+V^2(r)}{r}\,\dx{r}<\infty,
	\end{equation}
	for each \(r_1>0\) sufficiently small.
	\item There is a constant \(C>0\) such that 
	\begin{equation}\label{eq.fd-height-doubling}
		C^{-1}T(r)\le T(s)\le CT(r),
	\end{equation}
	whenever \(0<r\le r_1\) and \(r/2\le s\le r\). Here, \(T(r)\) is defined in \eqref{eq.fd-radial-quantities}.
	\item 
	\begin{equation}\label{eq.fd-unweighted-frequency}
		\frac{r\int_{B_r}|\nabla u|^2\,\dx{x}}
		{\int_{\partial B_r}u^2\,\dH{n-1}}\longrightarrow N.
	\end{equation}
\end{enumerate}
\end{corollary}

\begin{proof}
Integrating \eqref{eq.fd-frequency-inequality} from \(\varepsilon\) to
\(r\), discarding the nonnegative term
\(V(t)\mathsf H_\beta(t)/t\), and letting \(\varepsilon\to0^+\), we
obtain
\begin{equation}\label{eq.fd-tail-estimate}
	\int_0^r e^{Ct}\frac{\mathcal Q(t)+V^2(t)}{t}\,\dx{t}
	\le e^{Cr}\mathsf H_\beta(r)-\mathsf H(0^+)\longrightarrow0.
\end{equation}
Since \(e^{Ct}\) is bounded above and below on \((0,r_1]\), taking
\(r=r_1\) in \eqref{eq.fd-tail-estimate} gives
\eqref{eq.fd-tail-integrability}.

We now prove \eqref{eq.fd-height-doubling}. Applying the
Cauchy--Schwarz inequality to \eqref{eq.fd-e-gamma}, we obtain
\[
	\begin{aligned}
		|\gamma(r)|
		&\le\frac1{J(r)}
		\left(\int_{\partial B_r}d(x)u^2\,\dH{n-1}\right)^{1/2}\\
		&\qquad\times\left(\int_{\partial B_r}d(x)(r\nabla u\cdot\nu-D(r)u)^2\,\dH{n-1}\right)^{1/2}.
	\end{aligned}
\]
Since \(d(x)\le Crb(x_n)\) and \(d(x)\le Cra(x)\), we have 
\[
	\int_{\partial B_r}d(x)u^2\,\dH{n-1}\le CrJ(r),
\]
and 
\[
	\int_{\partial B_r}d(x)(r\nabla u\cdot\nu-D(r)u)^2\,\dH{n-1}
	\le CrJ(r)\mathcal Q(r).
\]
Thus \(|\gamma(r)|\le Cr\sqrt{\mathcal Q(r)}\), where
\(\mathcal Q(r)\) is defined in \eqref{eq.fd-Q-S}. From \(e(r)\le Cr\)
and the identity for \(\xi(r)\) in \eqref{eq.fd-square-algebra-proof-1},
we then get
\begin{equation}\label{eq.fd-xi-estimate}
	|\xi(r)|=|D(r)e(r)+\gamma(r)|
	\le Cr\bigl(D(r)+\sqrt{\mathcal Q(r)}\bigr).
\end{equation}
Since \(e^{Cr}\mathsf H_\beta(r)\) is nondecreasing, we get
\[
	\mathsf H_\beta(r)\le e^{C(r_1-r)}\mathsf H_\beta(r_1)
	\le e^{Cr_1}\mathsf H_\beta(r_1).
\]
Thus \(\mathsf H_\beta\) is bounded above on \((0,r_1]\).
Thanks to \eqref{eq.fd-Psi}, 
we have \(D(r)=\mathsf H_\beta(r)+\frac32+V(r)-\eta(r)\),
hence 
\begin{equation}\label{eq.fd-D-bound}
	D(r)\le C+V(r)\qquad(0<r\le r_1).
\end{equation}
It follows that 
\eqref{eq.fd-height-derivative} becomes
\[
	(\log T(r))'
	=\frac2r(\mathsf H_\beta(r)+V(r)-\eta(r))
	+\frac{\vartheta(r)-2\xi(r)}{r},
\]
where \(\vartheta(r)\) and \(\xi(r)\) are defined in
\eqref{eq.fd-boundary-errors}. It follows from \(|\vartheta|\le Cr\),
\eqref{eq.fd-xi-estimate}, and \(D\le C+V\) that
\[
	\left|\frac{\vartheta(r)-2\xi(r)}{r}\right|
	\le C(1+V(r)+\sqrt{\mathcal Q(r)}).
\]
Consequently,
\[
	\begin{split}
		\int_{r/2}^r|(\log T(t))'|\,\dx{t}
		\le{}&C\log2+2\int_{r/2}^r\frac{V(t)}t\,\dx{t}\\
		&+C\int_{r/2}^r(1+V(t)+\sqrt{\mathcal Q(t)})\,\dx{t}
	\end{split}
\]
for \(0<r\le r_1\). By the Cauchy--Schwarz inequality, we get
\[
	\int_{r/2}^r\frac{V(t)}t\,\dx{t}
	\le\sqrt{\log2}
	\left(\int_{r/2}^r\frac{V^2(t)}t\,\dx{t}\right)^{1/2},
\]
and
\[
	\int_{r/2}^r(V(t)+\sqrt{\mathcal Q(t)})\,\dx{t}
	\le Cr\left(\int_{r/2}^r\frac{V^2(t)+\mathcal Q(t)}t\,\dx{t}\right)^{1/2}.
\]
Recall \eqref{eq.fd-tail-estimate}, we obtain that \(\log T\) has bounded 
variation on
dyadic intervals; exponentiating yields \eqref{eq.fd-height-doubling}.

For \(0<r<r_1/2\) and \(r\le s\le2r\), the monotonicity of \(G(r)\) and
\eqref{eq.fd-height-doubling} give
\[
	V(r)=\frac{G(r)}{r^{n+1}T(r)}
	\le\left(\frac sr\right)^{n+1}\frac{T(s)}{T(r)}V(s)
	\le CV(s).
\]
Therefore, it follows from \eqref{eq.fd-tail-integrability} that
\[
	V^2(r)\log2
	\le C\int_r^{2r}\frac{V^2(s)}s\,\dx{s}
	\longrightarrow0\qquad(r\to0^+),
\]
Hence \(V(r)\to0\) as
\(r\to0^+\). Since \(\eta(r)=r^\beta\to0\), the second identity
in \eqref{eq.fd-Psi}
gives
\[
	D(r)\longrightarrow\frac32+\mathsf H(0^+)=N.
\]
Since \(a_0=1/H_0>0\) and
\(a=a_0+O(r)\), \(b=a_0+O(r)\) uniformly by
\eqref{eq.fd-elementary-coefficients}, we have
\[
	\left|
	\frac{r\int_{B_r}|\nabla u|^2\,\dx{x}}
	{\int_{\partial B_r}u^2\,\dH{n-1}}-D(r)
	\right|\le CrD(r).
\]
Since \(D(r)\to N<\infty\), this proves
\eqref{eq.fd-unweighted-frequency}.
\end{proof}

\subsection{Compactness of normalized blow-ups}
Let \(u\) be a subsonic variational solution
of \eqref{eq.quaslinear-free-boundary-problem2}, let \(x_0\in \Sigma_u\) be
fixed. The aim of this section is to 
establish the compactness of normalized blow-ups
\begin{equation}\label{eq.A_0v_r}
	\begin{aligned}
		A_0(r)&:=\left(r^{1-n}\int_{\partial B_r(x^0)}u^2\,\dH{n-1}\right)^{1/2},\\
		v_r(x)&:=\frac{u(x^0+rx)}{A_0(r)}.
	\end{aligned}
\end{equation}
It follows from \eqref{eq.fd-renormalized-Weiss} and \eqref{eq.fd-Weiss-frequency-relation} that
\[
	\frac{\Phi(r)-\mathcal B_\lambda(r)}{W(r)}=\left(\mathsf H_\beta(r)-\eta(r)\right)\frac{T(r)}{W(r)}-\frac{P(r)}{r^{n+1}W(r)}-\frac32\frac{W(r)-T(r)}{W(r)}.
\]
Moreover, as \(T(r)/W(r)=1+O(r)\), \(P(r)\le CrI(r)\), and \(I(r)=r^{n+1}T(r)D(r)\), we have
\[
	0\le\frac{P(r)}{r^{n+1}W(r)}\le CrD(r)\frac{T(r)}{W(r)}\longrightarrow0.
\]
Thus, we deduce from \eqref{eq.fd-D-V-limits} that
\begin{equation*}
	\lim_{r\to0^+}\frac{\Phi(r)-\mathcal B_\lambda(r)}{W(r)}=\mathsf H(0^+).
\end{equation*}
Thanks to Lemma~\ref{lem.fd-height-positive}, we have \(A_0(r)>0\) and
\begin{equation}\label{eq.fd-unit-height}
	\int_{\partial B_1}v_r^2\,\dH{n-1}=1,\qquad
	v_r\ge0,\qquad v_r=0\quad\hbox{on }\{x_n\le0\}.
\end{equation}

\begin{proposition}\label{prop.fd-preliminary-compactness}
Let \(u\) be a subsonic variational solution
of \eqref{eq.quaslinear-free-boundary-problem2}, and let \(0\in \Sigma_u\). Assume that \(u\) satisfies the growth condition \eqref{eq.fd-growth}.
Then we have
\begin{enumerate}[label=(\roman*)]
\item For every fixed \(\rho>0\),
\begin{equation}\label{eq.fd-exact-height-ratio}
	\lim_{r\to0^+}\frac{A_0(r\rho)}{A_0(r)}=\rho^{N},
\end{equation}
where \(N\) is the frequency limit obtained in \eqref{eq.fd-D-V-limits}.
Moreover, there is a constant \(C\) such that
\begin{equation}\label{eq.fd-small-scale-energy}
	\frac{A_0(r\rho)}{A_0(r)}\le C\rho^{3/2},\qquad
	\int_{B_\rho}|\nabla v_r|^2\,\dx{x}\le C\rho^{n+1},
\end{equation}
for all sufficiently small
\(r\) and all \(0<\rho\le1\).
\item \(v_r\) is bounded in
\(W^{1,2}(B_R)\) for each fixed \(R>0\) and \(r\) sufficiently small. There exists
a subsequence \(r_j\to 0^+\) such that
\begin{equation*}
	v_{r_j}\rightharpoonup v_0
	\quad\text{in }W^{1,2}_{\mathrm{loc}}(\R^n),\qquad
	v_{r_j}\longrightarrow v_0
	\quad\text{in }L^2_{\mathrm{loc}}(\R^n).
\end{equation*}
The limit \(v_0\ge 0\) satisfies \(v_0=0\) almost everywhere on
\(\{x_n\le0\}\) and
\begin{equation*}
	\int_{\partial B_1}v_0^2\,\dH{n-1}=1,\qquad
	x\cdot\nabla v_0=Nv_0
	\quad\text{a.e. in }\R^n.
\end{equation*}
\item \(v_0\not\equiv 0\) is a homogeneous function of degree \(N\).
\end{enumerate}
\end{proposition}

\begin{proof}
It follows from \eqref{eq.fd-height-derivative},
\eqref{eq.fd-xi-estimate}, \eqref{eq.fd-D-bound} and
\eqref{eq.fd-vartheta-bound} that
\begin{equation}\label{eq.fd-integrable-height-error}
	\begin{aligned}
		(\log T(r))'&=\frac{2(D(r)-3/2)}r+\varepsilon_0(r),
		\qquad\varepsilon_0(r):=\frac{\vartheta(r)-2\xi(r)}r,\\
		|\varepsilon_0(r)|&\le C\bigl(1+V(r)+\sqrt{\mathcal Q(r)}\bigr),
	\end{aligned}
\end{equation}
where \(\vartheta\) and \(\xi\) are defined in
\eqref{eq.fd-boundary-errors} and \(\mathcal Q\) in \eqref{eq.fd-Q-S}. 
It follows from the estimate in \eqref{eq.fd-tail-estimate}
and the Cauchy--Schwarz inequality that 
\[
	\int_0^{r_1}\bigl(V(t)+\sqrt{\mathcal Q(t)}\bigr)\,\dx{t}
	\le r_1\left(\int_0^{r_1}\frac{V^2(t)+\mathcal Q(t)}t\,\dx{t}\right)^{1/2}
	<\infty,
\]
for a sufficiently small fixed \(r_1\). Hence \(V\in L^1(0,r_1)\)
and \(\varepsilon_0\in L^1(0,r_1)\). Since \(\mathsf H_\beta\ge0\) and
\(V\ge0\), we infer from 
\eqref{eq.fd-Psi} that \(D(r)\ge3/2-\eta(r)\). Recalling the first identity in 
\eqref{eq.fd-integrable-height-error}, we obtain
\[
	(\log T(r))'\ge-2r^{\beta-1}-C\left(1+V(r)+\sqrt{\mathcal Q(r)}\right).
\]
Integrating from \(r\rho\) to \(r\) gives \(T(r\rho)\le CT(r)\) uniformly for
\(0<\rho\le1\). Since
\begin{equation}\label{eq.fd-T-A0}
	T(r)=(a_0+O(r))\frac{A_0^2(r)}{r^3},
\end{equation}
the comparison \(T(r\rho)\le CT(r)\) gives
\[
	\frac{A_0(r\rho)}{A_0(r)}\le C\rho^{3/2},
\]
which is the first bound in \eqref{eq.fd-small-scale-energy}.

To prove \eqref{eq.fd-exact-height-ratio}, we integrate
the first identity in \eqref{eq.fd-integrable-height-error}
from \(r\) to \(r\rho\). Using the fact that \(D(t)\to N\), we have 
by a direct computation that
\[
	\left|\int_r^{r\rho}\frac{2(D(t)-N)}{t}\,\dx{t}\right|
	\le2|\log\rho|
	\sup_{\min\{r,r\rho\}\le t\le\max\{r,r\rho\}}|D(t)-N|
	\longrightarrow0,
\]
as \(r\to0^+\). Moreover, since
\(\varepsilon_0\in L^1(0,r_1)\), we obtain
\(\int_r^{r\rho}\varepsilon_0(t)\,\dx{t}\to0\). Hence
\[
	\log\frac{T(r\rho)}{T(r)}\longrightarrow2\left(N-\frac32\right)\log\rho.
\]
This together with \eqref{eq.fd-T-A0} gives
\eqref{eq.fd-exact-height-ratio}.
By \eqref{eq.fd-unweighted-frequency} in
Corollary~\ref{cor.fd-frequency-limits}, we have that
\[
	\frac{r\int_{B_r(x^0)}|\nabla u|^2\,\dx{x}}
	{\int_{\partial B_r(x^0)}u^2\,\dH{n-1}}\to N.
\]
This implies that 
\(r\int_{B_r(x^0)}|\nabla u|^2\,\dx{x}/\int_{\partial B_r(x^0)}u^2\,\dH{n-1}\) is bounded.
A change of variables
gives 
\begin{equation}\label{eq.fd-scaled-bulk-energy}
	\int_{B_\rho}|\nabla v_r|^2\,\dx{x}
	=\frac{r\rho\int_{B_{r\rho}}|\nabla u|^2\,\dx{x}}
	{\int_{\partial B_{r\rho}}u^2\,\dH{n-1}}
	\rho^{n-2}\frac{A_0^2(r\rho)}{A_0^2(r)}.
\end{equation}
Note that the first factor is bounded, and the first bound in
\eqref{eq.fd-small-scale-energy} gives
\(A_0^2(r\rho)/A_0^2(r)\le C\rho^3\). It follows that
\[
	\int_{B_\rho}|\nabla v_r|^2\,\dx{x}\le C\rho^{n+1},
\]
which gives the second bound in \eqref{eq.fd-small-scale-energy}. For any
fixed \(R\ge1\), it follows from \eqref{eq.fd-scaled-bulk-energy} and
\eqref{eq.fd-exact-height-ratio} that
\(\int_{B_R}|\nabla v_r|^2\) is bounded. Since \(v_r\) vanishes on the lower half \(B_R\cap\{x_n\le0\}\), it follows from the Poincaré inequality that
\[
	\int_{B_R}v_r^2\,\dx{x}\le CR^2\int_{B_R}|\nabla v_r|^2\,\dx{x}.
\]
Thus \(v_r\) is bounded in \(W^{1,2}(B_R)\) for every fixed \(R\), and
a diagonal argument yields the weak convergence and the strong
\(L^2_{\mathrm{loc}}\) convergence in (ii).

It remains to show that \(v_0\) is a homogeneous function of degree \(N\). A change of variable and a direct calculation gives 
\[
	\begin{aligned}
		&\int_{B_\tau\setminus B_\sigma}
		|x\cdot\nabla v_{r_j}-D(r_j|x|)v_{r_j}|^2\,\dx{x}\\
		&\qquad\le\frac{C}{A_0^2(r_j)r_j^n}
		\int_{\sigma r_j}^{\tau r_j}J(t)\mathcal Q(t)\,\dx{t}
		\le C_{\sigma,\tau}
		\int_{\sigma r_j}^{\tau r_j}\frac{\mathcal Q(t)}t\,\dx{t}
		\longrightarrow0.
	\end{aligned}
\]
Here used the inequality in
\eqref{eq.fd-height-doubling} and that \(a(x)\) is bounded from below.
Since \(D(r_j|x|)\to N\) uniformly on this fixed annulus and
\(v_j\) is bounded in \(L^2\) there, we obtain
\[
	x\cdot\nabla v_j-Nv_j\longrightarrow0
	\quad\text{in }L^2(B_\tau\setminus B_\sigma).
\]
Passing to the weak limit gives \(x\cdot\nabla v_0=Nv_0\) on every
such annulus, which proves the desired
homogeneity.
\end{proof}

We also remark that the preceding calculation also shows that 
for each fixed \(\rho>0\),
\begin{equation}\label{eq.fd-moment-limits}
	\begin{aligned}
		\int_{\partial B_\rho}v_r^2\,\dH{n-1}
		&\longrightarrow\rho^{n-1+2N},\\
		\int_{B_\rho}|\nabla v_r|^2\,\dx{x}
		&\longrightarrow N\rho^{n-2+2N}.
	\end{aligned}
\end{equation}

\begin{lemma}
\label{lem.fd-preliminary-normalized-flux}
Let \(u\) be a subsonic variational solution
of \eqref{eq.quaslinear-free-boundary-problem2}, and let \(0\in \Sigma_u\). Assume that \(u\) satisfies the growth condition \eqref{eq.fd-growth}. Let \(v_r\) be defined in \eqref{eq.A_0v_r} and set
\begin{equation}\label{eq.pc-hat-a-definition}
	\hat a_r(x):=H_0a(x^0+rx)
	=\frac{H_0}{H(|\nabla u(x^0+rx)|^2;rx_n)}.
\end{equation}
Then we have
\begin{equation}\label{eq.fd-hat-a-estimates}
	c\le\hat a_r\le C,\qquad
	\|\hat a_r-1\|_{L^\infty(B_R)}\le C_Rr,
\end{equation}
and
\begin{equation}\label{eq.fd-preliminary-flux-decomposition}
	\operatorname{div}(\hat a_r\nabla v_r)=\mu_r\ge0,\quad
	\Delta v_r=\mu_r+g_r,\quad
	g_r\longrightarrow0\quad\text{in }H^{-1}_{\mathrm{loc}}(\R^n),
\end{equation} 
for every fixed \(R>0\).
\end{lemma}

\begin{remark}
It follows that the weak limit \(v_0\) from
Proposition~\ref{prop.fd-preliminary-compactness} is subharmonic
in the distributional sense, and
\(\mu_{r_j}\stackrel{*}{\rightharpoonup}\Delta v_0\) as locally finite
measures.
\end{remark}

\begin{proof}
The estimate \eqref{eq.fd-hat-a-estimates} follows from
\(|\nabla u(x^0+rx)|^2\le C_0rx_n^+\) and smoothness of the local
branch at \((0,0)\).
Rescaling the nonnegative measure
\(\mu=\operatorname{div}(a\nabla u)\) defined in
Lemma~\ref{lem.fd-flux-testing}, we obtain
\[
	\mu_r(B)=\frac{H_0r^{2-n}}{A_0(r)}\mu(x^0+rB),
\]
for relatively compact
Borel sets \(B\subset\subset\Omega_r:=r^{-1}(\Omega-x^0)\).
In particular, \(\mu_r\ge0\). If \(K\subset\subset B_R\) and
\(0\le\zeta\in C_c^\infty(B_R)\) is a test function
which equals one near \(K\), then
\begin{equation}\label{eq.fd-mu-mass-bound}
	\begin{aligned}
		\mu_r(K)&\le\langle\mu_r,\zeta\rangle\\
		&=-\int_{B_R}\hat a_r\nabla v_r\cdot\nabla\zeta\,\dx{x}\\
		&\le C\|\nabla v_r\|_{L^2(B_R)}\|\nabla\zeta\|_{L^2(B_R)}\le C_K.
	\end{aligned}
\end{equation}
Set \(g_r:=\operatorname{div}((1-\hat a_r)\nabla v_r)\). On each
fixed ball \(B_R\), we compute directly to obtain
\[
	\|g_r\|_{H^{-1}(B_R)}
	\le\|1-\hat a_r\|_{L^\infty(B_R)}\|\nabla v_r\|_{L^2(B_R)}
	\longrightarrow0.
\]
Adding \(\operatorname{div}(\hat a_r\nabla v_r)\) and \(\operatorname{div}((1-\hat a_r)\nabla v_r)\) 
proves the decomposition. Passing to the limit in
\eqref{eq.fd-preliminary-flux-decomposition} identifies every measure
limit with \(\Delta v_0\). The weak-star convergence and subharmonicity follow from \(\mu_r\ge 0\) and \eqref{eq.fd-mu-mass-bound}.
\end{proof}

\subsection{Strong convergence in two dimensions}
Let \(n=2\) and let \(u\) be a subsonic variational solution
of \eqref{eq.quaslinear-free-boundary-problem2}.
Fix \(x^0\in\Sigma_u\) and translate \(x^0\) horizontally to the origin. 
It follows from Corollary~\ref{cor.fd-frequency-limits} and Proposition~\ref{prop.fd-preliminary-compactness} that
\begin{equation}\label{eq.pc-compactness-inputs}
	\begin{gathered}
		D(r)\to N\ge3/2,\qquad V(r)\to0,\qquad
		\int_0^{r_0}\mathcal Q(r)\,\frac{\dx{r}}{r}<\infty,\\
		\sup_{r\ \mathrm{small}}\|v_r\|_{W^{1,2}(B_R)}\le C_R,
		\qquad\int_{\partial B_1}v_r^2\,\dH{1}=1,\\
		\int_{B_\rho}|\nabla v_r|^2\,\dx{x}\le C\rho^3
		\quad(0<\rho<1).
	\end{gathered}
\end{equation}
we now recall the concentration-cancellation theory of Evans and M\"uller~\cite{MR1220787}.

\begin{lemma}
\label{lem.two-dimensional-positive-laplacian-cancellation}
Let \(O\subset\R^2\) be open. If \(w_j\rightharpoonup w\) in \(W^{1,2}_{\mathrm{loc}}(O)\) and \(\Delta w_j\ge0\) in \(\mathcal D'(O)\), then
\[
	\begin{aligned}
		(\partial_1w_j)^2-(\partial_2w_j)^2
		&\longrightarrow(\partial_1w)^2-(\partial_2w)^2,\\
		\partial_1w_j\partial_2w_j
		&\longrightarrow\partial_1w\partial_2w,
	\end{aligned}\quad\text{ in }\mathcal D'(O).
\]
\end{lemma}

We put the proof in Appendix~\ref{Appendix: proof of two-dimensional positive laplacian cancellation}. We next state a convergence theorem. Let
\(f_j\rightharpoonup f\) in \(W^{1,2}_{\mathrm{loc}}(\R^2)\), let
\(\mu_j\) be nonnegative Radon measures on \(\R^2\), and let
\(g_j\in H^{-1}_{\mathrm{loc}}(\R^2)\). Assume that
\begin{equation}\label{eq.pc-abstract-flux}
	\begin{gathered}
		\Delta f_j=\mu_j+g_j,\\
		\sup_j\mu_j(K)<\infty\quad(K\subset\subset\R^2),\\
		g_j\longrightarrow0\quad\text{in }H^{-1}_{\mathrm{loc}}(\R^2).
	\end{gathered}
\end{equation}
Then we have the following property.
\begin{lemma}
\label{lem.pc-self-contained-compactness}
Assume that there exists a constant \(N>0\) such that
\begin{equation}\label{eq.pc-abstract-radial}
	x\cdot\nabla f_j-Nf_j\longrightarrow0
	\quad\text{in }L^2(B_\tau\setminus B_\sigma)
\end{equation}
for every fixed \(0<\sigma<\tau<\infty\), where \(f_j,\mu_j,g_j\)
satisfy \eqref{eq.pc-abstract-flux}. Then \(f\) is a homogeneous
function of degree \(N\), and \(f_j\to f\) strongly in
\(W^{1,2}_{\mathrm{loc}}(\R^2\setminus\{0\})\).  
Assume in addition that 
\begin{equation}\label{eq.pc-small-ball-energy}
	\lim_{\rho\to0^+}\sup_j\int_{B_\rho}|\nabla f_j|^2\,\dx{x}=0.
\end{equation}
then \(f_j\to f\) strongly in \(W^{1,2}_{\mathrm{loc}}(\R^2)\).
\end{lemma}

\begin{proof}
Fix \(R>0\), and let \(w_j\in H_0^1(B_{2R})\) which satisfies \(\Delta w_j=g_j\). It 
follows from Lax--Milgram theorem and the energy estimate that
\[
	\begin{gathered}
		\int_{B_{2R}}\nabla w_j\cdot\nabla\varphi\,\dx{x}
		=-\langle g_j,\varphi\rangle,
		\qquad\varphi\in H_0^1(B_{2R}),\\
		\|w_j\|_{H_0^1(B_{2R})}
		\le C_R\|g_j\|_{H^{-1}(B_{2R})}\longrightarrow0.
	\end{gathered}
\]
Thus \(h_j:=f_j-w_j\rightharpoonup f\) in \(W^{1,2}(B_{2R})\) and \(\Delta h_j=\mu_j\ge0\) is subharmonic. It follows from Lemma~\ref{lem.two-dimensional-positive-laplacian-cancellation}, the
trace-free parts\footnote{Here, it means the matrix 
\(\nabla h\otimes\nabla h-\frac12|\nabla h|^2\operatorname{Id}=
\begin{pmatrix}
	\dfrac{h_1^2-h_2^2}{2} & h_1h_2\\[2mm]
	h_1h_2 & \dfrac{h_2^2-h_1^2}{2}
\end{pmatrix}\).} of \(\nabla h_j\otimes\nabla h_j\) converge in
the sense of distributions. Since \(f_j-h_j=w_j\),
\[
	\nabla f_j\otimes\nabla f_j-\nabla h_j\otimes\nabla h_j
	=\nabla w_j\otimes\nabla f_j
	+\nabla h_j\otimes\nabla w_j,
\]
and hence, we have
\[
	\begin{aligned}
		&\|\nabla f_j\otimes\nabla f_j
		-\nabla h_j\otimes\nabla h_j\|_{L^1(B_R)}\\
		&\qquad\le\|\nabla w_j\|_{L^2(B_R)}
		\bigl(\|\nabla f_j\|_{L^2(B_R)}
		+\|\nabla h_j\|_{L^2(B_R)}\bigr)\longrightarrow0.
	\end{aligned}
\]
Since \(f_j\to f\) weakly in \(W_{\rm loc}^{1,2}(\R^2)\), we have \(f_j\to f\) locally in \(L^2\). It follows from \eqref{eq.pc-abstract-radial} that \(x\cdot\nabla f=Nf\) almost everywhere away from the origin. It follows that \(f(\varrho,\theta)=\varrho^Nq(\theta)\) almost everywhere. Moreover, we have
\begin{equation}\label{eq.pc-strong-radial-derivative}
	x\cdot\nabla f_j\longrightarrow x\cdot\nabla f
	\quad\text{strongly in }L^2(B_\tau\setminus B_\sigma).
\end{equation}
After a diagonal extraction, we have the weak-star convergence
\begin{equation}\label{eq.pc-tensor-convergence}
	\nabla f_j\otimes\nabla f_j\,\dx{x}
	\stackrel{*}{\rightharpoonup}
	\nabla f\otimes\nabla f\,\dx{x}+\mathbf M
\end{equation}
for some symmetric matrix-valued measure \(\mathbf M\). For every
\(e\in\R^2\) and \(0\le\zeta\in C_c^\infty\), we compute to see that
\[
	\int\zeta\,\dx{(e^T\mathbf M e)}
	=\lim_j\int\zeta|e\cdot\nabla f_j|^2\,\dx{x}
	-\int\zeta|e\cdot\nabla f|^2\,\dx{x}\ge0,
\]
where the last inequality follows from weak lower semicontinuity. 
Hence \(\mathbf M\ge0\). The convergence of
the trace-free parts now gives \(M_{12}=M_{21}=0\) and
\(M_{11}=M_{22}=:m\). It follows that \(\mathbf M=m\operatorname{Id}\) with
\(m\ge0\).
Let \(0<\sigma<\tau<\infty\) and let
\(\zeta\in C_c^\infty(\{x\in\R^2:\sigma<|x|<\tau\})\). It follows from
\eqref{eq.pc-strong-radial-derivative} that
\[
	\begin{aligned}
		0&=\lim_j\int\zeta\bigl(|x\cdot\nabla f_j|^2
		-|x\cdot\nabla f|^2\bigr)\,\dx{x}\\
		&=\sum_{k,l=1}^2\int\zeta x_kx_l\,\dx{M_{kl}}
		=\int\zeta|x|^2\,\dx{m}.
	\end{aligned}
\]
Taking \(\zeta=\eta/|x|^2\) with arbitrary test function
\(\eta\in C_c^\infty(\{x\in\R^2:\sigma<|x|<\tau\})\) gives \(\int\eta\,\dx{m}=0\)
for every such \(\eta\). Since \(m\ge0\), we deduce that the measure
\(m\) vanishes on \(\{x\in\R^2:\sigma<|x|<\tau\}\).
Taking the trace in \eqref{eq.pc-tensor-convergence} and using
\(\mathbf M=m\operatorname{Id}\) with \(m=0\) on
\(\R^2\setminus\{0\}\), we obtain
\[
	|\nabla f_j|^2\,\dx{x}
	\stackrel{*}{\rightharpoonup}
	|\nabla f|^2\,\dx{x}
	\qquad\text{on }\R^2\setminus\{0\}.
\]
Equivalently, this means that
\[
	\int\zeta|\nabla f_j|^2\,\dx{x}
	\longrightarrow
	\int\zeta|\nabla f|^2\,\dx{x},
\] 
for \(0\le\zeta\in C_c^\infty(\R^2\setminus\{0\})\). Thanks to the weak convergence of the gradient, we have 
\[
	\int\zeta|\nabla f_j-\nabla f|^2\,\dx{x}\longrightarrow0
\]
this proves the strong convergence away from the origin. Under the additional assumption \eqref{eq.pc-small-ball-energy}, we
obtain from weak lower semicontinuity and the estimate
\[
	\int_{B_\rho}|\nabla f_j-\nabla f|^2\,\dx{x}
	\le2\int_{B_\rho}|\nabla f_j|^2\,\dx{x}
	  +2\int_{B_\rho}|\nabla f|^2\,\dx{x}
\]
that
\[
	\sup_j\int_{B_\rho}|\nabla f_j-\nabla f|^2\,\dx{x}
	\le4\sup_j\int_{B_\rho}|\nabla f_j|^2\,\dx{x}
	\longrightarrow0,
\]
as \(\rho\to 0^+\). Hence, for every \(\varepsilon>0\), there is \(\rho>0\) 
such that
\[
	\sup_j\int_{B_\rho}|\nabla f_j-\nabla f|^2\,\dx{x}<\varepsilon.
\]
Recall that strong convergence has been proved on \(B_R\setminus B_\rho\), it follows that
\[
	\int_{B_R\setminus B_\rho}|\nabla f_j-\nabla f|^2\,\dx{x}
	<\varepsilon
\]
for all sufficiently large \(j\). Adding the two bounds gives
\[
	\limsup_{j\to\infty}
	\int_{B_R}|\nabla f_j-\nabla f|^2\,\dx{x}
	\le2\varepsilon.
\]
Passing to the limit as \(\varepsilon\to0^+\) gives
\(\int_{B_R}|\nabla f_j-\nabla f|^2\to0\). Since \(R>0\) is arbitrary, we 
prove that the strong convergence in \(W^{1,2}_{\mathrm{loc}}(\R^2)\).
\end{proof}

\begin{theorem}
\label{thm.pc-strong-convergence}
Let \(n=2\) and let \(u\) be a subsonic variational solution
of \eqref{eq.quaslinear-free-boundary-problem2}, and let \(0\in \Sigma_u\). Assume that \(u\) satisfies the growth condition \eqref{eq.fd-growth}. Let \(v_{r_j}\) be a weakly convergent
sequence as in Proposition~\ref{prop.fd-preliminary-compactness}, then up to a further subsequence,
\[
	v_{r_j}\longrightarrow v_0
	\quad\text{strongly in }W^{1,2}_{\mathrm{loc}}(\R^2).
\]
Moreover, \(v_0\ge 0\) is nonnegative and is a homogeneous function of degree \(N\), \(v_0\) vanishes on \(\{x_2\le0\}\), and satisfies
\begin{equation}\label{eq.pc-limit-normalization}
	\int_{\partial B_1}v_0^2\,\dH{1}=1,
\end{equation}
where \(N:=\frac32+\mathsf H(0^+)\).
\end{theorem}

\begin{proof}
By Proposition~\ref{prop.fd-preliminary-compactness}, we have that 
\(v_j:=v_{r_j}\) converges weakly in
\(W^{1,2}_{\mathrm{loc}}\) and strongly in \(L^2_{\mathrm{loc}}\) to a
nonnegative \(v_0\) vanishing on \(\{x_2\le0\}\), normalized by
\eqref{eq.pc-limit-normalization}. On every fixed annulus, it follows from the coarea
formula and the estimate
\eqref{eq.fd-height-doubling} that
\[
	x\cdot\nabla v_j-D(r_j|x|)v_j\longrightarrow0
	\quad\text{in }L^2(B_\tau\setminus B_\sigma);
\]
since \(D(r_j|x|)\to N\) uniformly and \(v_j\) is bounded in
\(L^2\), this shows
\(x\cdot\nabla v_j-Nv_j\to0\) in \(L^2(B_\tau\setminus B_\sigma)\).
Thanks to Lemma~\ref{lem.fd-preliminary-normalized-flux}, we have \eqref{eq.pc-abstract-flux}, and it follows from 
\eqref{eq.pc-compactness-inputs} the assumption
\eqref{eq.pc-small-ball-energy} is satisfied. Thus the 
strong
convergence and the \(N\)-homogeneity of \(v_0\) follow from
Lemma~\ref{lem.pc-self-contained-compactness} and the Theorem is proved.
\end{proof}

\begin{proposition}
Let \(v_0\) be the strong blow-up limit obtained in Theorem~\ref{thm.pc-strong-convergence}. Then \(v_0\) has a continuous representative which satisfies
\begin{equation*}
	\Delta v_0\ge0,\qquad v_0\Delta v_0=0
	\quad\text{as Radon measures in }\R^2.
\end{equation*}
Moreover, \(v_0\) harmonic in \(\{v_0>0\}\) and satisfies 
\begin{equation}\label{eq.pc-limit-stationarity}
	\int_{\R^2}
	\bigl(|\nabla v_0|^2\operatorname{div}\xi
	-2\nabla v_0\cdot D\xi\,\nabla v_0\bigr)\,\dx{x}=0
\end{equation}
for every \(\xi\in C_c^1(\{x_2>0\};\R^2)\).
\end{proposition}

\begin{proof}
For convenience, let us denote \(v_j=v_{r_j}\) and write \(\hat a_j=\hat a_{r_j}\) where
\(\hat a_r\) is defined in \eqref{eq.pc-hat-a-definition}. Thanks to the homogeneity of \(v_0\), 
we may write \(v_0(\varrho,\theta)=\varrho^Nq(\theta)\). 
It follows from Lemma \ref{lem.fd-preliminary-normalized-flux} 
that \(\Delta v_0\ge0\). Since \(v_0\in W^{1,2}_{\mathrm{loc}}\),
we have \(q\in H^1(\mathbb S^1)\) and hence \(q\) is bounded. This together with \(N>0\) implies that \(v_0\) is
continuous at the origin with \(v_0(0)=0\). Observe that for \(\zeta\in C_c^\infty(\R^2)\),
\[
	\int\hat a_j|\nabla v_j|^2\zeta\,\dx{x}
	=-\int\hat a_jv_j\nabla v_j\cdot\nabla\zeta\,\dx{x}
\]
holds for every \(j\). Passing to the limit as \(j\to\infty\), using the strong convergence of 
\(v_j\to v_0\) and the local uniform convergence
\(\hat a_j\to1\), we obtain
\[
	-\int\nabla v_0\cdot\nabla(\zeta v_0)\,\dx{x}=0.
\]
Thus \(v_0\Delta v_0=0\), and \(\Delta v_0=0\) in \(\{v_0>0\}\).

Let us set \(\varepsilon_j=A_0^2(r_j)/r_j^2\), \(t_j=\varepsilon_j|\nabla v_j|^2\), and \(s_j=r_jx_2\). Let \(\xi\in C_c^1(\{x_2>0\};\R^2)\) be a test function, then we test \eqref{eq.domain-variation-formula} with \(y\mapsto r_j\xi(y/r_j)\) and compute to find that
\begin{equation}\label{eq.pc-scaled-stationarity}
	\begin{split}
		0={}&\int\left[
		\frac{H_0}{\varepsilon_j}F(t_j;s_j)\operatorname{div}\xi
		-2\hat a_j\nabla v_j\cdot D\xi\,\nabla v_j
		+\frac{H_0r_j}{\varepsilon_j}\partial_2F(t_j;s_j)\xi_2
		\right]\,\dx{x}\\
		&+\frac{H_0}{\varepsilon_j}
		\int\chi_{\{v_j>0\}}
		\operatorname{div}\bigl(\lambda(r_jx_2)\xi\bigr)\,\dx{x}.
	\end{split}
\end{equation}
It then follows from \(|F(t;s)-t/H_0|\le C(t+|s|)t\), \(|\partial_2F(t;s)|\le Ct\), and \(t_j+|s_j|\le C_Rr_j\) on \(B_R\) that
\begin{equation}\label{eq.pc-kinetic-freezing}
	\begin{split}
		\left|\frac{H_0}{\varepsilon_j}F(t_j;s_j)-|\nabla v_j|^2\right|
		&\le C_Rr_j|\nabla v_j|^2,\\
		\left|\frac{H_0r_j}{\varepsilon_j}\partial_2F(t_j;s_j)\right|
		&\le C_Rr_j|\nabla v_j|^2.
	\end{split}
\end{equation}
We now estimate 
\begin{equation}\label{eq.phase-term}
	\frac{H_0}{\varepsilon_j}
	\int\chi_{\{v_j>0\}}
	\operatorname{div}(\lambda(r_jx_2)\xi)\,\dx{x}.
\end{equation}
The bounds \(s/H_0\le\lambda(s)\le Cs\) for small \(s\ge0\) and the
boundedness of \(\lambda'(s)\) give
\begin{equation}\label{eq.pc-phase-div-bound}
	|\operatorname{div}(\lambda(r_jx_2^+)\xi)|
	\le C_\xi\lambda(r_jx_2^+).
\end{equation}
Since \(\int_{\R^2}\operatorname{div}(\lambda(r_jx_2)\xi)\,\dx{x}=0\),
we may write
\[
	\int\chi_{\{v_j>0\}}\operatorname{div}(\lambda(r_jx_2)\xi)\,\dx{x}
	=-\int(1-\chi_{\{v_j>0\}})
	\operatorname{div}(\lambda(r_jx_2^+)\xi)\,\dx{x}.
\]
Using \eqref{eq.pc-phase-div-bound}, we obtain by changing the variable that 
\begin{equation}\label{eq.pc-phase-variation}
	\begin{split}
		&\left|\frac{H_0}{\varepsilon_j}
		\int\chi_{\{v_j>0\}}
		\operatorname{div}(\lambda(r_jx_2)\xi)\,\dx{x}\right|\\
		&\qquad\le\frac{C_\xi H_0}{\varepsilon_j}
		\int_{B_R}(1-\chi_{\{v_j>0\}})\lambda(r_jx_2^+)\,\dx{x}\\
		&\qquad=\frac{C_\xi H_0}{A_0^2(r_j)}G(Rr_j)
		=C_\xi H_0V(Rr_j)\frac{J(Rr_j)}{Rr_jA_0^2(r_j)},
	\end{split}
\end{equation}
where \(G(r)\) and \(J(r)\) are defined in
\eqref{eq.fd-radial-quantities}, \(V(r)\) is given in \eqref{eq.fd-D-V}, and
\(A_0(r)\) is given in \eqref{eq.A_0v_r}. It follows from 
\eqref{eq.fd-height-doubling} that \(J(Rr_j)/Rr_jA_0^2(r_j)\) is bounded. Thanks to 
\eqref{eq.fd-D-V-limits}, we have 
\(V(Rr_j)\to0\). Hence it follows from \eqref{eq.pc-phase-variation} that \eqref{eq.phase-term}
converges to zero as \(j\to\infty\). Passing to the limit as \(j\to\infty\) and using the strong convergence of \(v_j\) gives the desired result.
\end{proof}

\begin{theorem}
\label{thm.pc-quantization}
Let \(v_0\) be the strong blow-up limit obtained in Theorem~\ref{thm.pc-strong-convergence}. Then \(N=N(x^0)\ge 2\) is a
positive integer and \(v_0=\mathcal P_N\), where
\[
	\mathcal P_N(\varrho,\theta):=
	\begin{cases}
		\sqrt{2/\pi}\,\varrho^N|\sin(N\theta)|,&0<\theta<\pi,\\
		0,&\theta=0\ \text{or}\ \pi\le\theta\le2\pi.
	\end{cases}
\]
\end{theorem}

\begin{proof}
The limiting identities \eqref{eq.pc-limit-stationarity} allow the angular argument in the proof of \cite[Theorem~10.1]{MR2810856} to be applied to \(v_0\). So we omit the details.
\end{proof}

\begin{proposition}[Local finiteness of the full weighted density set]
\label{prop.pc-local-finiteness}
Let \(n=2\) and let \(u\) be a subsonic variational solution
of \eqref{eq.quaslinear-free-boundary-problem2}. Assume that \(u\) satisfies the growth condition \eqref{eq.fd-growth}. Then \(\Sigma_u\) is locally finite in \(\Omega\).
\end{proposition}

\begin{proof}
The argument is similar to \cite[Theorem~10.2]{MR2810856}, so we
just give a sketch here. It follows from 
Lemma~\ref{Lemma: densities at the stagnation points} that
\(\Sigma_u\) relatively closed. Thus it
suffices to rule out distinct \(y_j\in\Sigma_u\) converging to
\(y^0\in\Sigma_u\). Translate \(y^0\) to the origin for convenience and set
\(r_j=|y_j|\) so that \(y_j/r_j\to z\in\{e_1,-e_1\}\). Applying the inequality
\eqref{eq.fd-uniform-moving-lower} to
\(K=\{0\}\cup\{y_j:j\ge1\}\) at \(y=y_j\) and \(s=r_j\rho\) gives
\[
	\frac{\rho\int_{B_\rho(z)}|\nabla v_{r_j}|^2\,\dx{x}}
	{\int_{\partial B_\rho(z)}v_{r_j}^2\,\dH{1}}
	\ge(1-Cr_j\rho)\left(\frac32-(r_j\rho)^\beta\right).
\]
Thanks to Theorem~\ref{thm.pc-quantization}, we have \(v_{r_j}\to\mathcal P_N\)
strongly in \(W^{1,2}_{\mathrm{loc}}\). Passing to the limit at fixed
\(\rho\) as \(j\to\infty\) gives
\[
	\frac{\rho\int_{B_\rho(z)}|\nabla\mathcal P_N|^2\,\dx{x}}
	{\int_{\partial B_\rho(z)}\mathcal P_N^2\,\dH{1}}
	\ge\frac32.
\]
Near \(z=e_1\) or \(z=-e_1\), we get
\[
	\frac{\mathcal P_N(z+\rho x)}{\rho}
	\longrightarrow N\sqrt{2/\pi}\,x_2^+
	\quad\text{strongly in }W^{1,2}(B_1).
\]
For \(cx_2^+\), both \(\int_{B_1}|\nabla(cx_2^+)|^2\) and
\(\int_{\partial B_1}(cx_2^+)^2\) equal \(c^2\pi/2\), so the quotient
tends to \(1\), contradicting \(3/2\). This completes the proof.
\end{proof}

\begin{corollary}
\label{cor.pc-natural-degeneracy}
Let \(n=2\) and let \(u\) be a subsonic variational solution
of \eqref{eq.quaslinear-free-boundary-problem2}. Assume that \(u\) 
satisfies the growth condition \eqref{eq.fd-growth}. Let \(0\in\Sigma_u\) and let 
\(N\) be the integer derived in Theorem \ref{thm.pc-quantization}. Then
\begin{equation*}
	\frac{u(rx)}{r^\alpha}\longrightarrow0
	\quad\text{strongly in }W^{1,2}_{\mathrm{loc}}(\R^2)
\end{equation*}
for every \(\alpha<N\). Moreover,
\begin{equation}\label{eq.pc-natural-zero}
	u_r(x):=\frac{u(rx)}{r^{3/2}}\longrightarrow0,
\end{equation}
locally uniformly in \(\R^2\) and strongly in \(W^{1,2}_{\mathrm{loc}}(\R^2)\), and
\begin{equation}\label{eq.pc-phase-density}
	\chi_{\{u_r>0\}}\longrightarrow\chi_{\{x_2>0\}}
	\quad\text{strongly in }L^1_{\mathrm{loc}}(\R^2),
\end{equation}
as \(r\to 0^+\).
\end{corollary}

\begin{proof}
Let us begin by recalling the following estimates proved in
\eqref{eq.fd-integrable-height-error} and
\eqref{eq.fd-D-V-limits}.
\[
	(\log T)'(r)=\frac{2(D(r)-3/2)}r+\varepsilon_0(r),
	\qquad \varepsilon_0\in L^1(0,r_1),\quad D(r)\to N.
\]
For \(\alpha<N\), we may choose \(0<\varepsilon<N-\alpha\) and \(r_*\) such
that \(D(r)\ge N-\varepsilon\) for \(r<r_*\). Integrating the identity with respect to
\((\log T)'\) 
from \(r\) to \(r_*\) gives
\[
	\log T(r_*)-\log T(r)
	\ge2(N-\varepsilon-3/2)\log(r_*/r)
	-\|\varepsilon_0\|_{L^1(0,r_*)},
\]
hence
\[
	T(r)\le Cr^{2(N-\varepsilon)-3},\qquad
	A_0^2(r)\le Cr^3T(r)\le Cr^{2(N-\varepsilon)}.
\]
It follows from Proposition~\ref{prop.fd-preliminary-compactness} that for each fixed \(R\),
\[
	\left\|x\mapsto\frac{u(rx)}{r^\alpha}\right\|_{W^{1,2}(B_R)}
	=\frac{A_0(r)}{r^\alpha}\|v_r\|_{W^{1,2}(B_R)}
	\le C_Rr^{N-\varepsilon-\alpha}\longrightarrow0.
\]
Taking \(\alpha=3/2<N\) proves strong convergence in \eqref{eq.pc-natural-zero} and \(T(r)\to0\). It follows from \eqref{eq.fd-growth} that \(u_r\) is locally Lipschitz, so their local \(L^2\) convergence implies local uniform convergence.

It remains to prove \eqref{eq.pc-phase-density}, it follows \(\lambda(s)\ge cs\) at small nonnegative heights that
\[
	\int_{B_R}x_2^+(1-\chi_{\{u_r>0\}})\,\dx{x}
	\le Cr^{-3}G(Rr)
	=CR^3V(Rr)T(Rr)\longrightarrow0.
\]
This controls the missing area in \(B_R\cap\{x_2\ge\eta\}\), while the upper strip \(B_R\cap\{0<x_2<\eta\}\) has area at most \(C_R\eta\). Both indicators vanish below the line; letting \(r\to0^+\) and then \(\eta\to0^+\) proves \eqref{eq.pc-phase-density}.
\end{proof}

\section{Nonexistence of full weighted density points}\label{sec.full-density-exclusion}

In this section, we exclude full weighted density points under the local curve and
finite-component hypotheses. Recall first the curve assumption in Definition~\ref{def.pc-local-curve}.
We retain the notation \(\sigma\) for its parametrization and
\(\sigma((-\tau_0,0))\), \(\sigma((0,\tau_0))\) for its two local
branches. Let \(x^0\in S_u\) be a stagnation point and we define the natural rescaling and its flux by
\[
	\begin{aligned}
		u_r(x)&:=r^{-3/2}u(x^0+rx),\\
		\hat a_r(x)&:=\frac{H_0}{H(r|\nabla u_r(x)|^2;rx_2)},\qquad
		\tilde\mu_r:=\operatorname{div}(\hat a_r\nabla u_r).
	\end{aligned}
\]
If \(x^0\in\Sigma_u\) is a full weighted density point, we apply \(v_r\) to denote the
boundary-normalized rescaling introduced in equation \eqref{eq.A_0v_r}.

\begin{lemma}
\label{lem.pc-phase-stabilization}
Let \(n=2\) and let \(u\) be a subsonic variational solution of
\eqref{eq.quaslinear-free-boundary-problem2}. Assume that \(u\)
satisfies the growth condition \eqref{eq.fd-growth}. Suppose that the
local curve hypothesis holds at the stagnation point \(x^0\in S_u\).
Then the measure \(\tilde\mu_r\) is nonnegative and satisfies
\begin{equation}\label{eq.pc-natural-flux-boundary}
	\sqrt{x_2}\,|D\chi_{\{u_r>0\}}|
	\le2\sqrt{C_0}\,\tilde\mu_r
	\qquad\text{in }\{x_2>0\},
\end{equation}
on every fixed compact subset of the upper half-plane and for all
sufficiently small \(r>0\). If \(x^0\in\Sigma_u\), then for every
compact subset \(K\subset\subset\{x_2>0\}\),
\begin{equation}\label{eq.pc-phase-stabilization}
	K\cap\partial\{u_r>0\}=\varnothing,
	\qquad K\subset\{u_r>0\}=\{v_r>0\}
\end{equation}
for all sufficiently small \(r>0\).
\end{lemma}

\begin{proof}
For convenience, we translate \(x^0\) to the origin. It follows from
\(\tilde\mu_r(B)=H_0r^{-3/2}\mu(rB)\) and 
Lemma~\ref{Lemma: weak v.s variational} that \(\tilde\mu_r\ge0\) and
\eqref{eq.pc-natural-flux-boundary} hold.

It follows from the Jordan--Schoenflies theorem \cite[Chs.~9--10, pp.~65--80]{Moise77}
and the curve assumption that there exists a disk \(U\subset\subset\Omega\) 
containing the origin such that
\(U\setminus\partial\{u>0\}\) has exactly two connected open components
\(U_+\) and \(U_0\). Moreover, \(\chi_{\{u>0\}}=1\) on \(U_+\) and
\(\chi_{\{u>0\}}=0\) on \(U_0\), and with
\(U\cap\partial\{u>0\}\) as the relative boundary of each.

Let now \(r_j\to0^+\) and let
\(z_j\in\partial\{u_{r_j}>0\}\) be a sequence of points converge to
\(z\in\{x_2>0\}\). Choose \(d>0\) so that
\(\overline{B_{3d}(z)}\subset\{x_2>0\}\). Since \(U\) is open and
\(r_j\to0\), we get \(r_jB_{3d}(z)\subset U\) for all large \(j\).
Each rescaled side \(r_j^{-1}U_+\) and \(r_j^{-1}U_0\) contains points
converging to \(z_j\), and also points lying outside \(B_d(z_j)\) for
large \(j\). Using the connectedness of these sides, we deduce that
every circle \(\partial B_s(z_j)\) meets both sides in
nonempty open arcs for \(d/2<s<d\). Hence the \(BV\) slice
\[
	\theta\longmapsto
	\chi_{\{u_{r_j}>0\}}(z_j+s(\cos\theta,\sin\theta))
\]
has total angular variation at least \(2\) for a.e. \(s\). Thanks to
the polar
slicing in \cite{MR3409135} and the local \(BV\) bound of
Lemma~\ref{Lemma: weak v.s variational}, we obtain the estimate
\begin{equation}\label{eq.pc-curve-perimeter-crossing}
	\begin{split}
		&|D\chi_{\{u_{r_j}>0\}}|
		 \bigl(B_d(z_j)\setminus\overline{B_{d/2}(z_j)}\bigr)\\
		&\quad\ge\int_{d/2}^d
		 \operatorname{Var}_{\mathbb S^1}
		 \left[\theta\longmapsto
		 \chi_{\{u_{r_j}>0\}}(z_j+s(\cos\theta,\sin\theta))\right]
		 \,\dx{s}\ge d.
	\end{split}
\end{equation}

Assume now \(x^0\in\Sigma_u\). It follows from \eqref{eq.pc-natural-zero} and 
\eqref{eq.fd-hat-a-estimates} that
\[
	\langle\tilde\mu_r,\zeta\rangle
	=-\int\hat a_r\nabla u_r\cdot\nabla\zeta\,\dx{x}
	\longrightarrow0,
\]
for \(\zeta\in C_c^\infty(\R^2)\). Since \(\tilde\mu_r\ge0\), this implies that
\(\tilde\mu_r\stackrel{*}{\rightharpoonup}0\).

If some compact \(K\subset\subset\{x_2>0\}\) met the rescaled free
boundary at arbitrarily small scales, we could choose
\(z_j\in K\cap\partial\{u_{r_j}>0\}\) with \(z_j\to z\in K\). Pick
\(d>0\) and a cutoff function \(\zeta\in C_c^\infty(B_{2d}(z))\) which satisfies
\(0\le\zeta\le1\) and \(\zeta=1\) on \(B_{3d/2}(z)\). Note that
\(B_d(z_j)\subset B_{3d/2}(z)\) for \(j\) sufficiently large.
Thus we may deduce from
\eqref{eq.pc-natural-flux-boundary} and
\eqref{eq.pc-curve-perimeter-crossing} that
\[
	\begin{aligned}
		\langle\tilde\mu_{r_j},\zeta\rangle
		&\ge\frac{1}{2\sqrt{C_0}}
		\int_{B_d(z_j)\setminus\overline{B_{d/2}(z_j)}}\sqrt{x_2}
		\,\dx{|D\chi_{\{u_{r_j}>0\}}|}\\
		&\ge\frac{d}{2\sqrt{C_0}}
		\min_{\overline{B_{2d}(z)}}\sqrt{x_2}>0.
	\end{aligned}
\]
This yields a contradiction to \(\tilde\mu_{r_j}\stackrel{*}{\rightharpoonup}0\).

Thus we conclude that \(K\cap\partial\{u_r>0\}=\varnothing\) 
for all \(r>0\) sufficiently small. 
Choose a ball \(B\) with \(K\subset B\) and \(\overline B\subset\{x_2>0\}\).
Applying the preceding conclusion to the compact set \(\overline B\) gives
\(\overline B\cap\partial\{u_r>0\}=\varnothing\) for all sufficiently small \(r\).
Thus \(B\cap\{u_r>0\}\) is both open and closed in the connected ball \(B\).
By Proposition~\ref{Proposition: 2-dimensional case},
\(\chi_{\{u_r>0\}}\to1\) in \(L^1(B)\), so for all sufficiently small \(r\)
the positive phase meets \(B\) and therefore contains \(B\).
Hence \(K\subset\{u_r>0\}=\{v_r>0\}\).
\end{proof}

\begin{remark}
Let us remark that the Radon measures defined in
Lemmas~\ref{lem.fd-preliminary-normalized-flux} and
\ref{lem.pc-phase-stabilization} satisfy
\[
	\mu_r=\frac{r^{3/2}}{A_0(r)}\tilde\mu_r.
\]
At full weighted density points, we have \(A_0(r)/r^{3/2}\to0\). This implies that the
normalized measures \(\mu_r\) may converge to the nonzero measure
\(\Delta\mathcal P_N\) even though
\(\tilde\mu_r\stackrel{*}{\rightharpoonup}0\). The local curve
argument removes their support from each compact subset of the upper
half-plane, as used in the following theorem.
\end{remark}

\Needspace{8\baselineskip}
\begin{theorem}
\label{thm.pc-exclusion}
Let \(n=2\) and let \(u\) be a subsonic variational solution of
\eqref{eq.quaslinear-free-boundary-problem2}. Assume that \(u\)
satisfies the growth condition \eqref{eq.fd-growth}. Suppose that the
local curve hypothesis holds at the stagnation point \(x^0\in S_u\). Then
\[
	x^0\notin\Sigma_u.
\]
\end{theorem}

\begin{remark}
If the curve hypothesis holds at every stagnation
point in an open set \(O\subset\subset\Omega\), then
\(\Sigma_u\cap O=\varnothing\).
\end{remark}

\begin{proof}
Let \(x^0\in\Sigma_u\) be a full weighted density point, 
we may translate \(x^0\) to zero without loss of generality.
Fix any ball \(B\subset\subset\{x_2>0\}\). Applying Lemma~\ref{lem.pc-phase-stabilization} to \(\overline B\), we have
\(B\subset\{u_r>0\}=\{v_r>0\}\) for all sufficiently small \(r\). Moreover, we have
\[
	\operatorname{div}(\hat a_r\nabla v_r)=0\quad\hbox{in }B.
\]
It follows from the second estimate in \eqref{eq.fd-hat-a-estimates} and the strong
\(W^{1,2}_{\mathrm{loc}}\) convergence in
Theorem~\ref{thm.pc-strong-convergence} that \(\Delta v_0=0\) in
\(B\), and hence throughout the connected upper half-plane. On the other hand,
By Theorem~\ref{thm.pc-quantization}, we know that
\(v_0=\mathcal P_N\) for an integer \(N\ge2\). Note that the function \(\mathcal P_N\)
is nonnegative, nonzero, and vanishes on the interior ray
\(\theta=\pi/N\), which contradicts the strong maximum principle. This completes the proof.
\end{proof}

We also have the following description of the free surface near the Stokes corner.

\begin{proposition}
\label{prop.pc-Stokes-tangents}
Let \(n=2\) and let \(u\) be a subsonic variational solution of
\eqref{eq.quaslinear-free-boundary-problem2}. Assume that \(u\)
satisfies the growth condition \eqref{eq.fd-growth}. 

Suppose that \(\Phi(x^0,u;0^+)=\sqrt3/(3H_0)\) and in addition that 
the local curve assumption
holds at \(x^0\) with parametrization \(\sigma\)
as in Definition~\ref{def.pc-local-curve}. Then the two one-sided
limits
\[
	\lim_{t\to0^\pm}
	\frac{\sigma(t)-x^0}{|\sigma(t)-x^0|}
\]
exist and satisfies
\begin{equation}\label{eq.pc-four-directions}
	\left\{
	\lim_{t\to0^+}\frac{\sigma(t)-x^0}{|\sigma(t)-x^0|},
	\;
	\lim_{t\to0^-}\frac{\sigma(t)-x^0}{|\sigma(t)-x^0|}
	\right\}
	=
	\left\{
	\left(\frac{\sqrt3}{2},\frac12\right),
	\left(-\frac{\sqrt3}{2},\frac12\right)
	\right\}.
\end{equation}
\end{proposition}

\begin{proof}
For convenience, we translate \(x^0\) to the origin. It follows from Proposition~\ref{Proposition: 2-dimensional case} that
\begin{equation}\label{eq.pc-function-convergence}
	u_r\longrightarrow u_{\mathrm{st}}
\end{equation}
locally uniformly and strongly in \(W^{1,2}_{\mathrm{loc}}\), and
\begin{equation}\label{eq.pc-phase-convergence}
	\chi_{\{u_r>0\}}\longrightarrow\chi_{\mathcal C_{\mathrm{st}}}
\end{equation}
strongly in
\(L^1_{\mathrm{loc}}\). Combine these convergence with
\(\hat a_r\to1\) in \(L^\infty_{\mathrm{loc}}\) and using \eqref{eq.pc-natural-flux-boundary}, we get
\[
	\tilde\mu_r\stackrel{*}{\rightharpoonup}\Delta u_{\mathrm{st}},
\]
which is a measure supported on the two Stokes rays.

Let \(t_j\to0\) be a vanishing sequence, set \(r_j=|\sigma(t_j)|\) and
\(z_j=\sigma(t_j)/r_j\). Then up to a subsequence, we have
\(z_j\to z\in\mathbb S^1\) with \(z_2\ge0\). Note that if \(z_2>0\) and \(z\)
is not a Stokes direction, one may choose \(d>0\) so that
\(\overline{B_{3d}(z)}\subset\{x_2>0\}\) and \(B_{3d}(z)\) misses the
Stokes rays. Thanks to \eqref{eq.pc-curve-perimeter-crossing}, we have
\[
	|D\chi_{\{u_{r_j}>0\}}|
	\bigl(B_d(z_j)\setminus\overline{B_{d/2}(z_j)}\bigr)\ge d.
\]
Choose a cutoff function \(\zeta\in C_c^\infty(B_{2d}(z))\) with
\(0\le\zeta\le1\) and \(\zeta=1\) on
\(B_{3d/2}(z)\supset B_d(z_j)\). Testing
\eqref{eq.pc-natural-flux-boundary} with this \(\zeta\) yields
\[
	\langle\tilde\mu_{r_j},\zeta\rangle
	\ge c_d:=\frac{d}{2\sqrt{C_0}}
	\min_{\overline{B_{2d}(z)}}\sqrt{x_2}>0.
\]
This gives a contradiction to 
\(\tilde\mu_r\stackrel{*}{\rightharpoonup}\Delta u_{\mathrm{st}}\). Hence
\begin{equation}\label{fourpt}
	z\in\left\{(1,0),
	\left(\frac{\sqrt3}{2},\frac12\right),
	\left(-\frac{\sqrt3}{2},\frac12\right),(-1,0)\right\}.
\end{equation}
Since \(t\mapsto\sigma(t)/|\sigma(t)|\) is continuous on each of the
two intervals \((0,\tau_0)\) and \((-\tau_0,0)\), and its cluster set
is a connected subset of the finite set in
\eqref{fourpt}, each branch has a unique limiting
direction.

It remains to identify these two points. Let \(p\) be a point 
which satisfies
\(|p|=1\) on a Stokes ray. Then the ball
\(B_\varepsilon(p)\subset\subset\{x_2>0\}\) meets both
\(\mathcal C_{\mathrm{st}}\) and
\(\R^2\setminus\overline{\mathcal C_{\mathrm{st}}}\) in open sets
for every sufficiently small
\(\varepsilon>0\).
Therefore \(B_\varepsilon(p)\) 
contains a ball \(B'\subset\mathcal C_{\mathrm{st}}\) and a
ball \(B''\subset\R^2\setminus\overline{\mathcal C_{\mathrm{st}}}\).
Using \eqref{eq.pc-function-convergence}, we know that \(B'\) 
contains a point of
\(\{u_r>0\}\) for all \(r>0\) sufficiently small, while \(B''\) contains 
a point of
\(\{u_r=0\}\) by
\eqref{eq.pc-phase-convergence} for all \(r>0\) sufficiently small.

A similar argument near each of the two Stokes
rays shows that the there are 
free-boundary points \(r_jp_j\in\partial\{u>0\}\) whose
normalized directions \(p_j\) converge to the chosen ray. Since each
branch has a unique direction limit, the two one-sided limits are
exactly the two Stokes directions. Their horizontal components
\(\pm\sqrt3/2\) are nonzero and of opposite sign, which yields the
asserted sign change.
\end{proof} 

We next prove the non-existence of full weighted density points under 
the local finite-component hypothesis.

\begin{lemma}\label{lem.fc-single-component}
Let \(n=2\), let \(u\) be a subsonic variational solution of
\eqref{eq.quaslinear-free-boundary-problem2}, and let
\(x^0\in\Omega\cap\{u=0\}\). Assume that \(u\) satisfies the growth
condition \eqref{eq.fd-growth} and that \(x^0\) has a relatively compact
open neighborhood \(U\subset\subset\Omega\) meeting only finitely many
connected components of \(\Omega\cap\{u=0\}\). 
Then there exist
\(\rho_0>0\) and a connected component \(\mathscr C\) of
\(\Omega\cap\{u=0\}\) containing \(x^0\) such that
\begin{equation}\label{eq.fc-single-component}
	\{u=0\}\cap B_{\rho_0}(x^0)\subset\mathscr C.
\end{equation}
\end{lemma}

\begin{proof}
Since the components of a relatively closed set are relatively closed,
we know that each component of \(\Omega\cap\{u=0\}\) is relatively closed in
\(\Omega\). Among the finitely many components meeting \(U\), let
\(\mathscr C\) be the one containing \(x^0\). The union of the
remaining components is relatively closed in \(\Omega\) and does not
contain \(x^0\). Hence there is a sufficiently small ball about \(x^0\), contained in
\(U\), that is disjoint from that union. This proves
\eqref{eq.fc-single-component}. 
\end{proof}

\begin{lemma}
\label{lem.fc-stabilization}
Let \(n=2\), let \(u\) be a subsonic weak solution of
\eqref{eq.quaslinear-free-boundary-problem2}, and let \(x^0\in\Sigma_u\). 
Assume that \(u\)
satisfies the growth condition \eqref{eq.fd-growth}.  

Suppose that
there are \(\rho_0>0\) and a connected subset
\(\mathscr C\subset\Omega\cap\{u=0\}\) containing \(x^0\)
such that \eqref{eq.fc-single-component} holds. Then, for every
compact set \(K\subset\subset\{x_2>0\}\),
\begin{equation}\label{eq.fc-stabilization}
	K\subset\{u_r>0\}=\{v_r>0\},
\end{equation}
for all \(r>0\) sufficiently small.
\end{lemma}

\begin{proof}
Without loss of generality, we translate \(x^0\) horizontally to zero, and define 
\begin{equation}\label{eq.fc-natural-flux}
	\hat a_r(x):=
	\frac{H_0}{H(r|\nabla u_r(x)|^2;rx_2)},\qquad
	\tilde\mu_r:=\operatorname{div}
	(\hat a_r\nabla u_r).
\end{equation}
Since \(r|\nabla u_r(x)|^2=|\nabla u(rx)|^2\le C_0rx_2^+\),
we get
\begin{equation}\label{eq.fc-freezing}
	c\le\hat a_r\le C,\qquad
	\|\hat a_r-1\|_{L^\infty(B_R)}\le C_Rr,
\end{equation}
for each fixed \(R\).
The positive Radon measure scales according to
\[
	\tilde\mu_r(B)=H_0r^{-3/2}\mu(rB)\ge0,
\]
and we deduce from \eqref{eq.pc-natural-zero} that 
\[
	\langle\tilde\mu_r,\zeta\rangle
	=-\int\hat a_r\nabla u_r\cdot\nabla\zeta\,\dx{x}
	\longrightarrow0,
\]
for every compactly supported smooth test function \(\zeta\).
It follows from the positivity 
that \(\tilde\mu_r\stackrel{*}{\rightharpoonup}0\) locally. 
For
weak solutions, we recall
\eqref{eq.flux-perimeter-measure} and have
\begin{equation*}
	\tilde\mu_r
	=\sqrt{x_2}\,\mathcal H^1\mres
	\partial_{\mathrm{reg}}\{u_r>0\}
	\qquad\text{in }\{x_2>0\}.
\end{equation*}
Let \(K'\subset\subset\{x_2>0\}\) be compact, then the weight
\(\sqrt{x_2}\) has a positive lower bound, while the exceptional
part of the topological free boundary has zero \(\mathcal H^1\)
measure. It follows that
\begin{equation}\label{eq.fc-length-zero}
	\mathcal H^1\bigl(K'\cap\partial\{u_r>0\}\bigr)
	\longrightarrow0.
\end{equation}
Suppose now that \eqref{eq.fc-stabilization} fails. Then there is
a sequence \(r_j\to0^+\) and points \(z_j\in K\cap\{u_{r_j}=0\}\) such that
\[
	z_j\longrightarrow z\in K.
\]
Choose \(d>0\) such that
\(\overline{B_{4d}(z)}\subset\subset\{x_2>0\}\), and take
\(j\) large enough so that \(z_j\in B_{d/4}(z)\) and
\(r_jz_j\in B_{\rho_0}\). Using the inclusion
\eqref{eq.fc-single-component}, one may place \(z_j\) in the connected
set \(\mathscr C_j:=r_j^{-1}\mathscr C\) which contains
zero. Note that its image under the continuous map \(x\mapsto|x-z_j|\)
is an interval containing both \(0\) and \(|z_j|>d\), thus
\begin{equation}\label{eq.fc-circle-zero}
	\partial B_s(z_j)\cap\{u_{r_j}=0\}\ne\varnothing
	\qquad(d/2<s<d).
\end{equation}

Define the set
\[
	\mathcal T_j:=\{s\in(d/2,d):
	\partial B_s(z_j)\subset\{u_{r_j}=0\}\}.
\]
Then \(\mathcal T_j\) is measurable.
Since \(B_d(z_j)\subset B_{2d}(z)\), we deduce from
\eqref{eq.pc-phase-density} and calculate directly to see that
\begin{equation*}
	\pi d\,\mathcal L^1(\mathcal T_j)
	\le\int_{\mathcal T_j}2\pi s\,\dx{s}
	\le\bigl|B_{2d}(z)\cap\{u_{r_j}=0\}\bigr|
	\longrightarrow0.
\end{equation*}
By \eqref{eq.fc-circle-zero}, the circle
contains a zero point
for every \(s\in(d/2,d)\setminus\mathcal T_j\), and contains also a positive
point by the definition of \(\mathcal T_j\). Therefore 
\(\partial B_s(z_j)\cap\partial\{u_{r_j}>0\}\ne\varnothing\)
for every \(s\in(d/2,d)\setminus\mathcal T_j\). Hence the radial projection of
\(\partial\{u_{r_j}>0\}\cap\bigl(B_d(z_j)\setminus\overline{B_{d/2}(z_j)}\bigr)\)
contains \((d/2,d)\setminus\mathcal T_j\). Using the coarea formula, we obtain the 
estimate 
\begin{equation}\label{eq.fc-length-lower}
	\begin{split}
		\mathcal H^1\bigl(\overline{B_{2d}(z)}\cap
		\partial\{u_{r_j}>0\}\bigr)
		&\ge\mathcal H^1(\partial\{u_{r_j}>0\}\cap\bigl(B_d(z_j)\setminus\overline{B_{d/2}(z_j)}\bigr))\\
		&\ge\frac d2-\mathcal L^1(\mathcal T_j)
		=\frac d2-o(1),
	\end{split}
\end{equation}
This yields a contradiction to \eqref{eq.fc-length-zero} and proves
\eqref{eq.fc-stabilization}.
\end{proof}

\begin{remark}
Let us remark that the finite-component hypothesis is used only through
\eqref{eq.fc-single-component}: it places the zero set near the origin
inside a single connected component \(\mathscr C\). The connectedness
of \(\mathscr C\) then gives the estimate for length
\eqref{eq.fc-length-lower}, which contradicts the vanishing of the
natural-scale flux.
\end{remark}

\begin{theorem}
\label{thm.fc-flat-exclusion}
Let \(n=2\) and let \(u\) be a subsonic weak solution of
\eqref{eq.quaslinear-free-boundary-problem2}. Assume that \(u\)
satisfies the growth condition \eqref{eq.fd-growth}. 
Assume in addition that the local finite-component hypothesis
holds at stagnation point \(x^0\in S_u\), then \(x^0\notin\Sigma_u\).
\end{theorem}

\begin{proof}
We translate \(x^0\in\Sigma_u\) horizontally to zero.
Thanks to Lemma~\ref{lem.fc-single-component}, we have 
\eqref{eq.fc-single-component} holds. It follows from  
Lemma~\ref{lem.fc-stabilization} that we may place every fixed ball
\(B\subset\subset\{x_2>0\}\) in \(\{v_r>0\}\) for
sufficiently small \(r\). Then we have
\[
	\operatorname{div}(\hat a_r\nabla v_r)=0
	\quad\text{in }B,
\]
where \(\hat a_r\) is defined in \eqref{eq.fc-natural-flux}.
It follows from Theorem~\ref{thm.pc-quantization} 
that \(v_r\to v_0=\mathcal P_N\) for an integer
\(N\ge2\). For every \(\varphi\in C_c^\infty(B)\), it 
follows from \eqref{eq.fc-freezing} and 
\(\nabla v_r\rightharpoonup\nabla v_0\) 
in \(B\) that
\[
	0=\lim_{r_j\to0^+}
	   \int_B\hat a_{r_j}\nabla v_{r_j}\cdot\nabla\varphi\,\dx{x}
	   =\int_B\nabla v_0\cdot\nabla\varphi\,\dx{x}.
\]
Thus \(v_0\) is harmonic throughout the upper half-plane. Since
\(v_0\ge0\) and \(v_0\not\equiv0\) on \(B\), the strong maximum
principle gives \(v_0>0\) is strictly positive 
in \(B\). On the other hand, we infer from  Theorem~\ref{thm.pc-quantization}
that \(v_0=\mathcal P_N\). As \(\mathcal P_N\) vanishes on
the interior ray \(\theta=\pi/N\), we are led to a contradiction. Thus 
\(x^0\notin\Sigma_u\).
\end{proof}

\section{Proofs of the main theorems}

In this section, we conclude the proofs of 
Theorem \ref{thm.main-A}, Theorem \ref{thm.main-B} and Theorem
\ref{thm.main-C}. Let \(n=2\) and recall the 
Definition~\ref{def.stagnation-hyperplane} and
\eqref{eq.fd-sets}, we have that 
\[
S_u=\{(a,0):(a,x_2^{\mathrm{st}})\in S\},\qquad
\Sigma_u=\{(a,0):(a,x_2^{\mathrm{st}})\in\Sigma\},
\]
where \(S\) and \(\Sigma\) are sets defined in 
\eqref{set.S} and \eqref{set.Sigma}, 
respectively.
Let \(x^{\mathrm{st}}=(x_1^{\mathrm{st}},x_2^{\mathrm{st}})\)
be a stagnation point, then it follows from \eqref{eq.upsi} that
\begin{equation}\label{eq.main-physical-return}
	\psi(x^{\mathrm{st}}+rx)
	=\bar\rho_0\sqrt{2g}\,
	  u\bigl((x_1^{\mathrm{st}},0)+r(x_1,-x_2)\bigr).
\end{equation}
Let \(n=2\) and let \(u\) be a subsonic variational solution of
\eqref{eq.quaslinear-free-boundary-problem2}. Assume that \(u\)
satisfies the growth condition \eqref{eq.fd-growth}. We infer from 
the first statement in
Lemma~\ref{Lemma: densities at the stagnation points} and \(\Theta\)
defined in \eqref{eq.normalized-weiss-density} that
\[
	\Theta(x^{\mathrm{st}})=\Phi((x_1^{\mathrm{st}},0),u;0^+).
\]
It follows from
Lemma~\ref{Lemma: weak v.s variational} and
Proposition~\ref{Proposition: 2-dimensional case} that 
\(\Theta(x^{\mathrm{st}})\) takes only three possible values given in 
\eqref{eq.bu-three-densities}. Thanks to Theorem~\ref{thm.cusp-exclusion}, 
we know that every stagnation point
is either a Stokes point or a point of
\(\Sigma_u\). Moreover, the set of stagnation points 
\(S_u\) is relatively closed.

\begin{proof}[Proof of Theorem~\ref{thm.main-A}]
It follows from Theorem~\ref{thm.fc-flat-exclusion} that
\(\Sigma_u=\varnothing\), and from
Corollary~\ref{cor.fd-density-alternatives} that every point of
\(S_u\) has weighted density \(\sqrt{3}/(3H_0)\). It follows from
Lemma~\ref{Lemma:measure1} and \eqref{eq.main-physical-return}
that \(r^{-3/2}\psi(x^{\mathrm{st}}+rx)\)
converges to \(\psi_{\mathrm{st}}\) locally uniformly and strongly in
\(W^{1,2}_{\mathrm{loc}}\). 
Finally, it follows from
Proposition~\ref{Proposition: iso-stp} \(S_u\) is closed and 
is locally finite.
\end{proof}

\begin{proof}[Proof of Theorem~\ref{thm.main-B}]
It follows from Proposition~\ref{Proposition: iso-stp} that the set
\(S_u\setminus\Sigma_u\) is discrete, so it is countable. It then follows 
from
Proposition~\ref{prop.pc-local-finiteness} that \(\Sigma_u\) is
locally finite, hence countable on a countable compact exhaustion. Thus
\(S_u\) is at most countable. It follows from
relative closedness of \(S_u\) that the limit of a convergent sequence
of distinct stagnation points lies in \(S_u\), and from
Corollary~\ref{cor.bu-Stokes-accumulation} that, if a subsequence
consists of Stokes points, this limit lies in \(\Sigma_u\). It follows
from Lemma~\ref{Lemma:measure1} and
Proposition~\ref{Proposition: 2-dimensional case} that the stated
Stokes limits hold after \eqref{eq.main-physical-return}, and from
Theorem~\ref{thm.pc-quantization} and
Corollary~\ref{cor.pc-natural-degeneracy} that the corresponding full
density limits hold.
\end{proof}

\begin{proof}[Proof of Theorem~\ref{thm.main-C}]
It follows from Corollary~\ref{cor.fd-density-alternatives} and
Theorem~\ref{thm.pc-exclusion} that the point \((x_1^{\mathrm{st}},0)\in S_u\)
has Stokes density \(\sqrt{3}/(3H_0)\), hence is not in \(\Sigma_u\). It then follows from
Proposition~\ref{prop.pc-local-finiteness} that \(\Sigma_u\) is locally
finite. By Proposition~\ref{Proposition: iso-stp} we deduce that it
is isolated in \(S_u\). It follows from Lemma~\ref{Lemma:measure1},
Proposition~\ref{Proposition: 2-dimensional case} and
\eqref{eq.main-physical-return} that the asserted function and phase
limits hold. It follows from Proposition~\ref{prop.pc-Stokes-tangents}
that the normalized directions are
\((\sqrt3/2,1/2)\) and \((-\sqrt3/2,1/2)\), and after reflection that
they are \((\sqrt3/2,-1/2)\) and \((-\sqrt3/2,-1/2)\); the horizontal
components are nonzero and of opposite sign, so the branches lie on
opposite sides of \(x_1=x_1^{\mathrm{st}}\), with secant slopes
\(-1/\sqrt3\) and \(1/\sqrt3\). Finally, it follows from the compactness
argument in the proof of Theorem~\ref{thm.main-A} that \(S\cap O\) is
locally finite whenever the curve hypothesis holds throughout \(O\).
\end{proof}


\appendix

\section{Proof of Lemma \ref{Lemma: weak v.s variational}}\label{Appendix: proof of weak v.s variational}

\begin{proof}
	Let \(u\) be a subsonic variational solution. Recall Remark \ref{rem.flux-measure}, we know that
	\begin{equation}\label{eq.appendix-flux-measure}
		\mu:=\operatorname{div}\left(\frac{\nabla u}{H(|\nabla u|^2;x_n)}\right)
	\end{equation}
	is a nonnegative Radon measure supported on \(\Omega\cap\partial\{u>0\}\). We divide the proof into three steps.

	\emph{Step I.}
	Let \(B\subset\subset\Omega\cap\{x_n>0\}\) be a fixed ball and let \(L\) be a Lipschitz constant of \(u\) on \(B\). For sufficiently small \(\varepsilon\), let us consider
	\[
		u_\varepsilon:=u\circ(\operatorname{Id}+\varepsilon\phi),
		\qquad
		K(\varepsilon):=\int_B F(|\nabla u_\varepsilon|^2;x_n)\,\dx{x},
	\]
	where \(\phi\in C_c^1(B;\R^n)\). For fixed \(s\), a direct computation gives that
	\[
		D_P^2\bigl[F(|P|^2;s)\bigr]
		=\frac{2}{H(|P|^2;s)}\operatorname{Id}
		-\frac{4\partial_1H(|P|^2;s)}{H(|P|^2;s)^2}P\otimes P
	\]
	is positive definite on the ball \(\{|P|^2<t_0\}\). Hence, for each fixed \(s\), the function \(P\mapsto F(|P|^2;s)\) is convex. Thanks to \eqref{eq.subsonic-admissibility}, we get that \((|\nabla u|^2,x_n)\) stays in a compact subset of \(\cR\) on \(B\). Thus, both \(\nabla u\) and \(\nabla u_\varepsilon\) and the segment joining them lie in \(\{|P|^2<t_0\}\) for sufficiently small \(\varepsilon\). Using the convexity, we have 
	\[
		F(|\nabla u_\varepsilon|^2;x_n)-F(|\nabla u|^2;x_n)
		\ge\frac{2\nabla u}{H(|\nabla u|^2;x_n)}
		\cdot(\nabla u_\varepsilon-\nabla u),
	\]
	for a.e.\ \(x\in B\). Integration the inequality over \(B\), we get
	\begin{align*}
		K(\varepsilon)-K(0)
		&\ge2\int_B\frac{\nabla u}{H(|\nabla u|^2;x_n)}
		\cdot\nabla(u_\varepsilon-u)\,\dx{x}\\
		&=-2\int_B(u_\varepsilon-u)\,\dx{\mu}\\
		&\ge-2L|\varepsilon|\int_B|\phi|\,\dx{\mu}.
	\end{align*}
	Thus, we obtain
	\[
		|K'(0)|\le2L\int_B|\phi|\,\dx{\mu}.
	\]
	The same change of variables implies
	\[
		\left.\frac{\dx{}}{\dx{\varepsilon}}
		\int_B\lambda(x_n)\chi_{\{u_\varepsilon>0\}}\,\dx{x}
		\right|_{\varepsilon=0}
		=-\int_B\chi_{\{u>0\}}
		\operatorname{div}(\lambda(x_n)\phi)\,\dx{x}.
	\]
	Using \eqref{eq.domain-variation-formula}, we obtain
	\[
		\left|\int_B\chi_{\{u>0\}}
		\operatorname{div}(\lambda(x_n)\phi)\,\dx{x}\right|
		\le2L\int_B|\phi|\,\dx{\mu}.
	\]
	Since \(x_n>0\) on \(\overline B\), the function \(\lambda(x_n)\) is \(C^1\) and strictly positive on \(\overline B\). Hence every test function \(\zeta\in C_c^1(B;\R^n)\) can be written as \(\zeta=\lambda(x_n)\phi\) with \(\phi\in C_c^1(B;\R^n)\). Note that for this choice we have \(\operatorname{div}(\lambda(x_n)\phi)=\operatorname{div}\zeta\). The preceding inequality therefore becomes
	\[
		\left|\int_B\chi_{\{u>0\}}\operatorname{div}\zeta\,\dx{x}\right|
		\le2L\int_B\frac{|\zeta|}{\lambda(x_n)}\,\dx{\mu}.
	\] 
	It follows that
	\begin{equation}\label{eq.local-BV-domination}
		\chi_{\{u>0\}}\in BV(B),
		\qquad
		\lambda(x_n)\,\dx{|D\chi_{\{u>0\}}|}
		\le2L\,\dx{\mu}\quad\text{in }B.
	\end{equation}
    Since \(B\) was arbitrary, every compact subset of \(\Omega\cap\{x_n>0\}\)
    is covered by finitely many such balls. Note that \(|D\chi_{\{u>0\}}|\) is a
    Radon measure on this open set, the bounds on these balls add up, and we obtain the
    desired local \(BV\) property.

	Now fix an open localization \(O\subset\subset\Omega\) on which the local gradient
	estimate \eqref{eq.local-gradient-growth} holds with one constant \(C_0\), and set
	\[
		\nu:=\lambda(x_n)\,|D\chi_{\{u>0\}}|\qquad\text{on }O\cap\{x_n>0\}.
	\]
	Since \eqref{eq.local-BV-domination} holds for every ball \(B'\subset\subset\Omega\cap\{x_n>0\}\), it follows that \(\nu\ll\mu\) on \(O\cap\{x_n>0\}\). To bound the density of \(\nu\) with respect to \(\mu\), fix \(z\in O\cap\{x_n>0\}\) and take \(0<\rho\ll1\) such that \(B_\rho(z)\subset\subset O\cap\{x_n>0\}\). For \(x\in B_\rho(z)\) we have \(0<x_n\le z_n+\rho\), and it follows from \eqref{eq.local-gradient-growth} that \(\lvert\nabla u\rvert^2\le C_0(z_n+\rho)\) for a.e.\ \(x\in B_\rho(z)\). Thus \(u\) is a Lipschitz function on \(B_\rho(z)\) with constant \(L=\sqrt{C_0(z_n+\rho)}\). Applying \eqref{eq.local-BV-domination} on \(B_\rho(z)\) gives
	\[
		\nu(B_\rho(z))\le2\sqrt{C_0(z_n+\rho)}\,\mu(B_\rho(z)).
	\]
	Passing to the limit as \(\rho\to0^+\), we have
	\[
		\frac{\dx{\nu}}{\dx{\mu}}(z)
		=\lim_{\rho\to0^+}\frac{\nu(B_\rho(z))}{\mu(B_\rho(z))}
		\le2\sqrt{C_0z_n},
	\]
	for \(\mu\)-almost every \(z\in O\cap\{x_n>0\}\). Thus
	\begin{equation}\label{eq.local-weighted-domination}
		\lambda(x_n)\,\dx{|D\chi_{\{u>0\}}|}
		\le2\sqrt{C_0x_n}\,\dx{\mu}
		\quad\text{in }O\cap\{x_n>0\}.
	\end{equation}
	It follows from \eqref{eq.lambda} that \(\lambda(x_n)\ge x_n/H_0\) at \(x_n\ge 0\). Thus, we deduce from \eqref{eq.local-weighted-domination} that \(\sqrt{x_n}\,\dx{|D\chi_{\{u>0\}}|}\le2H_0\sqrt{C_0}\,\dx{\mu}\). Thus we prove two inequalities in \eqref{eq.variational-perimeter-domination}.

	\emph{Step II.}
	For any ball \(B_{2\rho}(z)\subset\subset\Omega\), set 
	\[
		\varphi(x):=\min\Bigl\{1,\ \frac{(2\rho-\lvert x-z\rvert)^+}{\rho}\Bigr\}.
	\] 
	It is easy to check that \(0\le\varphi\le1\), \(\varphi=1\) on 
	\(B_\rho(z)\), \(\operatorname{supp}\varphi\subset\overline{B_{2\rho}(z)}\), 
	and \(|\nabla\varphi|\le1/\rho\) a.e. Testing the 
	nonnegative Radon measure \eqref{eq.appendix-flux-measure} with \(\varphi\) yields
	\begin{equation}\label{eq.flux-ball-bound}
		\mu(B_\rho(z))
		\le\frac{C}{\rho}\int_{B_{2\rho}(z)}
		\frac{|\nabla u|}{H(|\nabla u|^2;x_n)}\,\dx{x}.
	\end{equation}
	We now apply \eqref{eq.flux-ball-bound} to \(B_{2r}(y)\subset\subset\Omega\) with \(y_n=0\). Thanks to \eqref{eq.local-gradient-growth} and \(H(|\nabla u|^2;x_n)\ge H_{\rm min}\), we have
	\begin{equation}\label{eq.flux-bound-halfspace}
		\mu(B_r(y))
		\le C(n)\frac{\sqrt{C_0}}{H_{\min}}\,r^{n-1/2}.
	\end{equation}
	We now integrate \eqref{eq.variational-perimeter-domination} over \(B_r(y)\cap\{x_n>\tau\}\) for \(\tau>0\), which gives
	\[
		\int_{B_r(y)\cap\{x_n>\tau\}}\sqrt{x_n}\,\dx{|D\chi_{\{u>0\}}|}\le2H_0\sqrt{C_0}\,\mu(B_r(y)\cap\{x_n>\tau\}).
	\]
	Passing to the limit as \(\tau\to0^+\), we have
	\[
		\begin{aligned}
			\int_{B_r(y)\cap\{x_n>0\}}\sqrt{x_n}\,\dx{|D\chi_{\{u>0\}}|}
			&\le2H_0\sqrt{C_0}\,\mu(B_r(y))\\
			&\le C(n)\frac{H_0C_0}{H_{\min}}\,r^{n-1/2},
		\end{aligned}
	\]
	where we used \eqref{eq.flux-bound-halfspace} in the
	last inequality. 
	This proves \eqref{eq.weighted-perimeter-bound} for variational solutions.

	We now prove the estimate
	\begin{equation}\label{eq.mu-Pi-zero}
		\mu(\Omega\cap\Pi)=0.
	\end{equation}
	Let \(K\subset\subset\Omega\) and choose \(K'\subset\subset\Omega\) with \(K\subset\subset K'\subset\subset\Omega\). Cover \(K\cap\Pi\) by at most \(C_K\rho^{1-n}\) balls \(B_\rho(y_i)\) of radius \(\rho\) centered on \(\Pi\). Note that the balls \(B_{2\rho}(y_i)\) lie in \(K'\) for sufficiently small \(\rho\). We obtain that \eqref{eq.local-gradient-growth} holds on \(K'\) with one fixed constant \(C_0\). It follows that \eqref{eq.flux-bound-halfspace} holds on each \(B_\rho(y_i)\) with the same \(C_0\). Consequently,
	\[
		\mu(K\cap\Pi)
		\le\sum_i\mu(B_\rho(y_i))
		\le C_K\rho^{1-n}\rho^{n-1/2}
		=C_K\rho^{1/2}\longrightarrow0.
	\]
	Since \(K\subset\subset\Omega\) was arbitrary, we have \(\mu(\Omega\cap\Pi)=0\).

	\emph{Step III.}
	Let now \(u\) be a subsonic variational solution. Since 
	\(\nabla u\in L^\infty_{\mathrm{loc}}(\Omega)\) and \(H\ge H_{\min}>0\) 
	(cf. equations \eqref{eq.subsonic-admissibility} and \eqref{eq.esm-H}), the flux \(\nabla u/H(|\nabla u|^2;x_n)\) is locally bounded. 
	Hence, for every compact subset \(K\subset\subset\Omega\), 
	there is a constant \(C_K\) such that \(\mu(B_\rho(z))\le C\rho^{n-1}\) 
	for every ball \(B_{2\rho}(z)\subset\subset K\) where 
	\(K\subset\subset\Omega\) is compact. A covering argument 
	implies that \(\mu\) vanishes on sets of zero \(\cH^{n-1}\)-measure. 
	Moreover, since \(K\subset\subset\Omega\cap\{x_n>0\}\) and \(\lambda(x_n)\ge x_n/H_0>0\) for \(x_n>0\) (cf. equation \eqref{eq.lambda}), there is \(c_K>0\) such that \(\lambda(x_n)\ge c_K\) on \(K\). We now cover the set \(K\) by finitely many balls \(B\subset\subset\Omega\cap\{x_n>0\}\) and apply the estimate \eqref{eq.local-BV-domination} on each of them, we get \(\dx{|D\chi_{\{u>0\}}|}\le C_K\,\dx{\mu}\) for some constant \(C_K\). It follows that
	\(|D\chi_{\{u>0\}}|\ll\mu\ll\cH^{n-1}\) on \(K\). Note that the exceptional set
	\(\Gamma_+\setminus\partial_{\mathrm{reg}}\{u>0\}\) has zero \(\cH^{n-1}\)-measure. 
	Indeed,  by Definition \ref{Definition: subsonic weak solution}, we have
	\(\cH^{n-1}\bigl((\Gamma_+\setminus\partial_{\mathrm{reg}}\{u>0\})\cap\{x_n>1/j\}\bigr)=0\) for every \(j\), and the set is the union of these intersections over \(j\). Since both measures are locally absolutely continuous with respect to \(\cH^{n-1}\), neither of them charges this set.

	On each regular patch, it follows from the Gauss--Green formula and an integration by parts that
	\[
		\dx{|D\chi_{\{u>0\}}|}=\dH{n-1},
		\quad
		\dx{\mu}
		=-\frac{\nabla u\cdot\nu_u}{H(|\nabla u|^2;x_n)}\,\dH{n-1}
		=\frac{\sqrt{x_n}}{H_0}\,\dH{n-1}.
	\]
	Here the last equality follows from the arguments in Remark \ref{rem.weak-boundary-condition}. Using the above identities above and the zero \(\cH^{n-1}\)-measure of the set \(\Gamma_+\setminus\partial_{\mathrm{reg}}\{u>0\}\), we get \eqref{eq.flux-perimeter-measure}. Thanks to \eqref{eq.mu-Pi-zero}, we have that the zero extension of the weighted perimeter measure equals \(H_0\mu\) in \(\Omega\).
\end{proof}

\section{Proof of Lemma \ref{lem.two-dimensional-positive-laplacian-cancellation}}\label{Appendix: proof of two-dimensional positive laplacian cancellation}

\begin{proof}

Write \(\mu_j=\Delta w_j\ge0\) for simplicity. It follows that
\(\mu_j\) is a Radon measure. Let \(K\subset\subset O\) be 
compact and let \(\eta\in C_c^\infty(O)\) be such that \(\eta\ge1\) on 
\(K\) and \(\eta\ge0\), then we obtain
\[
	\mu_j(K)\le\int\eta\,\dx{\mu_j}
	=-\int\nabla w_j\cdot\nabla\eta\le C_\eta.
\]
Thus, \(\mu_j\stackrel{*}{\rightharpoonup}\mu=\Delta w\)
locally.

We first prove the statement for smooth \(w_j\). Let 
\(\varphi\in C_c^\infty(O)\) and choose
\(0\le\chi\in C_c^\infty(O)\) which equal to one on a neighborhood \(V\)
of \(\operatorname{supp}\varphi\). Define \(\nu_j=\chi\mu_j\), and set 
\begin{equation}\label{eq.two-dimensional-potential}
	p_j(x)=\frac1{2\pi}\int\log|x-y|\,\dx{\nu_j(y)},\qquad
	h_j=w_j-p_j.
\end{equation}
Since \(\nu_j=\chi\Delta w_j\) and \(w_j\rightharpoonup w\) in \(W^{1,2}_{\mathrm{loc}}(O)\), the measures \(\nu_j\) have common compact support, uniformly bounded total mass, and a uniform \(H^{-1}(\R^2)\) bound.

Let \(B_0\subset\R^2\) be a fixed bounded set. After replacing the fundamental 
solution \(\Gamma(x)=(2\pi)^{-1}\log|x|\) by its compactly supported 
truncation \(K=\kappa\Gamma\), where \(\kappa\in C_c^\infty(\R^2)\) equals 
one near the origin and on the set 
\(D:=\{x-y:\ x\in B_0,\ y\in\operatorname{supp}\nu_j\}\), 
we have \(\Delta K=\delta_0+f\) with \(f\in C_c^\infty\). 
It follows that \(\lvert\hat K(\xi)\rvert\le C/(1+\lvert\xi\rvert^2)\). 
Using Plancherel's formula, we get \(\|K*\nu_j\|_{H^1}\le C\|\nu_j\|_{H^{-1}}\). 
Since \(\kappa=1\) on \(D\), we obtain \(K*\nu_j=p_j\) on \(B_0\). 
As \(B_0\) was arbitrary, the potentials \(p_j\) are uniformly bounded in 
\(H^1\) on every fixed bounded set. We now claim that it suffices to establish the following two quadratic convergences for \(p_j\) defined in \eqref{eq.two-dimensional-potential}:
\begin{equation}\label{eq.two-dimensional-quadratic-p}
	\begin{aligned}
		(\partial_1p_j)^2-(\partial_2p_j)^2
		&\longrightarrow(\partial_1p)^2-(\partial_2p)^2,\\
		\partial_1p_j\partial_2p_j
		&\longrightarrow\partial_1p\partial_2p,
	\end{aligned}
	\qquad\text{in }\mathcal D'(V).
\end{equation}
Indeed, since \(\Delta h_j=\Delta w_j-\Delta p_j=(1-\chi)\mu_j=0\) on \(V\), the functions \(h_j\) are harmonic on \(V\). 
Also, \(p_j\rightharpoonup p=\Gamma*\nu\) locally in \(H^1\), 
where \(\nu=\chi\mu\); since \(h_j=w_j-p_j\) and \(h_j\rightharpoonup h=w-p\), 
the interior estimates for harmonic functions give \(\nabla h_j\to\nabla h\) locally uniformly on \(V\). 
Consequently, since \(\nabla h_j\to\nabla h\) locally uniformly, expanding the quadratic products of \(\nabla w_j=\nabla p_j+\nabla h_j\) reduces the assertion to \eqref{eq.two-dimensional-quadratic-p}, and this proves the claim.

We now identify \(\R^2\) with \(\mathbb C\) and 
define \(c_j=\partial_1p_j-i\partial_2p_j\), then
\[
	c_j(z)=\frac1{2\pi}\int\frac{\dx{\nu_j(\zeta)}}{z-\zeta}.
\]
Let \(\varphi\in C_c^\infty(\R^2)\) be a test function and let us define
\[
	C_\varphi(\zeta)=\int_{\R^2}\frac{\varphi(z)}{z-\zeta}\,\dx{z},\qquad
	K_\varphi(\zeta,\eta)=\frac{C_\varphi(\zeta)-C_\varphi(\eta)}{\zeta-\eta}\quad(\zeta\ne\eta).
\]
It follows from the smoothness of \(C_\varphi\) and the local integrability of \(1/z\) that \(K_\varphi\) is bounded on every compact product and continuous off its diagonal. Fubini's theorem and partial fractions then give
\begin{equation}\label{eq.two-dimensional-complex-kernel}
	\int\varphi\,c_j^2
	=\frac1{4\pi^2}
	\iint K_\varphi(\zeta,\eta)\,\dx{\nu_j(\zeta)}\,\dx{\nu_j(\eta)}.
\end{equation}
It remains to pass to the limit in this double integral. Let us remark that 
positive measure \(\nu\in H^{-1}(\R^2)\) has no atoms. The absence of atoms of \(\nu\), together with the weak convergence \(\nu_j\rightharpoonup\nu\), implies
\begin{equation}\label{eq.two-dimensional-no-diagonal-mass}
	\lim_{\delta\to0^+}\limsup_{j\to\infty}\sup_a\nu_j(B_\delta(a))=0;
\end{equation}
Let \(K_*\subset\R^2\) be a compact set containing the common supports of the \(\nu_j\) and \(\nu\), and let \(\theta_\delta:[0,\infty)\to[0,1]\) be smooth, equal to zero on \([0,\delta]\) and to one on \([2\delta,\infty)\). Define
\[
	K_{\varphi,\delta}(\zeta,\eta)=
	\begin{cases}
		\theta_\delta(|\zeta-\eta|)K_\varphi(\zeta,\eta),&\zeta\ne\eta,\\
		0,&\zeta=\eta.
	\end{cases}
\]
For fixed \(\delta>0\), \(K_{\varphi,\delta}\) is a bounded continuous kernel on \(K_*\times K_*\) that agrees with \(K_\varphi\) whenever \(|\zeta-\eta|\ge2\delta\), so the weak convergence gives
\[
	\iint K_{\varphi,\delta}\,\dx{\nu_j}\,\dx{\nu_j}
	\longrightarrow
	\iint K_{\varphi,\delta}\,\dx{\nu}\,\dx{\nu}.
\]
Moreover, if \(C_\varphi^*\) bounds \(|K_\varphi|\) on \(K_*\times K_*\), then
\begin{equation}\label{eq.two-dimensional-truncation-error}
	\begin{aligned}
		\left|\iint(K_{\varphi}-K_{\varphi,\delta})\,\dx{\nu_j}\,\dx{\nu_j}\right|
		&\le C_\varphi^*\,(\nu_j\otimes\nu_j)(\{|\zeta-\eta|<2\delta\})\\
		&\le C_\varphi^*M\sup_a\nu_j(B_{2\delta}(a)).
	\end{aligned}
\end{equation}
It follows from \eqref{eq.two-dimensional-no-diagonal-mass} that the limsup, 
as \(j\to\infty\), of the right-hand side of 
\eqref{eq.two-dimensional-truncation-error} tends to zero as \(\delta\to0^+\).
Since \(\nu\) is atomless, the same bound holds with \(\nu\) in place of \(\nu_j\), 
and its right-hand side also tends to zero as \(\delta\to0^+\). 
Hence 
\[
	\lim_{j\to\infty}\iint K_\varphi\,\dx{\nu_j}\,\dx{\nu_j}
	=\lim_{\delta\to0^+}\lim_{j\to\infty}
	\iint K_{\varphi,\delta}\,\dx{\nu_j}\,\dx{\nu_j}
	=\iint K_\varphi\,\dx{\nu}\,\dx{\nu}.
\]
It remains to prove
\[
	\iint K_\varphi\,\dx{\nu}\,\dx{\nu}
	=4\pi^2\int\varphi c^2,
\]
where \(c=\partial_1p-i\partial_2p\). 

The identity for \(\nu\) follows by approximating \(\nu\) by positive
mollifications, for which the potentials converge in \(H^1_{\mathrm{loc}}\)
and the double integrals converge by the preceding off-diagonal argument.
It remains to remove the smoothness assumption in
\eqref{eq.two-dimensional-complex-kernel}. Let \(\sigma\) be a compactly
supported element of \(H^{-1}(\R^2)\) and write
\(\sigma_\varepsilon=\rho_\varepsilon*\sigma\). It follows from 
\(\sigma_\varepsilon\to\sigma\) in \(H^{-1}(\R^2)\) that
\(\Gamma*\sigma_\varepsilon\) converges to \(\Gamma*\sigma\) in
\(H^1_{\mathrm{loc}}\). Hence their complex gradients converge in
\(L^2_{\mathrm{loc}}\) and their squares converge in \(L^1_{\mathrm{loc}}\).
Let us recall the off-diagonal argument above, we apply the argument to \(\sigma\) and
\(\sigma_\varepsilon\) with a common bound, it follows that the double integrals
converge as \(\varepsilon\to0^+\). Passing to the limit in
\eqref{eq.two-dimensional-complex-kernel} written for
\(\sigma_\varepsilon\) gives the identity for \(\sigma=\nu_j\) and \(\sigma=\nu\). This together with the
convergence of the double integrals established above gives
\(c_j^2\to c^2\) in the sense of distributions. Taking real and imaginary parts proves
\eqref{eq.two-dimensional-quadratic-p}, and the locally uniformly
convergent harmonic gradients give the corresponding convergences for
\(w_j=p_j+h_j\). This concludes the proof.

\end{proof}

\subsubsection*{Acknowledgement}
\hyphenpenalty=10
\sloppy
L.~Du is supported by the National Natural Science Foundation of China under Grants 12125102, 12526202, 12671245, 12631009. C.~Yang is supported by the National Natural Science Foundation of China under Grant 12601435.

\subsubsection*{Data availability} No data was used in this research.

\subsubsection*{Conflicts of interest}
The authors declare that they have no conflicts of interest.

\subsubsection*{AI declarations}
In this version, OpenAI’s GPT was used only for grammar, spelling and citation checks, with all edits reviewed to safeguard mathematical reasoning.


\begin{thebibliography}{99}

	\bibitem{MR618549}H.~W. Alt and L.~A. Caffarelli, Existence and regularity for a minimum problem with free boundary, {\it J. Reine Angew. Math.} {\bf 325} (1981), 105--144.

	\bibitem{MR752578}H.~W. Alt, L.~A. Caffarelli and A. Friedman, A free boundary problem for quasilinear elliptic equations, {\it Ann. Scuola Norm. Sup. Pisa Cl. Sci.} (4) {\bf 11} (1984), no.~1, 1--44.


	\bibitem{MR772122}H.~W. Alt, L.~A. Caffarelli and A. Friedman, Compressible flows of jets and cavities, {\it J. Differential Equations} {\bf 56} (1985), no.~1, 82--141.

	\bibitem{MR666110}C.~J. Amick, L.~E. Fraenkel and J.~F. Toland, On the Stokes conjecture for the wave of extreme form, {\it Acta Math.} {\bf 148} (1982), 193--214.

	\bibitem{MR96477}L. Bers, {\it Mathematical aspects of subsonic and transonic gas dynamics}, Surveys in Applied Mathematics, Vol. 3, Wiley, New York, 1958.

	\bibitem{MR990856}L.~A. Caffarelli, A Harnack inequality approach to the regularity of free boundaries. I. Lipschitz free boundaries are \(C^{1,\alpha}\), {\it Rev. Mat. Iberoamericana} {\bf 3} (1987), no.~2, 139--162.

	\bibitem{MR1029856}L.~A. Caffarelli, A Harnack inequality approach to the regularity of free boundaries. III.\ Existence theory, compactness, and dependence on \(X\), {\it Ann. Scuola Norm. Sup. Pisa Cl. Sci.} (4) {\bf 15} (1989), no.~4, 583--602.

	\bibitem{MR973745}L.~A. Caffarelli, A Harnack inequality approach to the regularity of free boundaries. II. Flat free boundaries are Lipschitz, {\it Comm. Pure Appl. Math.} {\bf 42} (1989), no.~1, 55--78.


	\bibitem{MR3437861}G. Chen, F. Huang and T. Wang, Subsonic-sonic limit of approximate solutions to multidimensional steady Euler equations, {\it Arch. Ration. Mech. Anal.} {\bf 219} (2016), no.~2, 719--740.

	\bibitem{MR2291790}G.~Chen, C.~M.~Dafermos, M.~Slemrod and D.~Wang, On two-dimensional sonic-subsonic flow, {\it Comm. Math. Phys.} {\bf 271} (2007), no.~3, 635--647.

	\bibitem{MR1911248}S. Chen, Z. Xin and H. Yin, Global shock waves for the supersonic flow past a perturbed cone, {\it Comm. Math. Phys.} {\bf 228} (2002), no.~1, 47--84.

	\bibitem{MR3139610}D. Coutand, J. Hole and S. Shkoller, Well-posedness of the free-boundary compressible 3-D Euler equations with surface tension and the zero surface tension limit, {\it SIAM J. Math. Anal.} {\bf 45} (2013), no.~6, 3690--3767.

	\bibitem{MR2608125}D. Coutand, H. Lindblad and S. Shkoller, A priori estimates for the free-boundary 3D compressible Euler equations in physical vacuum, {\it Comm. Math. Phys.} {\bf 296} (2010), no.~2, 559--587.

	\bibitem{MR2980528}D. Coutand and S. Shkoller, Well-posedness in smooth function spaces for the moving-boundary three-dimensional compressible Euler equations in physical vacuum, {\it Arch. Ration. Mech. Anal.} {\bf 206} (2012), no.~2, 515--616.


	\bibitem{MR2813524}D. De~Silva, Free boundary regularity for a problem with right hand side, {\it Interfaces Free Bound.} {\bf 13} (2011), no.~2, 223--238.

	\bibitem{MR4097326}M.~M. Disconzi and C. Luo, On the incompressible limit for the compressible free-boundary Euler equations with surface tension in the case of a liquid, {\it Arch. Ration. Mech. Anal.} {\bf 237} (2020), no.~2, 829--897.

	\bibitem{MR1134129}G.~C. Dong, {\it Nonlinear partial differential equations of second order}, translated from the Chinese by Kai Seng Chou,
	Translations of Mathematical Monographs, 95, Amer. Math. Soc., Providence, RI, 1991.

	\bibitem{DPY25} L.~Du, Y. Pu, and J. Yang, Singular profile of free boundary of incompressible inviscid fluid with external force, {\it SIAM J. Math. Anal.} {\bf 57} (2025), 4275--4313.

	\bibitem{MR4595616}L.~ Du, J. Huang and Y. Pu, The free boundary of steady axisymmetric inviscid flow with vorticity I: near the degenerate point, {\it Comm. Math. Phys.} {\bf 400} (2023), no.~3, 2137--2179.

	\bibitem{MR3178073}L.~Du, S. Weng and Z. Xin, Subsonic irrotational flows in a finitely long nozzle with variable end pressure, {\it Comm. Partial Differential Equations} {\bf 39} (2014), no.~4, 666--695.

	\bibitem{MR3196988}L.~Du, C. Xie and Z. Xin, Steady subsonic ideal flows through an infinitely long nozzle with large vorticity, {\it Comm. Math. Phys.} {\bf 328} (2014), no.~1, 327--354.

	\bibitem{MR2824469}L. Du, Z. Xin and W. Yan, Subsonic flows in a multi-dimensional nozzle, {\it Arch. Ration. Mech. Anal.} {\bf 201} (2011), no.~3, 965--1012.


	\bibitem{MR4808256}L.~Du and C. Yang, The free boundary of steady axisymmetric inviscid flow with vorticity II: near the non-degenerate points, {\it Comm. Math. Phys.} {\bf 405} (2024), no.~11, Paper No. 262, 58 pp.

	\bibitem{MR3409135}L.~C. Evans and R.~F. Gariepy, {\it Measure theory and fine properties of functions}, revised edition,
	Textbooks in Mathematics, 2015.

	\bibitem{MR1220787}L.~C. Evans and S. M\"{u}ller, Hardy spaces and the two-dimensional Euler equations with nonnegative vorticity, {\it J. Amer. Math. Soc.} {\bf 7} (1994), no.~1, 199--219.

	\bibitem{MR86556}R.~S. Finn and D. Gilbarg, Asymptotic behavior and uniqueness of plane subsonic flows, {\it Comm. Pure Appl. Math.} {\bf 10} (1957), 23--63.

	\bibitem{MR92912}R.~S. Finn and D. Gilbarg, Three-dimensional subsonic flows, and asymptotic estimates for elliptic partial differential equations, {\it Acta Math.} {\bf 98} (1957), 265--296.

	\bibitem{MR1368401}M. Giaquinta and S. Hildebrandt, {\it Calculus of variations. I}, Grundlehren der mathematischen Wissenschaften, 310, Springer, Berlin, 1996.

	\bibitem{MR4072680}D. Ginsberg, H. Lindblad and C. Luo, Local well-posedness for the motion of a compressible, self-gravitating liquid with free surface boundary, {\it Arch. Ration. Mech. Anal.} {\bf 236} (2020), no.~2, 603--733.

	\bibitem{MR2547977}J. Jang and N. Masmoudi, Well-posedness for compressible Euler equations with physical vacuum singularity, {\it Comm. Pure Appl. Math.} {\bf 62} (2009), no.~10, 1327--1385.

	\bibitem{MR3280249}J. Jang and N. Masmoudi, Well-posedness of compressible Euler equations in a physical vacuum, {\it Comm. Pure Appl. Math.} {\bf 68} (2015), no.~1, 61--111.

	\bibitem{MR1981993}H. Lindblad, Well-posedness for the linearized motion of a compressible liquid with free surface boundary, {\it Comm. Math. Phys.} {\bf 236} (2003), no.~2, 281--310.

	\bibitem{Lin05} H.~Lindblad, Well-posedness for the motion of a compressible liquid with free surface boundary, {\it Comm. Math. Phys.} {\bf 260} (2005), no.~2, 319--392.

	\bibitem{MR3887218}C. Luo, On the motion of a compressible gravity water wave with vorticity, {\it Ann. PDE} {\bf 4} (2018), no.~2, Paper No. 20, 71 pp.

	\bibitem{MR3812074}H. Lindblad and C. Luo, A priori estimates for the compressible Euler equations for a liquid with free surface boundary and the incompressible limit, {\it Comm. Pure Appl. Math.} {\bf 71} (2018), no.~7, 1273--1333.

	\bibitem{MR4439376}C. Luo and J. Zhang, Local well-posedness for the motion of a compressible gravity water wave with vorticity, {\it J. Differential Equations} {\bf 332} (2022), 333--403.

	\bibitem{MR3218831}T. Luo, Z. Xin and H. Zeng, Well-posedness for the motion of physical vacuum of the three-dimensional compressible Euler equations with or without self-gravitation, {\it Arch. Ration. Mech. Anal.} {\bf 213} (2014), no.~3, 763--831.

	\bibitem{MR1446239}J.~B. McLeod, The Stokes and Krasovskii conjectures for the wave of greatest height, {\it Stud. Appl. Math.} {\bf 98} (1997), no.~4, 311--333.


	\bibitem{Moise77}E.~E. Moise, {\it Geometric topology in dimensions 2 and 3}, Graduate Texts in Mathematics, 47, Springer, New York, 1977.

	\bibitem{McC25}S.~McCurdy, Non-existence of cusps for a free-boundary problem for water waves, {\it Interfaces Free Bound.} {\bf 27} (2025), no.~1, 1--11.

	\bibitem{MR752600}P.~I. Plotnikov, Justification of the Stokes conjecture in the theory of surface waves, {\it Dinamika Sploshn. Sredy} No. 57 (1982), 41--76.

	\bibitem{MR1883094}P.~I. Plotnikov, Proof of the Stokes conjecture in the theory of surface waves, {\it Stud. Appl. Math.} {\bf 108} (2002), no.~2, 217--244.

	\bibitem{MR4556790}O.~Savin and H. Yu, Regularity of the singular set in the fully nonlinear obstacle problem, {\it J. Eur. Math. Soc.} (JEMS) {\bf 25} (2023), no.~2, 571--610.

	\bibitem{MR2038344}P.~I. Plotnikov and J.~F. Toland, Convexity of Stokes waves of extreme form, {\it Arch. Ration. Mech. Anal.} {\bf 171} (2004), no.~3, 349--416.

	\bibitem{MR2858161} G.~G. Stokes, {\it Mathematical and physical papers. Volume 1}, reprint of the 1880 original,
	Cambridge Library Collection, Cambridge Univ. Press, Cambridge, 2009.

	\bibitem{MR513927}J.~F. Toland, On the existence of a wave of greatest height and Stokes's conjecture, {\it Proc. Roy. Soc. London Ser. A} {\bf 363} (1978), no.~1715, 469--485.

	\bibitem{MR2560044}Y.~L. Trakhinin, Local existence for the free boundary problem for nonrelativistic and relativistic compressible Euler equations with a vacuum boundary condition, {\it Comm. Pure Appl. Math.} {\bf 62} (2009), no.~11, 1551--1594.

	\bibitem{MR2810856}E. V\u{a}rv\u{a}ruc\u{a} and G.~S. Weiss, A geometric approach to generalized Stokes conjectures, {\it Acta Math.} {\bf 206} (2011), no.~2, 363--403.

	\bibitem{MR2995099}E. V\u{a}rv\u{a}ruc\u{a} and G.~S. Weiss, The Stokes conjecture for waves with vorticity, {\it Ann. Inst. H. Poincar\'{e} C Anal. Non Lin\'{e}aire} {\bf 29} (2012), no.~6, 861--885.

	\bibitem{MR3225630}E. V\u{a}rv\u{a}ruc\u{a} and G.~S. Weiss, Singularities of steady axisymmetric free surface flows with gravity, {\it Comm. Pure Appl. Math.} {\bf 67} (2014), no.~8, 1263--1306.

	\bibitem{MR3048597}C. Wang and Z. Xin, On a degenerate free boundary problem and continuous subsonic-sonic flows in a convergent nozzle, {\it Arch. Ration. Mech. Anal.} {\bf 208} (2013), no.~3, 911--975.

	\bibitem{MR1620644}G.~S. Weiss, Partial regularity for weak solutions of an elliptic free boundary problem, {\it Comm. Partial Differential Equations} {\bf 23} (1998), no.~3-4, 439--455.

	\bibitem{MR1759450}G.~S. Weiss, Partial regularity for a minimum problem with free boundary, {\it J. Geom. Anal.} {\bf 9} (1999), no.~2, 317--326.

	\bibitem{MR4238496}G.~S. Weiss, Bernoulli type free boundary problems and water waves, in {\it Geometric measure theory and free boundary problems}, 89--136, Lecture Notes in Math., 2284, Springer, Cham, 2021.


	\bibitem{MR2507638}S. Wu, Almost global wellposedness of the 2-D full water wave problem, {\it Invent. Math.} {\bf 177} (2009), no.~1, 45--135.

	\bibitem{MR2782254}S. Wu, Global wellposedness of the 3-D full water wave problem, {\it Invent. Math.} {\bf 184} (2011), no.~1, 125--220.

	\bibitem{MR2375709}C. Xie and Z. Xin, Global subsonic and subsonic-sonic flows through infinitely long nozzles, {\it Indiana Univ. Math. J.} {\bf 56} (2007), no.~6, 2991--3023.

	\bibitem{MR2607929}C. Xie and Z. Xin, Existence of global steady subsonic Euler flows through infinitely long nozzles, {\it SIAM J. Math. Anal.} {\bf 42} (2010), no.~2, 751--784.

	\bibitem{MR2644144}C. Xie and Z. Xin, Global subsonic and subsonic-sonic flows through infinitely long axially symmetric nozzles, {\it J. Differential Equations} {\bf 248} (2010), no.~11, 2657--2683.

	\bibitem{MR2533922}Z. Xin, W. Yan and H. Yin, Transonic shock problem for the Euler system in a nozzle, {\it Arch. Ration. Mech. Anal.} {\bf 194} (2009), no.~1, 1--47.

	\enlargethispage{\baselineskip}
	\bibitem{MR2216450}Z. Xin and H. Yin, Global multidimensional shock wave for the steady supersonic flow past a three-dimensional curved cone, {\it Anal. Appl. (Singap.)} {\bf 4} (2006), no.~2, 101--132.

\end{thebibliography}
\end{document}